\documentclass[11pt]{amsart}
\usepackage{graphicx}
\usepackage[utf8]{inputenc}
\usepackage[T1]{fontenc}
\usepackage{textcomp}
\usepackage[english]{babel}
\usepackage{url}
\usepackage{amsmath,amssymb,amsthm}
\usepackage{mathrsfs}
\usepackage{csquotes}

\usepackage{import}
\usepackage[dvipsnames]{xcolor}
\usepackage[colorlinks=true,linktoc=page]{hyperref}
\hypersetup{urlcolor=NavyBlue, citecolor=red, linkcolor=RubineRed}

\usepackage[all,cmtip]{xy}
\usepackage{manfnt}
\usepackage[indentafter]{titlesec}
\titleformat{name=\section}{}{\thetitle.}{0.8em}{\centering \bfseries\scshape}
\titleformat{name=\subsection}[runin]{}{\thetitle.}{0.5em}{\bfseries}[.]
\titleformat{name=\subsubsection}[runin]{}{\thetitle.}{0.5em}{\bfseries}[.]
\titleformat{name=\paragraph,numberless}[runin]{}{}{0em}{\bfseries}[.]
\titlespacing{\paragraph}{0em}{0em}{0.5em}
\titleformat{name=\subparagraph,numberless}[runin]{}{}{0em}{\bfseries}[.]
\titlespacing{\subparagraph}{0em}{0em}{0.5em}
\usepackage[backend=biber, style=alphabetic, natbib=true,hyperref=true]{biblatex}
\theoremstyle{plain}
\newtheorem{Theo}{Theorem}[section] %compteur commençant par le numéro de la section (on pourrait aussi faire commencer par le numéro de la sous-section - remplacer "section" par "subsection")
\newtheorem{Pro}[Theo]{Proposition}        %même compteur que pour les théorèmes      
\newtheorem{Lem}[Theo]{Lemma}            %etc...
\newtheorem{cor}[Theo]{Corollary}
\newtheorem{prop}[Theo]{Properties}

\theoremstyle{definition}
\newtheorem{Defi}[Theo]{Definition}
\newtheorem{exam}[Theo]{Example}

\newtheorem{nota}[Theo]{Notation}

\newtheorem{hyp}[Theo]{Convention}

\theoremstyle{remark}
\newtheorem{rem}[Theo]{Remark}

\def\ogg~{{\rm \og}}   % guillemets ouvrants
\def\nn{\noindent}

\def\q{\nn}
\def\qq{\nn\quad}
\def\qqq{\nn\quad\quad}

\def\emptyset{\varnothing}

\def\NN{{\mathbb N}}    %naturels
\def\ZZ{{\mathbb Z}}     %entiers relatifs
\def\RR{{\mathbb R}}    %réels
\def\AA{{\mathbb A}}    % espace affine 
\def\PP{{\mathbb P}}

   \def\cM{{\mathcal M}}      \def\cT{{\mathcal T}} \def\cC{{\mathcal C}}   \def\cO{{\mathcal O}}         \def\cK{{\mathcal K}}     \def\cL{{\mathcal L}} \def\cR{{\mathcal R}}    

 \def\mfp{{\mathfrak p}}                           

\newcommand\ct{\operatorname{cotan}}         %cotangente
\newcommand{\dx}{\partial_x}                    
\newcommand{\dy}{\partial_y}

\newcommand{\car}{\operatorname{Card}}
\newcommand{\Deg}{\operatorname{deg}}
\newcommand{\disf}[2]{D^+(#1,#2)}
\newcommand{\diso}[2]{D^-(#1,#2)}
\newcommand{\fdisf}[2]{\cO(D^+(#1,#2))}
\newcommand{\fdiso}[2]{\cO(D^-(#1,#2))}
\newcommand{\h}[1]{\mathscr{H} (#1)}
\newcommand{\A}[1]{\AA^{1,\mathrm{an}}_{#1}}

\newcommand{\Po}[1]{\PP^{1,\mathrm{an}}_{#1}}
\newcommand{\D}[1]{\frac{\mathrm{d}}{\mathrm{d#1}}}

\newcommand{\LL}[2]{\cL_{#1}(#2)}

\newcommand{\Lk}[1]{\LL{k}{#1}}
\newcommand{\nsp}[1]{\rVert #1\rVert_{\mathrm{Sp}}}
\newcommand{\Nsp}[2]{\rVert #2\rVert_{\mathrm{Sp},#1}}
\newcommand{\piro}[2]{\pi_{#1/#2}}
\newcommand{\pik}[1]{\piro{#1}{k}}

\newcommand{\Log}[1]{\operatorname{Log}_{#1}}

\newcommand{\nor}[1]{\rVert #1\rVert}

\newcommand{\discf}[3]{D^+_{#1} (#2,#3)}
\newcommand{\disco}[3]{D^-_{#1} (#2,#3)}
\newcommand{\fdiscf}[3]{\cO(D^+_{#1} (#2,#3))}

\newcommand{\fcouf}[3]{\cO(C^+ (#1,#2,#3))}

\newcommand{\Couf}[4]{C^+_{#1} (#2,#3,#4)}

\newcommand{\Uni}[2]{\sigma_{#2/#1}}
\newcommand{\uni}[1]{\Uni{k}{#1}}

\newcommand{\Fonction}[4]{\begin{array}[c]{rcl} 
                            #1&\longrightarrow&#2\\ #3&\mapsto&
                                                                #4\\ \end{array}}

\newcommand{\fra}[1]{\frac{1}{#1}}
\newcommand{\Dmod}[2]{#1-\text{\bf Mod}(#2)}
\newcommand{\hh}[1]{\mathscr{H}(#1)}
\newcommand{\Rd}[1]{\cR^{#1,\mr{Sp}}}

\renewcommand\phi{\varphi}
\renewcommand\epsilon{\varepsilon}
\def\ct{\hat{\otimes}}

\def \ot{\otimes}
\def \<{\langle}
\def \>{\rangle}

\def\|{\rVert}
\def\*{\blacklozenge}

\def\Rm{\Rd{M}}

\def\b{\mathbf}

\def\mr{\mathrm}

\def\dT{\frac{\mr{d}}{\mr{dT}}}
\def\d{\frac{\mr{d}}{\mr{dS}}}
\def\dx{\frac{\mr{d}}{\mr{dS(x)}}}
\def\dz{\D{Z}}
\def\dy{\D{S(y)}}

\def \R+{\RR_{+}}

\def\Ak{\A{k}}

\def \Pk{\Po{k}}

\def \-1{^{-1}}

\def \Hx{\mathscr{H}(x)}
\def \Hy{\mathscr{H}(y)}
\def \Hz{\mathscr{H}(z)}

\def \rk{\tilde{k}}
\def \crk{\mathrm{char}(\rk)}

\def\loga{\Log{a}}

\def\forp{\operatorname{\b{Frob}}_p}

\def\Bigcup{\bigcup\limits}

\def\Lim{\lim\limits}

\def\Inf{\inf\limits}
\def\Min{\min\limits}
\def\Max{\max\limits}
\def\(({(\!(}
\def\)){)\!)}
\def\DD{\mathscr{D}}

\def\kac{\widehat{k^{alg}}}
\def\wac{\widehat{\Omega^{alg}}}  
\def\EE{\mathscr{E}}

\def\F{\mathscr{F}}

\numberwithin{equation}{section}

\hypersetup{
pdfauthor = {Tinhinane Amina AZZOUZ},
pdfsubject = {p-adic differential equations},
pdftitle = {Spectrum of p-adic linear differential equations
  I},
pdfkeywords = {p-adic differential equations, Spectral theory, Berkovich
  spaces, Radius of convergence},
}

\begin{document}
\DeclareRobustCommand{\subtitle}[1]{\\#1}
\title{Spectrum of p-adic linear differential equations
  I\subtitle{\tiny{The shape of the spectrum}}}

% Information for first author
\author{Tinhinane A. AZZOUZ}
\address{Tsinghua University, Yau Mathematical Sciences Center,
  Yanqi Lake Beijing Institute of Mathematical Sciences and applications,
  Beijing, China}

\email{azzouzta@bimsa.cn}
\thanks{The author was supported by NSFC Grant \#12250410239.}
% \dedicatory{This paper is dedicated to our advisors.}
\begin{abstract}
  This paper extends our previous works \cite{Cons, Azz21} on
  determining the spectrum, in the Berkovich sense, of ultrametric
  linear differential equations. Our previous works focused on
  equations with constant coefficients or over a field of formal power
  series. In this paper, we investigate the spectrum of $p$-adic
  differential equations at a generic point on a quasi-smooth
  curve. This analysis allows us to establish a significant connection
  between the spectrum and the spectral radii of convergence of a differential
  equation when considering the affine line. Furthermore, the spectrum
  offers a more detailed decomposition compared to Robba's
  decomposition based on spectral radii \cite{Rob75a}.
\end{abstract}
\maketitle
\section{Introduction}
In the ultrametric setting, linear differential equations present
phenomena that do not appear over the complex field.  Indeed, the
solutions of such equations may fail to converge everywhere, even
without the presence of poles. This leads to a non-trivial notion of
radius of convergence, and its knowledge allows to obtain several
interesting information about the equation. Notably, it controls the
{\it finite dimensionality} of the de Rham cohomology. In practice, the radius of convergence is really hard to compute and it
represents one of the most complicate features in the theory of $p$-adic
differential equations \cite{and,np2,np3,np4}. The radius of
convergence can be expressed as the spectral norm of a certain
operator. A natural notion refining it is the entire spectrum
of this operator, in the sense of Berkovich.

In our previous works \cite{Cons, Azz21}, we determine the spectrum of differential equations over a field of power
series, and in the ultrametric case, differential equation with constant coefficients. The general
case in the $p$-adic situation is the subject of this paper, and it is
much more difficult, since it requires a fine improvement of
some of the fundamental theorems in the theory of $p$-adic
differential equations.

Let $(k,|.|)$ be an algebraically closed complete
ultrametric field of characteristic zero, and let $\Ak$ be the
Berkovich affine line. We fix $T$ to be a coordinate function on
$\Ak$. For any positive real number $r$ and $c\in k$, let $x_{c,r}$ be
the point of $\Ak$ associated to the multiplicative norm $k[T]\to
\R+$, $\sum_ia_i(T-c)^ i\mapsto \max_i|a_i|r^ i$. For a point $x\in
\Ak$, we denote by $\Hx$ the associated complete residue field, and by
$r_k(x)$ its radius (cf. \eqref{eq:60}).

This work is in line with the work on differential modules defined over a
generic point of a quasi-smooth\footnote{{\it Quasi-smooth} means that its sheaf
  of differential forms $\Omega_X$ is locally free of rank 1, see
  \cite[Definition~3.1.11]{Duc}.} curve, initiated  by Dwork
in \cite{Dwo73}. We will restrict ourselves to 
points $x$ of $\Ak\setminus k$, and consider differential module $(M,\nabla)$ of
finite rank over $(\Hx, d)$), where $d$
is a ``reasonable'' $k$-linear bounded derivation defined over $\Hx$. We will explain later why we do
not lose generality by such a restriction. For a
point $x\in\Ak\setminus k$ and a differential module $(M,\nabla)$ over
$(\Hx,d)$, we set $\Rd{(M,\nabla)}_1(x)\leq
\cdots\leq\Rd{(M,\nabla)}_n(x)$ to be the spectral radii of
convergence, as originally given in \cite{dwork} for
$\Rd{(M,\nabla)}_1(x)$, and generalized for $i\geq 2$ in \cite{Ked}.

In this setting, the spectrum of $(M,\nabla)$ is the spectrum in the sense of
Berkovich (see \cite[Chapter 7]{Ber})
$\Sigma_{\nabla,k}(\Lk{M})$ of $\nabla$ as an element of $\Lk{M}$, the
$k$-Banach algebra of $k$-linear bounded endomorphisms of $M$. This
spectrum is a subset of $\Ak$ instead of just $k$, and also enjoys the
same nice properties as in the complex case, like non-emptiness
and compactness. To avoid any confusion, we
fix another coordinate function $S$ on $\Ak$ for
differential modules.

Now let us comeback to the relation between
the spectrum and the radius of convergence, consider a point $x\in
\Ak\setminus k$ and a differential module $(M,\nabla)$ over
$(\Hx,\d)$. On the one hand we have $\nsp{\nabla}=\lim\limits_{n\to\infty}\nor{\nabla^n}^{\fra{n}}=\frac{\omega}{\Rd{(M,\nabla)}_1(x)}$,
where $\omega:=\lim|n!|^{\fra{n}}$ (cf. \cite[p. 676]{dwork} or \cite[Definition~9.4.4]{Ked}). On the
other hand, the spectral norm $\nsp{\nabla}$ is
also equal to the radius of the smallest closed disk centered at zero and
containing $\Sigma_{\nabla,k}(\Lk{M})$ (cf. \cite[Theorem~7.1.2]{Ber}). In our first computation of
the spectrum in \cite{Cons}, we prove that the spectrum of a linear differential
equation with constant coefficients $(M,\nabla)$ is a finite union of
closed disks $\cup_i D_i$. Furthermore, we observe that for each
$\Rd{(M,\nabla)}_i(x)$ there exists $D_j$ such that the smallest closed disk centered at zero
containing $D_j$ has radius equal to
$\frac{\omega}{\Rd{(M,\nabla)}_i(x)}$, and conversely for each
$D_i$. We may ask ourselves, how far can this relations between the
spectrum and the radii of convergence be generalized? In our work \cite{Azz21}, we
observe another phenomenon, the spectrum is intrinsic to the choice
of the derivation. This means that given a differential module
$(M,\nabla)$ over $(\Hx,d)$ for some $x\in \Ak\setminus k$ and $g\in
\Hx\setminus\{0\}$, then even if the kernel of $(M,\nabla)$ and
$(M,g\nabla)$  define the same differential system, the respective
spectra of these differential modules may differ. For exemple, suppose
that $k$ is trivially valued and let $x_{0,r}\in \Ak\setminus k$ with
$r<1$, then the spectrum of $(\Hx,S\d)$ is equal to $\ZZ\cup\{x_{0,1}\}$
(cf. \cite[Proposition~3.7]{Azz21}, this computation is generalized in
Section \ref{sec:bf-spectr-deriv}), however the spectrum of $(\Hx,\d)$
is equal to the closed disk $\disf{0}{\fra{r}}$ (cf. \cite{Cons}). So, the
following questions arise:
\begin{itemize}
\item What is the most suitable choice of derivations?
\item If we choose a derivation different from $\d$, does the relation
  between the spectrum and the radii of convergence remain?
\end{itemize}
The derivation $\d$ appears to be the most natural, however we quickly
realize that the determination of the spectrum becomes very
hard in this case. Indeed, the usual technique used for the computation of the
radii, like the ramification of the indeterminate or the push-forward by the Frobenius map, cannot be
used for the determination of the spectrum. That is why we privilege
the use of the derivation $S\d$ in \cite{Azz21} to determine the
spectrum of a differential module over a field of formal power
series. The interesting part is that by choosing $S\d$, not only the
determination of the spectrum becomes
affordable, but in this situation we can also recover all the data of radii of
convergence and others information in the spectrum. For these reasons, in
this paper, for a point $x_{c,r}\in \Ak\setminus k$, we claim that
$(S-c)\d$ is the most convenient choice.

In this paper we focus mainly on the case where $\crk=p>0$, and
provide an algorithm to determine the spectrum of any differential
module $(M,\nabla)$ over $(\Hx,(S-c)\d)$, with $x=x_{c,r}$ in
$\Ak$. Moreover, we establish a link between the spectrum
and all radii of convergence, which is the purpose of the main result of the paper. For $(\Omega,|.|)$ an extension of $(k,|.|)$, we define the following map
  \begin{equation}
    \Fonction{\delta_{(\Omega,|.|)}:\Omega}{\R+}{z}{\inf_{n\in \ZZ}|z-n|.}
  \end{equation}
  We may denote it only by $\delta$ instead of
  $\delta_{(\Omega,|.|)}$. We denote by $\pik{\Omega}:\A{\Omega}\to
  \Ak$ the canonical projection. The main result of the paper is the following.

 \begin{Theo}[Theorem~\ref{sec:bf-spectr-comp-3}]\label{sec:introduction-1}
  Assume that $\crk=p>0$ and $x:=x_{0,r}\in \Ak\setminus k$. Let
  $(M,\nabla)$ be a differential module $(\Hx, S\d)$. Let
  $\forp:\Ak\to \Ak$, $S\mapsto S^p$. We denote by $\cR_a(x)$ the spectral
  radii of $(\Hx,S\d-a)$ and $\forp^l(x)$ by $x^{p^l}$.
\begin{itemize}

\item There exist $z_1,\cdots,z_{\nu}\in \Ak\setminus k$ and
  $a_1,\cdots, a_\mu\in k$, such that
  \[\Sigma_{\nabla,k}(\Lk{M})=\{z_1,\cdots,z_\nu, a_1,\cdots,
    a_\mu\}+\ZZ_p,\]
 where $z_i$ has the same type as $x$, and $(\nu,\mu)$ is not equal to
 $(0,0)$.
\item We can choose $z_i$ and $a_j$ such that the set $\{z_1,\cdots,z_\nu, a_1,\cdots,
    a_\mu\}$ has minimal cardinality. Indeed it is enough to keep only
    $z_i$ and $a_j$, for which we have $\{z_i\}+\ZZ_p\cap
    \{z_{i'}\}+\ZZ_p=\emptyset$ and $\{a_j\}+\ZZ_p\cap \{a_{j'}\}+\ZZ_p=\emptyset$
    for $i\ne i'$ and $j\ne j'$.
  
  \item We choose $\{z_1,\cdots,z_\nu, a_1,\cdots,
    a_\mu\}$ to be minimal. Then we have a unique (up to an isomorphism) decomposition
    \begin{equation}\label{eq:78}
      (M,\nabla)=\bigoplus_{i=1}^{\nu}(M_{z_i},\nabla_{z_i})\oplus \bigoplus_{j=1}^{\mu}(M_{a_j},\nabla_{a_j}),
    \end{equation}
    such that, $\Sigma_{\nabla_{z_i},k}(\Lk{M_{z_i}})=\{z_i\}+\ZZ_p$
    and $\Sigma_{\nabla_{a_j},k}(\Lk{M_{a_j}})=\{a_j\}+\ZZ_p$.
    \item Let $c_i\in k$ and $r_i>0$ such that $z_i=x_{c_i,r_i}$. If
    $|p|^{l}\leq r_i< |p|^{l-1}$, with $l\in \NN\setminus\{0\}$, then $\car(\{z_i\}+\ZZ_p)=p^l$ and $\{z_i\}+\ZZ_p=\{x_{c_i,r_i},
    x_{c_i+1,r_i},\cdots, x_{c_i+p^l-1,r_i}\}$. If $r_i\geq1$ then we
    have $\car(\{z_i\}+\ZZ_p)=1$ and
    $\{z_i\}+\ZZ_p=\{x_{c_i,r_i}\}$.
  \item If $r_i>1$, let $P_{z_i}(S\d)$ be a differential polynomial
    associated to $(M_{z_i},\nabla_{z_i})$. Then the image by
    $\pi_{\widehat{\Hx^{alg}}/k}$ of all roots of $P_{z_i}(T)$ (the
    commutative polynomial associated to $P_{z_i}(S\d)$) is equal to $z_i$. 
  \item If $|p|^l<r_i\leq|p|^{l-1}$, let $P_{z_i}(p^lS\d)$ be a differential
    polynomial associated to $(\forp^l)_*(M,\nabla)$ (as a differential
    module over $(\h{x^{p^l}},p^lS\d)$). Then the image by
    $\pi_{\widehat{\h{x^{p^l}}^{alg}}/k}$ of all roots of $P_{z_i}(T)$
    (the commutative polynomial associated to $P_{z_i}(p^lS\d)$) is equal to $\{x_{c_i,r_i},
    x_{c_i+1,r_i},\cdots, x_{c_i+p^l-1,r_i}\}$. In the special case
    where $r_i=|p|^{l-1}$ we have $\{x_{c_i,r_i},
    x_{c_i+1,r_i},\cdots, x_{c_i+p^l-1,r_i}\}=\{x_{c_i,r_i},
    x_{c_i+1,r_i},\cdots, x_{c_i+p^{l-1}-1,r_i}\}$.

  \item If $r_i\geq 1$. For all $a\in k$, the differential module $(M_{z_i},\nabla_{z_i}-a)$
    is pure (all its spectral radii are equal). For $a\in \disf{c_i}{r_i}\cap k$ we have
    \begin{equation}
      \Rd{(M_{z_i},\nabla_{z_i}-a)}_1(x)=\frac{\omega}{r_i}r,
    \end{equation}
    and for all $a\in k\setminus \disf{c_i}{r_i}$
    \begin{equation}
      \Rd{(M_{z_i},\nabla_{z_i}-a)}_1(x)=\frac{\omega}{|a-c_i|}r.
    \end{equation}
 \item If $|p|^l\leq r_i<|p|^{l-1}$. For all $a\in k$, the differential module $(M_{z_i},\nabla_{z_i}-a)$
    is pure. We have for all $a\in \bigcup_{j=0}^{p^l-1} \disf{c_i+j}{r_i}\cap k$
    \begin{equation}
      \Rd{(M_{z_i},\nabla_{z_i}-a)}_1(x)=\left(\frac{|p|^l\omega}{r_i}\right)^{\fra{p^l}}r,
    \end{equation}
    and for all $a\in k\setminus \bigcup_{j=0}^{p^l-1} \disf{c_i+j}{r_i}$
    \begin{equation}
      \Rd{(M_{z_i},\nabla_{z_i}-a)}_1(x)=\cR_{a-c_i}(x).
    \end{equation}
  \item For all $a\in k$, the differential module
    $(M_{a_i},\nabla_{a_i}-a)$ is pure. More precisely, for all $a\in
    \{a_i\}+\ZZ_p$, $(M_{a_i},\nabla_{a_i}-a)$ is solvable (all its radii
    are equal to $r$), and for all $a\in
    k\setminus \{a_i\}+\ZZ_p$, we have $\Rd{(M_{a_i},\nabla_{a_i}-a)}_1(x)=\cR_{a-a_i}(x)$.
\end{itemize}
\end{Theo}

Let $(M,\nabla)$ be as in
Proposition~\ref{sec:introduction-1} and
$\Sigma_{\nabla,k}=\{\omega_1,\cdots,\omega_\mu\}+\ZZ_p$ be the
spectrum of $\nabla$, where $\{\omega_1,\cdots,\omega_\mu\}$ has minimal
cardinality and $\omega_i\in \Ak$. We emphasize that if
$(M,\nabla)=\bigoplus_{i=1}^\mu(M_{\omega_i},\nabla_{\omega_i})$ is as in
\eqref{eq:78}, then
\begin{equation}
  \label{eq:77}
  \underbrace{\Rd{(M_{\omega_1},\nabla_{\omega_1})}_1(x),\cdots,\Rd{(M_{\omega_1},\nabla_{\omega_1})}}_{\dim
  M_{\omega_1} \text{ times}},\cdots, \underbrace{\Rd{(M_{\omega_\mu},\nabla_{\omega_\mu})}_1(x),\cdots,\Rd{(M_{\omega_\mu},\nabla_{\omega_\mu})}}_{\dim
  M_{\omega_\mu} \text{ times}},
\end{equation} 
after a suitable permutation corresponding exactly to $\Rd{(M,\nabla)}_1(x),\cdots,
\Rd{(M,\nabla)}_n(x)$. %

This theorem gives a kind of geometrical incarnation of Robba’s
 decomposition by spectral radii \cite{Rob75a}, that stating a differential module $(M,\nabla)$ over $(\Hx, d)$
 decomposes into a direct sum of pure differential modules, and each
 component has all its spectal radii equal to one of the value of the spectral radii of convergence. However, the
 decomposition provided by the spectrum is finer than the one provided by radii. Moreover, when we  vary $a\in k$,  the radii of $(M,\nabla-a)$
 are well controlled by the points of the spectrum of $\nabla$. Let us
 give an example where we see concretely how the decomposition
 provided by the spectrum is that finer.

 \begin{exam}
   Let $(M,\nabla)$ be a differential module over $(\h{x_{0,1}},S\d)$,
   associated to the differential polynomial
   $P(d)=d^2-[(a+b)S+b]d+bS(aS+b)$ with $a,b\in k$ and $1<|a|<|b|$. By Young's theorem,
   $(M,\nabla)$ is a pure differential module with radius equal to
   $\frac{\omega}{|b|}$. However, by Theorem~\ref{sec:introduction}, we have
   $\Sigma_\nabla=\{x_{0,|b|},x_{b,|a|}\}$. This means that we can
   decompose $(M,\nabla)$ with respect to the spectrum.
 \end{exam}
 
 More generally we have the following results.
\begin{Theo}[Theorem~\ref{sec:spectr-diff-equat-1}]
  Assume that $\crk=p>0$. Let $\cC$ be a quasi-smooth curve and $x\in
  \cC$ of type (2) or (3). Let $(M,\nabla)$ be a differential module over
  $(\Hx,d)$, where $d=\psi^*(S\d)$, $\psi$ is a finite étale morphism from a
  neighbourhood of $x$ to $\Pk$, with $\psi(x)=x_{0,r}$. Then there exist $z_1,\cdots,
  z_\mu\in\Ak$, with $(\{z_i\}+\ZZ_p)\cap
  (\{z_j\}+\ZZ_p)=\emptyset$ for $i\ne j$, such that: 
  \begin{equation}
     \Sigma_{\nabla,k}(\Lk{M})= \{z_1,\cdots,
     z_\mu\}+\ZZ_p.
   \end{equation}
 \end{Theo}

 Note that for any point $x$ in a quasi-smooth curve $\cC$, there exists a finite étale morphism $\psi: Y\to W$, where $Y$ is an affinoid neighbourhood of $x$
and $W$ an affinoid domain of the projective Berkovich line $\Pk$
(cf. \cite[Theorem~4.5.4]{Duc},\cite[Theorem~3.12]{np2}). If $x$ has a neighbourhood isomorphic  to an open disk or annulus,
then the link between the radii and the spectrum is similar to the
case of the affine line. More generally,  for the case of a
quasi-smooth curve, to establish the link
between the spectral radii of convergence and the spectrum we need extra material, notably the
continuity of the spectrum on branches out of $x$. This is the object of our work \cite{AzzC}.   

The proof of the main result requires to devalop some key techniques. The most important one
 is a kind of spectral version of Young's theorem, that states the following:
\begin{Theo}[Theorem~\ref{sec:spectr-vers-youngs-9}]\label{sec:introduction}
   Let $(\Omega,|.|)$ be a complete extension of $(k,|.|)$ and
   $d:\Omega\to \Omega$ be a $k$-linear bounded
    derivation. We set  $\pik{\wac}:\A{\wac}\to \Ak$ to be the
    canonical projection. Let $(M,\nabla)$ be a differential module over
    $(\Omega,d)$, $P(T)=\sum_{i=0}^{n-1}a_i T^i+T^n\in \Omega[T]$ and $\{z_1,\cdots,
    z_n\}\subset \Omega^{alg}$ be the multiset of the roots of
    $P(T)$. Suppose that in some basis the associated matrix of
    $(M,\nabla)$ is
\begin{equation}G=\left(
     \raisebox{0.5\depth}{%
       \xymatrixcolsep{1ex}%
       \xymatrixrowsep{1ex}%
       \xymatrix{0\ar@{.}[rrr]& & & 0&-a_0\ar@{.}[dddd] \\
        1\ar@{.}[rrrddd]& 0\ar@{.}[rr]\ar@{.}[rrdd]& &0\ar@{.}[dd]& \\
        0\ar@{.}[dd]\ar@{.}[rrdd]& &  &  &\\
        & &  &0& \\
        0\ar@{.}[rr]& &0&1& -a_{n-1}\\ }%
        }
       \right)
       .\end{equation}
    If $ \min_i
    r_k(\pik{\wac}(z_i))>\nor{d}$ ($\nor{d}$ is the operator norm of $d$), then
    \begin{equation}\Sigma_{\nabla,k}(\Lk{M})=
      \pik{\wac}(\{z_1,\cdots,z_n\}).\end{equation}
  \end{Theo}%
The other key technique is to provide the relation between the
 spectrum of a differential module and the spectrum of its
 push-forward by an étale morphisms (mainly by the Frobenius map).

The paper is organized as follows. Section~\ref{sec:preliminaries}
recalls all the necessary material from Berkovich geometry. In
particular, we recall some important tools developed in our previous
work \cite{Cons}.

Section~\ref{sec:spectr-diff-module-4} is divided
into two parts. In the first one, we determine the spectrum of
$S\d:\Hx\to\Hx$, where $x\in \Ak\setminus k$. We treat separately the
case where $\crk=p>0$ from the case where $\crk=0$, and the case where
$x\in(0,\infty)$ form the case where $x\not\in (0,\infty)\cup
k$. Indeed, for each situation we use different methods. Assume $\crk=p>0$,
if $x\in (0,\infty)$, we use the push-forward  by Frobenius map, and
if $x\not\in (0,\infty)\cup k$, we use the push-forward by logarithm
map. In the case where $\crk=0$, if $x\in (0,\infty)$, we compute the
spectrum without extra method, and if $x\not\in (0,\infty)\cup k$ we
also use the push-forward by logarithm map. In the second part,  the
spectrum of $S\d$ allows us to
deduce easily the spectrum of any differential module with regular
singularities. On the other hand, in the $p$-adic case, we prove that the
variation of the spectrum satisfies a continuity property. However, if
$\crk=0$, the variation is surprisingly not continuous
at all.

In Section~\ref{sec:spectr-vers-youngs-10}, we provide the spectral version of Young's
theorem. Since in $\crk=0$, the radii
are either solvable or small, this spectral version permits to determine the spectrum of any differential module
with non solvable radii. In particular, we can recover the result
of our paper \cite{Azz21} easily and without using Turrittin's
theorem.

Section~\ref{sec:spectr-diff-module-5} is devoted to prove the main result of the
paper, i.e. the determination of the spectrum of any differential module
$(M,\nabla)$ over $(\h{x_{0,r}},S\d)$, when
$\crk=p>0$. For that we proceed as follows: we start by establishing, when
the radii are small, the link
between the spectrum and the spectral radii of convergence, especially when we
choose $ p^l(S-c)\d$ as derivation. Then we use the Frobenius map
(pull-back and push-forward) to compute the spectrum and establish the
link with the spectral radii of convergence. In the last part of the section,
we explain how we can deduce from the main result the shape of the spectrum of a
differential module $(M,\nabla)$ over $(\Hx,d)$, where $x$ is a point of
a quasi-smooth curve of type (2) or (3), and $d$ is well chosen
$k$-linear bounded derivation.

{\bf Acknowledgments} The author expresses her gratitude to
A. Pulita and J. Poineau for useful comments and suggestions. She also thanks F. Baldassarri,
F. Beukers, A. Ducros, F. Truc, H. Diao, E. Lecouturier and S. Palcoux for useful occasional discussions and suggestions.

\tableofcontents

\section{Preliminaries}\label{sec:preliminaries}
\subsection{Definitions and notations}
All rings are supposed to have a unit element. We will denote by $\RR$ the field of real numbers, by $\ZZ$ the ring
of integers and by $\NN$ the set of nonnegative integers. We set
$\R+$ for $\{r\in\RR;\; r\geq 0\}$ and $\R+^*$ for $\R+\setminus\{0\}$.

In all the paper, we fix $(k,|.|)$ to be an ultrametric complete field
of characteristic $0$. Let $|k|$ be $\{|a|;\; a\in k\}$. Let $E(k)$ be the category whose objects are
$(\Omega,|.|_\Omega)$, where $\Omega$ is a field extension of $k$,
complete with respect to the valuation $|.|_\Omega$, and whose
isomorphisms are isometric rings morphisms. For
$(\Omega,|.|_\Omega)\in E(k)$, let $\Omega^{alg}$  be an algebraic closure
of $\Omega$, the absolute value extends uniquely to an absolute value
defined on $\Omega^{alg}$. We denote by $\wac$ the completion of
$\Omega^{alg}$ with respect to this absolute value.

\subsubsection{Differential modules}Let $\Omega\in E(k)$, in all the
paper any $k$-linear derivation $d:\Omega\to
\Omega$ is supposed to be bounded. Recall that a differential module $(M,\nabla)$ over $(\Omega,d)$ is
an $\Omega$-vector space with $\dim_\Omega(M)<+\infty$ and equipped with
a $k$-linear map $\nabla: M\to M$, called connection of $M$,
satisfying $\nabla(fm)=df.m+f.\nabla(m)$ for all $f\in \Omega$ and
$m\in M$.

\begin{nota}
  Let $(\Omega,d)$ be a differential field. We denote by
  $\Dmod{d}{\Omega}$  the category of differential modules over
  $(\Omega,d)$ whose arrows are morphisms of differential modules.
\end{nota}
\begin{nota}
  Let $(\Omega,d)$ be a differential field. Let
  $\DD_{\Omega,d}:=\bigoplus_{i\in\NN}\Omega\cdot d^i$ be the ring of
  differential polynomials on $d$ with coefficients in $\Omega$, where
  the multiplication is non-commutative and defined as follows:
  $d\cdot f=d(f)+f\cdot d$ for all $f\in \Omega$. Let
  $P(d)=g_0+\cdots+ g_{n-1}d^{n-1}+d^n$ be a monic differential
  polynomial. The quotient $\DD_{\Omega,d}/\DD_{\Omega,d}\cdot P(d)$ is an
  $\Omega$-vector space of dimension $n$, the
  multiplication by $d$ induces a structure of differential module over
  $(\Omega,d)$. If there is no confusion about the derivation we will
  simply denote  by $\DD_{\Omega}$. We denote by $\text{\rm
    Hom}_{\DD_{\Omega,d}}(M,N)$ the set of arrows between the two
  objects $M$ and $N$ of $\Dmod{d}{\Omega}$.
\end{nota}

\subsubsection{\bf Analytic spaces}In this paper we will consider $k$-analytic spaces in the sense of
Berkovich (see \cite{Ber}). We denote by $\Ak$ (resp. $\Pk$) the
analytic affine (resp. projective)
line over the ground field, with coordinate $T$. Let $(X, \cO_X)$ be an analytic
space. For any $x\in X$, the residue field of the local ring $\cO_{X,x}$
is naturally valued  and we denote by $\Hx$ its completion.

Let $\Omega\in E(k)$ and $c\in\Omega$. For $r\in \R+^*$ we set

\begin{equation}\discf{\Omega}{c}{r}:=\{x\in \A{\Omega};\; |T(x)-c|\leq r\}\end{equation}
and
\begin{equation}\disco{\Omega}{c}{r}:=\{x\in \A{\Omega};\; |T(x)-c|< r\}\end{equation}
Denote by $x_{c,r}$ the unique point in the Shilov boundary of
$\discf{\Omega}{c}{r}$.
For $r_1$, $r_2\in\R+$, such $0<r_1\leq r_2$ we set
\begin{equation}\Couf{\Omega}{c}{r_1}{r_2}:=\{x\in\A{\Omega};\; r_1\leq |T(x)-c|\leq
  r_2\}\end{equation}
and for $r_1<r_2$ we set:
\begin{equation}\Couf{\Omega}{c}{r_1}{r_2}:=\{x\in\A{\Omega};\; r_1< |T(x)-c|<
  r_2\}\end{equation}
We will drop the subscript $\Omega$ when no confusion is possible.

  \begin{Defi}\label{sec:type-points-a_k}
    Let $x\in \Ak$ and $y\in\pik{\widehat{k^{alg}}}\-1(x)$. We define
    the radius of $x$ to be the value:
    \begin{equation}
      \label{eq:60}
    r_k(x)=\Inf_{c\in k^{alg}}|T(y)-c|.  
  \end{equation}

    It does not depend on the choice of $y$. We will drop $k$ when no
    confusion is possible. 
  \end{Defi}

  \begin{rem}
    Assume that $k$ is algebraically closed. For a point
    $x_{c,r}$ of type (1), (2) or (3), we have $r_k(x_{c,r})=r$.
  \end{rem}

  \begin{rem}
    Let $\Omega\in E(k)$ and $\pik{\Omega}:\A{\Omega}\to \Ak$ be the
    canonical projection. Let
    $x\in\Ak$ and $y\in \pik{\Omega}\-1(x)$. In general, we have
    \begin{equation}r_k(x)\ne r_\Omega(y).\end{equation}
    We will show further when the equality holds.
  \end{rem}

Assume that $k$ is algebraically closed. Let $c\in k$. The following
map
\begin{equation}\Fonction{[0,+\infty)}{\Ak}{r}{x_{c,r}}\end{equation}
induces a homeomorphism between $[0,+\infty)$ and its image.

\begin{nota}\label{sec:type-points-a_k-1}
  We will denote by $[x_{c,r},\infty)$ (resp. $(x_{c,r},+\infty)$) the image of
  $[r,+\infty)$ (resp. $(r,\infty)$), by $[x_{c,r},x_{c,r'}]$
  (resp. $(x_{c,r},x_{c,r'}]$, $[x_{c,r},x_{c,r'})$, $(x_{c,r},x_{c,r'})$)  the image
  of $[r,r']$ (resp. $(r,r']$, $[r,r')$, $(r,r')$).
\end{nota}

\begin{nota}\label{sec:bf-analytic-spaces}
  Let $X$ be an affinoid domain of $\Ak$ and $f\in\cO_X(X)$. We can
  see $f$ as an analytic morphism $X\to \Ak$ that we still denote by
  $f$. In particular, for a polynomial $\sum_i a_iT^i\in k[T]$ we
  denote by $\sum_ia_i x^ i$, the image of $x$ by this polynomial.
\end{nota}

 \subsubsection{\bf Universal points and fiber of a point under 
      extension of scalars }\label{sec:univ-points-fiber-1}

    \begin{Defi}\label{sec:univ-points-fiber-4}
      A point $x\in \Ak$ is said to be {\em universal} if, for
      any $\Omega\in E(k)$, the tensor norm on the algebra $\Hx\ct_k
      \Omega$ is multiplicative. In this case, it defines a point of
      $\pik{\Omega}\-1(x)$ in $\A{\Omega}$ that we denote by $\uni{\Omega}(x)$. 
    \end{Defi}

    \begin{Pro}\label{sec:univ-points-fiber-2}
      If $k$ is algebraically closed, any point $x\in
      \Ak$ is universal.
    \end{Pro}
    \begin{proof}
     See \cite[Corollary~3.14.]{poi}. 
    \end{proof}

    \begin{Theo}\label{sec:univ-points-fiber}
      Suppose that $k$ is algebraically closed. Let $\Omega\in E(k)$
      algebraically closed.
      \begin{itemize}
      \item If $x$ is of  type ($i$), where $i\in\{1,2\}$, then so is
        $\uni{\Omega}(x)$. If $x$ is of type ($j$), where $j\in\{3,4\}$, then
        $\uni{\Omega}(x)$ is of type ($j$) or (2).
      \item The fiber $\pik{\Omega}\-1(x)$ is connected and the
        connected components of
        $\pik{\Omega}\-1(x)\setminus\{\uni{\Omega}(x)\}$ are
        open disks with boundary $\{\uni{\Omega}(x)\}$. Moreover they
        are open in $\A{\Omega}$.
      \end{itemize}
    \end{Theo}
    \begin{proof}
      See \cite[Theorem~2.2.9]{np2}.
    \end{proof}

    \begin{cor}\label{sec:berkovich-line}
      Let $x\in\Ak$ be a
     point of type (i), where $i\in\{2,3, 4\}$. Let $\Omega\in E(k)$
     algebraically closed such that there is no isometric
     $k$-embedding $\Hx\hookrightarrow \Omega$. Then \begin{equation}\pik{\Omega}\-1(x)=\{\uni{\Omega}(x)\}.\end{equation} 
   \end{cor}

   \begin{proof}
     Recall that $\pik{\Omega}\-1(x)\setminus \{\uni{\Omega}(x)\}$
     is a disjoint union of open disks
     (cf. Theorem~\ref{sec:univ-points-fiber}). Therefore, if it is
     not empty, it contains points of
     type (1) which gives rise to isometric $k$-embeddings
     $\Hx\hookrightarrow \Omega$, which contradicts the hypothesis.
   \end{proof}

   \begin{Lem}\label{sec:univ-points-fiber-3}
     Let $\Omega\in E(k)$ such that $k^{alg}\subset\Omega$. Let $x\in
     \Ak$. Then for any $y\in \pik{\widehat{k^{alg}}}\-1(x)$ we have:
     \begin{equation}r_{\Omega}(\Uni{\widehat{k^{alg}}}{\Omega}(y))=r_k(x).\end{equation}
   \end{Lem}

   \begin{proof} If $x$ is of type (1), then for any $y\in
     \pik{\widehat{k^{alg}}}\-1(x)$, the point 
     $\Uni{\widehat{k^{alg}}}{\Omega}(y)$ is of type (1). Hence, we
     obtain $r_{\Omega}(\Uni{\widehat{k^{alg}}}{\Omega}(y))=r_k(x)=0$.
     
   If $x$ is of type (2) or (3), then any $y\in
   \pik{\widehat{k^{alg}}}\-1(x)$ is of the form $x_{c,r_k(x)}$, where
   ${c\in k^{alg}}$. Since the morphism
   ${\fdiscf{\widehat{k^{alg}}}{c}{r_k(x)}\to \Hy}$ is isometric, then so
   is \begin{equation}
     \fdiscf{\Omega}{c}{r_k(x)}\to\Hy\ct_{\widehat{k^{alg}}}\Omega.\end{equation}
   Therefore, we have
   $\Uni{\widehat{k^{alg}}}{\Omega}(y)=x_{c,r_k(x)}$ in
   $\A{\Omega}$. Hence
   $r_\Omega(\Uni{\widehat{k^{alg}}}{\Omega}(y))=r_k(x)$.
   
   Now suppose that $x$ is a point of type (4), then for any
   $y\in\pik{\kac}\-1(x)$ there exists a family of nested disks
   $\EE$ indexed by $(I,\leq)$ such that $\bigcap_{i\in
     I}\discf{\kac}{c_i}{r_i}=\{y\}$. Note that we have
   $r_k(x)=r_k(y)=\inf_{i\in I}r_i$. Then we have:
   \begin{equation}\piro{\Omega}{\kac}\-1(y)=\bigcap_{i\in I}\discf{\Omega}{c_i}{r_i}.\end{equation}
We distinguish two cases: the first is $\bigcap_{i\in
  I}\discf{\Omega}{c_i}{r_i}=\{\Uni{\Omega}{\kac}(y)\}$. Then, we
have:
\begin{equation}r(\Uni{\kac}{\Omega}(y))=\inf_{i\in I}r_i=r_{\kac}(y)=r_k(x).\end{equation}
The second is $\bigcap_{i\in
  I}\discf{\Omega}{c_i}{r_i}=\discf{\Omega}{c}{r_{\kac}(y)}$, where
$c\in\Omega\setminus\kac$. Here, $\Uni{\kac}{\Omega}(y)$ coincides
with the Shilov boundary of
$\discf{\Omega}{c}{r_{\kac}(y)}$. Therefore we have
\begin{equation}r(\Uni{\kac}{\Omega}(y))=r_{\kac}(y)=r_k(x).\end{equation}\end{proof}

\subsubsection{\bf Sheaf of differential forms and étale morphisms}Here we do
not give the general definition of  sheaf of differential forms given in
\cite[\S 1.4.]{ber2}, but only how it looks like in the case of an analytic
domain of $\Ak$. Let $X$ be an analytic domain of $\Ak$. Let $T$ be
the global coordinate function on $\Ak$ fixed
above. It induces a global coordinate function $T$ on $X$. The sheaf
of differential forms $\Omega_{X/k}$ of $X$ is free with $\mr{dT}$ as
a basis.

Let $\dT:\cO_X\to \cO_X$
be the formal derivation with respect to $T$. In this setting the canonical
derivation $d_{X/k}$ satisfies:
\begin{equation}
  \label{eq:8}
  \Fonction{d_{X/k}: \cO_X(U)}{\Omega_{X/k}(U)}{f}{\dT(f)\cdot \mr{dT}},
\end{equation}
where $U$ is an open subset of $X$.

\begin{Lem}\label{sec:sheaf-diff-etale-3}
  Let $X$ and $Y$ be two connected analytic domain of $\Ak$ and let $T$
  (resp. $S$) be a coordinate function defined on $X$ (resp. $Y$). Let $\phi: Y\to X$
  be a finite morphism of $k$-analytic spaces and let
  $\phi^\#:\phi\-1 (\cO_X)\to\cO_Y$ be the induced sheaves
  morphism. If $\phi$ is étale then for each analytic sub domain $Y'$ of $Y$
  \begin{equation}
    \label{eq:9}
\Fonction{\phi^*(\Omega_{X/k})(Y')=(\phi\-1(\Omega_{X/k})(Y')\ot_{\phi\-1(\cO_X)(Y')}\cO_Y(Y')}{\Omega_{Y/k}(Y')}{h.dT\ot g}{\d(\phi^\#(T)) \phi^\#(h) g.dS}
\end{equation}
is an isomorphism of $\cO_Y(Y')$-modules (resp. $\cO_Y(Y')$-Banach
module if $Y'$ is an affinoid domain). If morover $X$ and $Y$ are
smooth, if the morphisms \eqref{eq:9} are isomorphisms then $\phi$ is étale. 
\end{Lem}

\begin{proof}
  See~\cite[Proposition~3.5.3]{ber2}.
\end{proof}

\begin{rem}\label{sec:sheaf-diff-etale-2}
  Note that, any morphism $\phi:Y\to X$ between two connected open analytic  domains
  of $\Ak$ is obtained by a convenient choice of an element $f$ of
  $\cO_Y(Y)$, which is the image of $T$ by $\phi^\#$. In this setting,
  assume that $\phi$ is finite, then  $\phi$ is étale if and only if $\d(f)$
  is invertible in $\cO_Y(Y)$.
\end{rem}

\begin{cor}
  Let $\phi:Y\to X$ be a finite morphism between connected open analytic domains of
  $\Ak$. If char($k$)$=0$, then for each $x\in X$ of type (2), (3) or
  (4) there exists an affinoid neighbourhood $U$ of $x$ in $X$ such that
  $\phi|_{\phi\-1(U)}:\phi\-1(U)\to U$ is an étale morphism.
\end{cor}

\begin{proof}
  Let $f$ be $\phi^{\#}(T)$. Since $\crk=0$, for each $x\in X$ not of type
  (1) we have: $\d(f)(x)=0$ if and only if
  $f\in k$. Since $\phi$ is finite, $f\not\in k$. Hence, there exists
  an affinoid neighbourhood $U$ of $x$ such that $\d(f)$ is
  invertible in $\cO(U)$. The result follows by Remark~\ref{sec:sheaf-diff-etale-2}. 
\end{proof}

\begin{Lem}\label{sec:sheaf-diff-etale}
  Let $\phi:Y\to X$ be a finite morphism between affinoid domains of
  $\Ak$ and let $T$
  (resp. $S$) be a coordinate function defined on $X$ (resp. $Y$). Let $y$ be a point of type (2), (3) or (4). Then the induced
  extension $\h{\phi(y)}\hookrightarrow \h{y}$ is finite and we have
  \begin{equation}\h{y}=\h{\phi(y)}(S(y))\simeq \bigoplus_{i=0}^{n-1}\h{\phi(y)}\cdot S(y)^i.\end{equation}
\end{Lem}

\begin{proof}
  Since $\phi:Y\to X$ is finite, we have $[\h{y}:\h{\phi(y)}]=n$ for
  some $n\in\NN$ (cf. \cite[Lemma~2.24]{np2}). Therefore, $S(y)$ is algebraic over
  $\h{\phi(y)}$. Hence, $\h{\phi(y)}(S(y))$ is a complete intermediate finite
  extension that contains
  $k(S(y))$. Because $k(S(y))$ is dense in $\Hy$, we obtain
   \begin{equation}\h{y}=\h{\phi(y)}(S(y))\simeq \bigoplus_{i=0}^{n-1}\h{\phi(y)}\cdot S(y)^i.\end{equation}
 \end{proof}

 \begin{Pro}\label{sec:sheaf-diff-etale-1}
   Assume that $k$ is algebraically closed. Let $\phi: Y\to X$ be a finite étale cover between two analytic domains of
   $\Ak$. Let $y\in Y$ and $x=\phi(y)$. Let $U$ be a neighbourhood
   of $y$ such that $U$ is a connected component of
   $\phi\-1(\phi(U))$. Then we have
   \begin{equation}[\Hy:\Hx]=\# U\cap\phi\-1({b})\end{equation}
   for each $b\in \phi(U)\cap k$.
 \end{Pro}

 \begin{proof}
   See \cite[Corollary 9.17]{Bak10} and \cite[(3.5.4.3)]{Duc}.
 \end{proof}

 \subsubsection{\bf Differential equations over an affinoid domain}
 Let $X$ be an analytic domain of $\Ak$. A differential equation over
 $X$ is a locally free $\cO_X$-module $\F$ of finite rank together with a
 connection $\nabla:\F\to \F\ot_{\cO_X}\Omega_{X/k}$. Consider an
 isomorphism between $\cO_X$ and $\Omega_{X/k}$ i.e a morphism of
 $\cO_X$-module of the form:

 \begin{equation}\label{eq:10}
 \Fonction{\Omega_{X/k}(U)}{\cO_X(U)}{f\cdot\mr{dT}}{g\cdot f,}
\end{equation}
where $U$ is an open subset of $X$ and $g$ is an invertible element of
$\cO_X(X)$. By composing with $d_{X/k}$  we obtain a $k$-linear derivation $d$ on 
$\cO_X$, and we have $d=g\dT$. For each $x\in X\setminus k$, $d$
extends to a $k$-linear bounded derivation on $\Hx$. Moreover, by composing with the isomorphism \eqref{eq:10}
$(\F_x\ot_{\cO_{X,x}}\Hx,\nabla)$ can be seen as a differential module
over $(\Hx,d)$.

\subsubsection{\bf Spectrum of a differential module} Let us recall the definition of the spectrum in the sense of Berkovich.

\begin{Defi}
  Let $E$ be a $k$-Banach algebra and $f\in E$. The spectrum of $f$ is
  the set $\Sigma_{f,k}(E)$ of points $x\in \Ak$ such that the element
  $f\ot 1-1\ot T(x)$ is not invertible in the $k$-Banach algebra $E\ct_k\Hx$.
\end{Defi}

\begin{rem}
  If there is no confusion we denote the spectrum of $f$, as an
  element of $E$, just by $\Sigma_f$.
\end{rem}

\begin{rem}
  The set $\Sigma_f\cap k$ coincides with the classical spectrum,
  i.e.
  \begin{equation}\Sigma_f\cap k=\{c\in k;\; f-c\; \text{ is not invertible in } E\}.\end{equation}
\end{rem}

Let $\Omega\in E(k)$. As in our paper \cite{Cons}, for each
differential module $(M,\nabla)$ in $\Dmod{d}{\Omega}$, we assign  a spectrum in the sense of
Berkovich in the following way: we can endow naturally $M$ by a
structure of $\Omega$-Banach space, hence by a structure of $k$-Banach space, for which
$\nabla:M\to M$ is a bounded, i.e. $\nabla$ is an element of $\Lk{M}$ the $k$-Banach
algebra of $k$-linear bounded operators with respect to operator norm. Then the
spectrum of $(M,\nabla)$ is the spectrum of $\nabla$ as an element of
$\Lk{M}$.\footnote{Since the $\Omega$-Banach structures on $M$ are
  equivalent, the associated spectrum is well defined.} This spectrum
is a compact non-empty set, moreover the smallest closed disk centered
at zero  and containing $\Sigma_{\nabla,k}(\Lk{M})$ has radius equal to
$\nsp{\nabla}:=\lim\limits_{m\to
    +\infty}\nor{\nabla}_\mr{op}^\fra{m}$
(c.f. \cite[Theorem~7.1.2]{Ber}). Stressing also that this spectrum is
invariant by bi-bounded isomorphisms of differential modules. However,
it depends on the choice of the derivation. Indeed, for example the
derivations $\dT$ and $T\dT$ defined over $\h{x_{0,r}}$, both are
associated to the trivial differential equation defined over some open
neighborhood $X$ of $x_{0,r}$ (i.e $(\cO_X,d_{X/k})$), but we proved,
in $\crk=0$, that they have a
different spectra (see \cite{Cons} and \cite{Azz21}). In the next
section we will see that, in some cases, two different derivations
can be linked by an étale morphism, and so is for their spectra. Since
the derivation $(T-c)\dT$ with $(c\in k)$ has a good behavior under
ramification of $(T-c)$, we will privilege the choice of these kind of
derivations.

Let us recall some technical results from \cite{Cons}, and another one that are
very useful for the computation of the spectrum.

\begin{Lem}[{\cite[Lemma~2.20]{Cons}}]\label{sec:spectr-diff-module-2}
  Let $E$ be a $k$-Banach algebra and $f\in E$. If $a\in \Ak\setminus
  \Sigma_f$, then the largest open disk centered at $a$ contained in
  $\Ak\setminus \Sigma_f$ has radius equal to $\nsp{(f-a)\-1}\-1$.
\end{Lem}

\begin{Lem}[{\cite[Lemma~2.29, Remark~2.30]{Cons}}]\label{sec:spectr-vers-youngs-1}
  Let $(M,\phi)$, $(M_1,\phi_1)$ and $(M_2,\phi_2)$ be three
  $k$-Banach space endowed with a $k$-linear bounded operators, such
  that we have an exact sequence
\begin{equation}0\to M_1\overset{i}{\to} M\overset{p}{\to} M_2\to 0,\end{equation}
with $i\circ \phi_1=\phi\circ i$ and $p\circ \phi =\phi_2\circ p$.
Then we have
\begin{equation}
  \label{eq:53}
  (\Sigma_{\phi_1,k}(\Lk{M_1})\cup\Sigma_{\phi_2,k}(\Lk{M_2}))\setminus(\Sigma_{\phi_1,k}(\Lk{M_1})\cap
\Sigma_{\phi_2,k}(\Lk{M_2}))\subset\Sigma_{\phi,k}(\Lk{M}),
\end{equation}
and
\begin{equation}
  \label{eq:59}
  \Sigma_{\phi,k}(\Lk{M})\subset
\Sigma_{\phi_1,k}(\Lk{M_1})\cup\Sigma_{\phi_2,k}(\Lk{M_2}).
\end{equation}

Moreover, if $x\notin \Sigma_{\phi,k}(\Lk{M})$, then $\phi_1\ot 1-1\ot
T(x)$ is left invertible and $\phi_2\ot 1-1\ot T(x)$ is right invertible.
\end{Lem}

\begin{cor}\label{sec:spectr-vers-youngs-2}
 We keep the assumptions of Lemma~\ref{sec:spectr-vers-youngs-1}. If
 in addition we have another exact sequence of the form:
\begin{equation}0\to M_2\overset{i'}{\to} M\overset{p'}{\to} M_1\to 0,\end{equation}
with $i'\circ \phi_2=\phi\circ i'$ and $p'\circ \phi =\phi_1\circ
p$. Then we have $\Sigma_{\phi,k}(\Lk{M})=
\Sigma_{\phi_1,k}(\Lk{M_1})\cup\Sigma_{\phi_2,k}(\Lk{M_2})$
\end{cor}
\begin{rem}\label{sec:spectr-diff-module}
  In particular, if $(M,\nabla)=(M_1,\nabla_1)\oplus (M_2,\nabla_2)$
  as differential modules, then we have $\Sigma_\nabla=\Sigma_{\nabla_1}\cup\Sigma_{\nabla_2}$.
\end{rem}

\begin{Lem}[{\cite[p.2]{bousp}}]\label{sec:spectr-diff-module-1}
  Let $P(T)\in k[T]$. Let $E$ be a $k$-Banach algebra and let $f\in
  E$. Then we have:
  \begin{equation}\Sigma_{P(f),k}(E)=P(\Sigma_{f,k}(k)).\end{equation}
\end{Lem}

\subsection{Push-forward and pull-back of a differential module and
  their spectra}

Let $(\Omega,d)$ be a differential field and $(\Omega',d')$ be a
finite differential extension of $(\Omega,d)$. Then we have two
natural functors:

\begin{equation}
  \label{eq:2}
\Fonction{\Dmod{d}{\Omega}}{\Dmod{d'}{\Omega'}}{(M,\nabla)}{(M\ot_\Omega
      \Omega',\nabla_{M\ot_\Omega\Omega'})}  
\end{equation}

  where $\nabla_{M\ot_\Omega \Omega'}:= \nabla\ot 1+1\ot d'$. Note that if
$\dim_\Omega M=n$, than so is $(M\ot_\Omega
  \Omega',\nabla_{M\ot_\Omega\Omega'})$,

  \begin{equation}
    \label{eq:3}
  \Fonction{\Dmod{d'}{\Omega'}}{\Dmod{d}{\Omega}}{(M,\nabla)}{(M_\Omega,\nabla_\Omega)}   
  \end{equation}
  where $M_\Omega$ is the restriction of scalars of $M$ via
  $\Omega\hookrightarrow \Omega'$, and $\nabla_\Omega=\nabla$ as
  $k$-linear maps. If $[\Omega':\Omega]=n'$ and $\dim_{\Omega'} M=n$,
  then $\dim_\Omega M_\Omega=n.n'$.

     \begin{nota}
                From now on we will fix $S$ to be a coordinate
                function of the analytic domain (of the affine line) where the linear
                differential equation is defined, and $T$ to be a
                coordinate function on $\Ak$ (for the computation of
                the spectrum).
              \end{nota}

              Let $Y$ and $X$ be two connected affinoid domains of $\Ak$ and $Z$
              (resp. $S$) a coordinate function on $X$ (resp. $Y$). Let
              $\phi: Y\to X$ be a finite étale morphism and
              $\phi^\#:\cO_X\to \phi_*\cO_Y$ be the induced sheaves morphism. Let
              $f:=\phi^{\#}(Z)$ and $f':=\d(f)$. Since $\phi$ is
              étale, $f'$ is invertible in $\cO_Y(Y)$
              (cf. Lemma~\ref{sec:sheaf-diff-etale-3}). To any bounded
              derivation $d=g\dz$ on $X$ we assign the bounded
              derivation on $Y$
              \begin{equation}
                \label{eq:14}
                \phi^*d:=\frac{\phi^{\#}(g)}{f'}\d
              \end{equation}
called the {\em pull-back} of $d$ by $\phi$.

Let $y\in Y$ and $x=\phi(y)$. Since the derivation $\d$ (resp. $\dz$)
extends to a bounded derivation on $Y$ (resp. $X$), the derivation $d$
(resp. $\phi^*d$) extends to a bounded derivation on $Y$ (resp. $X$). We have the commutative
diagram:
\begin{equation}\xymatrix{\Hx\ar@{^{(}->}[r]\ar[d]_d&\Hy\ar[d]^{\phi^*d}\\
    \Hx\ar@{^{(}->}[r]& \Hy}
\end{equation}
Since $\phi^{\#}$ induces a finite
extension $\Hx \hookrightarrow \Hy$, we have the {\em push-forward} functor
by $\phi$ defined as in \eqref{eq:3}:

\begin{equation}
  \label{eq:21}
  \Fonction{\phi_*:
    \Dmod{\phi^*d}{\Hy}}{\Dmod{d}{\Hx}}{(M,\nabla)}{(\phi_* M,\phi_*\nabla)}
\end{equation}
and the {\em pull-back} functor by $\phi$ defined as in~\eqref{eq:2}:

\begin{equation}
  \label{eq:19}
  \Fonction{\phi^*: \Dmod{d}{\Hx}}{\Dmod{\phi^*d}{\Hy}}{(M,\nabla)}{(\phi^*M,\phi^*\nabla)}
\end{equation}

\begin{rem}
In Sections~\ref{sec:bf-spectr-deriv} and \ref{sec:bf-frob-spectr-5}, we will describe the above
functors more explicitly.    
\end{rem}

\begin{Pro}\label{sec:push-forw-spectr}
  We have the set-theoretic equality:
  \begin{equation}\Sigma_{\nabla,k}(\Lk{M})=\Sigma_{\phi_*\nabla,k}(\Lk{\phi_*M}).
  \end{equation}
\end{Pro}

\begin{proof}
  Since $M$ and $\phi_*M$ are the same as $\Hx$-Banach spaces, then
  they are isomorphic as $k$-Banach spaces. As $\nabla$ and $\phi_*\nabla$
  coincide as $k$-linear maps, the equality of spectra holds.
\end{proof}

\section{Spectrum of a differential module with regular
  singularities}\label{sec:spectr-diff-module-4}
\begin{hyp} We assume that $k$ is algebraically closed.
\end{hyp}
\begin{nota}
  We set
  \begin{equation}
    \label{eq:4}
    \omega:=\lim\limits_{{n\to +\infty}}|n!|^{\fra{n}}=
    \begin{cases}
      |p|^{\fra{p-1}}& \crk=p>0\\
      1 & \crk=0
    \end{cases}
  \end{equation}
\end{nota}
In this part we compute the spectrum of a differential module $(M,\nabla)$ over $(\Hx,S\d)$  with
regular singularities, defined as below, for $x\in \Ak\setminus
k$. We will observe that if we fix a differential equation
$(\F,\nabla)$ over
$\Ak\setminus\{0\}$ with regular singularities (we will explain the
meaning later), we observe that the spectrum of $(\F_x\ot \Hx, \nabla)$
(as differential module over $(\Hx,S\d)$) has an interesting behavior
when we vary $x$.

\begin{Defi}
  A differential module $(M,\nabla)$ over $(\Hx,S\d)$ is said to
  be regular singular, if there exists a basis for which the
  associated matrix $G$, i.e.
\begin{equation}\nabla
    \begin{pmatrix}
      f_1\\ \vdots\\ f_n
    \end{pmatrix}
    =
    \begin{pmatrix}
      S\d f_1\\ \vdots\\ S\d f_n
    \end{pmatrix}
    +
    G \begin{pmatrix}
      f_1\\ \vdots\\ f_n
    \end{pmatrix},
  \end{equation}
  has constant entries (i.e $G\in\cM_n(k)$).  
\end{Defi}

\begin{rem}\label{sec:spectr-diff-module-3}
  As explained in \cite[Proposition~3.15]{Cons}, given a
  differential module over $(\Omega,d)$, with $\Omega\in E(k)$, such
  that:
  \begin{equation}\nabla
    \begin{pmatrix}
      f_1\\ \vdots\\ f_n
    \end{pmatrix}
    =
    \begin{pmatrix}
      d f_1\\ \vdots\\ d f_n
    \end{pmatrix}
    + G \begin{pmatrix} f_1\\ \vdots\\ f_n
    \end{pmatrix},
  \end{equation}
  with $G\in\cM_n(k)$ for some basis of $M$, then the computation
  of the spectrum of $\nabla$ is reduced to the computation of the
  spectrum of $d$. Indeed, the spectrum of $\nabla$ is
  $\Sigma_\nabla=\bigcup_{i=1}^N(a_i+\Sigma_d)$, where
  $\{a_1,\cdots, a_N\}$ are the eigenvalues of $G$.
\end{rem}

\subsection{Spectrum of the derivation $S\d$}\label{sec:bf-spectr-deriv}
\begin{Lem}\label{sec:spectr-deriv-sd}
  Let $x=x_{0,r}$, with $r>0$. The norm and spectral semi-norm of
  $S\d$ as an element of $\Lk{\h{x}}$ satisfy:
    \begin{equation}\nor{S\d}=1, \qq \nsp{S\d}=1.
    \end{equation}
  \end{Lem}

  \begin{proof}
    Since $\nor{S}=|S|=r$ and $\nor{\d}=\fra{r}$ (cf. \cite[Lemma~4.4.1]{and}),
    we have $\nor{S\d}\leq 1$. Hence also, $\nsp{S\d}\leq 1$. The map
   \begin{equation}\Fonction{\Lk{\Hx}}{\Lk{\Hx}}{\phi}{S\-1\circ\phi\circ
       S}
   \end{equation}
   is bi-bounded and induces a change of basis. Therefore, 
    as
    $S\-1\circ(S\d)\circ S=S\d+1$, we have $\nsp{S\d}=\nsp{S\d+1}$. Since $1$
    commutes with $S\d$, we have:
    \begin{equation}
      1=\nsp{1}=\nsp{S\d+1-S\d}\leq\max(\nsp{S\d+1},\nsp{S\d}).
    \end{equation}
    Consequently, we obtain
    
    \begin{equation}\nor{S\d}=\nsp{S\d}=1.
    \end{equation}
  \end{proof}

\subsubsection{The case of positive residual characteristic}

We assume here that $\crk=p>0$.
\paragraph{The case where $x\in (0,\infty)$}
We start with the case where $x=x_{0,r}$. In order to determine the
spectrum of $S\d$, we will use the push-forward by the Frobenius
map. We refer the reader to
Section~\ref{sec:frobenius-map} for the definition of Frobenius map
$\forp:\Ak\to \Ak$
and its properties. Since it induces an étale map $(\forp)^n:\Ak\setminus
\{0\}\to \Ak\setminus\{0\}$, the push-forward functor
(cf. (\ref{eq:21})) is well defined for any $x_{0,r}\in \Ak$ with
$r>0$. Recall that $(\forp)^n(x_{0,r})=x_{0,r^{p^n}}$ and
$[\h{x_{0,r}}:\h{x_{0,r^{p^n}}}]=p^n$ (cf. Properties~\ref{sec:forbenius-map}).

Let $x=x_{0,r}$ and $y:=(\forp)^n(x)$. According to formula~(\ref{eq:14})
the pull-back of the derivation $p^nS\d:\Hy\to \Hy$ is the derivation
$S\d:\Hx\to \Hx$. To avoid confusion in the following, we set $p^nS(y)\D{S(y)}:=p^nS\d$ and $S(x)\D{S(x)}:=S\d$. 

Now let $(M_{p^n},\nabla_{p^n})$ be the push-forward of the differential module
$(\Hx,S\d)$ by $(\forp)^n$. Since $M_{p^n}\simeq\Hx$ as an $\Hy$-Banach
space, and according to Lemma~\ref{sec:sheaf-diff-etale}, we can take
$\{1,S(x),\cdots,S(x)^{p^n-1}\}$ as a basis of $(M_{p^n},\nabla_{p^n})$. Since
\begin{equation}\nabla_{p^n}(S(x)^i)=S(x)\D{S(x)}(S(x)^i)=iS(x)^i,
\end{equation}
in this
basis  we have:

\begin{equation}
  \label{eq:22}
   \nabla_{p^n}%
    \left(
     \raisebox{0.5\depth}{%
       \xymatrixcolsep{1ex}%
       \xymatrixrowsep{1ex}%
       \xymatrix{f_1\ar@{.}[dddd]\\ \\\\ \\f_{p^n}\\}}%
   \right)
     =
\left(
     \raisebox{0.5\depth}{%
       \xymatrixcolsep{1ex}%
       \xymatrixrowsep{1ex}%
     \xymatrix{p^nS(y)\D{S(y)} 
       f_1\ar@{.}[ddd]\\ \\ \\p^nS(y)\D{S(y)}  f_{p^n}\\}}
   \right)
   +
\left(
     \raisebox{0.5\depth}{%
       \xymatrixcolsep{1ex}%
       \xymatrixrowsep{1ex}%
   \xymatrix{0& 0\ar@{.}[rr]\ar@{.}[ddrr] &    &0\ar@{.}[dd]\\
                    0\ar@{.}[ddrr]\ar@{.}[dd]&1\ar@{.}[ddrr]                 &    & \\
                      &                   &    &0\\
                    0\ar@{.}[rr]&                   &0  &p^n-1\\}
}
   \right)
   \left(
     \raisebox{0.5\depth}{%
       \xymatrixcolsep{1ex}%
       \xymatrixrowsep{1ex}%
       \xymatrix{f_1\ar@{.}[dddd]\\ \\\\ \\f_{p^n}\\}}%
   \right).
 \end{equation}
 On the other hand, we have an isomorphism of differential modules
 \begin{equation}
   \label{eq:23}
   (M_{p^n},\nabla_{p^n})\simeq \bigoplus_{i=0}^{p^n-1}(\Hy,p^nS(y)\D{S(y)} +i).
 \end{equation}

 \begin{nota}
   For simplicity, here we still denote $S(x)\D{S(x)}$ by $S\d$.
 \end{nota}

 \begin{Pro}\label{sec:bf-case-positive}
   Let $r>0$. Let $x$ be $x_{0,r}$. The spectrum of $S\d$ as an element of
   $\Lk{\Hx}$ is
   \begin{equation}\Sigma_{S\d,k}(\Lk{\Hx})=\ZZ_p.
   \end{equation}
 \end{Pro}

 \begin{proof}
  Since for all $l\in \NN$ we have $S\d (S(x)^l)-l(S(x)^l)=0$, we
  obtain \begin{equation}\NN\subset\Sigma_{S\d,k}(\Lk{\Hx}).
  \end{equation}
  By compactness of
  the spectrum, we
  obtain \begin{equation}\ZZ_p\subset\Sigma_{S\d,k}(\Lk{\Hx}).
  \end{equation}
  Let $n\in\NN\setminus\{0\}$. We set $y:=(\forp)^n(x)=x_{0,p^n}$. Let $(M_{p^n},\nabla_{p^n})$ be the push-forward of $(\Hx,S\d)$ by
  $(\forp)^n$. On the one hand, according to
  Proposition~\ref{sec:push-forw-spectr} we have
  \begin{equation}
    \Sigma_{S\d,k}(\Lk{\Hx})=\Sigma_{\nabla_{p^n},k}(\Lk{M_{p^n}}).\end{equation}
On
  the other hand, since we have the isomorphism~(\ref{eq:23}), we have 
  \begin{equation}\Sigma_{\nabla_{p^n},k}(\Lk{M_{p^n}})=\bigcup_{i=0}^{p^n-1}\Sigma_{p^nS(y)\D{S(y)}
      +i,k}(\Lk{\Hy})=\bigcup_{i=0}^{p^n-1}p^n\Sigma_{S\d,k}(\Lk{\Hy})+i
  \end{equation}
  (cf. Remark~\ref{sec:spectr-diff-module} and
  Lemma~\ref{sec:spectr-diff-module-1}). By
  Lemma~\ref{sec:spectr-deriv-sd}, we know that $\nsp{S\d}=1$ in $\Lk{\Hy}$. Therefore, we have
  $\Sigma_{S\d,k}(\Lk{\Hy})\subset\disf{0}{1}$ (cf. \cite[Theorem~7.1.2]{Ber}). Consequently, 
  $\Sigma_{\nabla_{p^n},k}(\Lk{M_{p^n}})\subset\bigcup_{i=0}^{p^n-1}\disf{i}{|p|^n}$. Applying this process for all $n$, we obtain
  \begin{equation}\Sigma_{S\d,k}(\Lk{\Hx})\subset\bigcap_{n\in\NN\setminus\{0\}}\bigcup_{i=0}^{p^n-1}\disf{i}{|p|^n}=\ZZ_p,
  \end{equation}
  which ends the proof.
\end{proof}

\paragraph{The case where $x\not\in (0,\infty)$} We now assume that $x\in \Ak$ is a point of type (2), (3) or (4) not
of the form $x_{0,r}$. Then there exists $c\in k\setminus\{0\}$ such that
$x\in\diso{c}{|c|}$. The logarithm map (cf. Section~\ref{sec:logarithm-2})
\begin{equation}\Log{c}:\diso{c}{|c|}\to\Ak
\end{equation}
is well defined and induces an infinite étale cover. Let
$y$ be $\Log{c}(x)$. Let $r_k:\Ak\to\R+$ be the radius map
(cf. Definition~\ref{sec:type-points-a_k}) and
$\omega=|p|^{\fra{p-1}}$. We have for $n\in \NN\setminus\{0\}$

\begin{itemize}
\item if $0< r_k(x)< |c|\omega$, then 
$0<r_k(y)<\omega$ and $[\Hx:\Hy]=1$,
\item if
$|c|\omega^{\fra{p^{n-1}}}\leq r_k(x)< |c|\omega^{\fra{p^n}}$, then $\frac{\omega}{|p|^{n-1}}\leq r_k(y)<\frac{\omega}{p^n}$ and
$[\Hx:\Hy]=p^n$ 
\end{itemize}
(cf. Properties~\ref{sec:logarithm}). Note that, since
$|T(x)|=|c|$  the inequalities above do not depend on the choice of
$c$. According to Formula~(\ref{eq:14}) the pull-back of the
derivation $\d:\Hy\to \Hy$ is the derivation $S\d:\Hx\to\Hx$.

Let $(M,\nabla)$ be the push-forward of the differential module
$(\Hx,S\d)$ by $\Log{c}$. Assume that $|c|\omega^{\fra{p^{n-1}}}\leq
r_k(x)< |c|\omega^{\fra{p^n}}$. By Lemma~\ref{sec:sheaf-diff-etale}, we can take
$\{1,S(x),\cdots, S(x)^{p^n-1}\}$ as a basis of $(M,\nabla)$. Since
\begin{equation}\nabla(S(x)^i)=S\d(S(x)^i)=iS(x)^i,
\end{equation}
in this
basis  we have:

\begin{equation}
  \label{eq:25}
   \nabla%
    \left(
     \raisebox{0.5\depth}{%
       \xymatrixcolsep{1ex}%
       \xymatrixrowsep{1ex}%
       \xymatrix{f_1\ar@{.}[dddd]\\ \\\\ \\f_{p^n}\\}}%
   \right)%
     =
\left(
     \raisebox{0.5\depth}{%
       \xymatrixcolsep{1ex}%
       \xymatrixrowsep{1ex}%
     \xymatrix{\d
       f_1\ar@{.}[ddd]\\ \\ \\\d f_{p^n}\\}}
   \right)
   +
\left(
     \raisebox{0.5\depth}{%
       \xymatrixcolsep{1ex}%
       \xymatrixrowsep{1ex}%
   \xymatrix{0& 0\ar@{.}[rr]\ar@{.}[ddrr] &    &0\ar@{.}[dd]\\
                    0\ar@{.}[ddrr]\ar@{.}[dd]&1\ar@{.}[ddrr]                 &    & \\
                      &                   &    &0\\
                    0\ar@{.}[rr]&                   &0  &p^n-1\\}
}
   \right)
   \left(
     \raisebox{0.5\depth}{%
       \xymatrixcolsep{1ex}%
       \xymatrixrowsep{1ex}%
       \xymatrix{f_1\ar@{.}[dddd]\\ \\\\ \\f_{p^n}\\}}%
   \right).
 \end{equation}
 Hence, $(M,\nabla)$ is a differential module with constant
 coefficient in the sense of \cite{Cons}. Let us recall the result
 concerning
 differential modules with constant coefficients:

  \begin{Theo}[{\cite[Theorem~4.14]{Cons}}]\label{sec:spectr-diff-equat-2}
          Let $x\in\Ak$          
          be a point of type (2), (3) or (4). Let $(M,\nabla)$ be a differential module over
          $(\Hx,\d)$ such that there exists a basis for which the associated matrix $G$
          has constant entries (i.e. $G\in\cM_\nu(k)$), and let
          $\{a_1,\cdots, a_N\}$ be the set of eigenvalues of $G$. Then we have:
            \begin{itemize}
           \item if
              $x$ is a point of type (2) or (3), then
              \begin{equation}\Sigma_{\nabla,k}(\Lk{M})=\bigcup_{i=1}^{N}\disf{a_i}{\frac{\omega}{r(x)}}.
              \end{equation}
            \item if $x$ is a point of type (4), then \begin{equation}\Sigma_{\nabla,k}(\Lk{M})=
                \begin{cases}
                  \Bigcup_{i=1}^{N}\disf{a_i}{\frac{\omega}{r(x)}}&
                  \text{If } \crk>0\\
                  &\\
                  \Bigcup_{i=1}^{N}\overline{\diso{a_i}{\frac{1}{r(x)}}}&
                  \text{If } \crk=0\\
                \end{cases},
              \end{equation}
          \end{itemize}         
\end{Theo}

 \begin{Pro}\label{sec:case-posit-resid-1}
   Let $x\in\Ak$ be a point of type (2), (3) or (4) not of the form
   $x_{0,r}$. Let $c\in k\setminus\{0\}$ such that $x\in
   \diso{c}{|c|}$. Let $y$ be $\Log{c}(x)$. If $r_k(x)\leq|c|\omega$, then the spectrum of $S\d$, as an element of $\Lk{\Hx}$, is
   \begin{equation}\Sigma_{S\d,k}(\Lk{\Hx})=\disf{0}{\frac{\omega}{r_k(y)}}.
   \end{equation}
   If $|c|\omega^{\fra{p^{n-1}}}<r_k(x)\leq
   |c|\omega^{\fra{p^n}}$ with $n\in \NN\setminus\{0\}$, then the
   spectrum is a disjoint union of $p^n$ closed disks 
   \begin{equation}\Sigma_{S\d,k}(\Lk{\Hx})=\bigcup_{i=0}^{p^n-1}\disf{i}{\frac{\omega}{r_k(y)}}.
   \end{equation}
 \end{Pro}

 \begin{proof}
   Assume that $r_k(x)<|c|\omega$. Since $[\Hx :\Hy]=1$, the push-forward of
   $(\Hx,S\d)$ by $\Log{c}$ is isomorphic to $(\Hy,\d)$. Therefore, by
   Theorem~\ref{sec:spectr-diff-equat-2} and Proposition
   \ref{sec:push-forw-spectr} we obtain
   \begin{equation}\Sigma_{S\d,k}(\Lk{\Hx})=\Sigma_{\d,k}(\Lk{\Hy})=\disf{0}{\frac{\omega}{r_k(y)}}.
   \end{equation}
   We now assume that $|c|\omega^{\fra{p^{n-1}}}\leq
   r_k(x)<|c|\omega^{\fra{p^n}}$ with $n\in\NN\setminus\{0\}$. Let
   $(M,\nabla)$ be the push-forward of $(\Hx,S\d)$ by $\Log{c}$. Since we
   have the formula~(\ref{eq:25}) and according to
   Propositions~\ref{sec:push-forw-spectr}
   and Theorem~\ref{sec:spectr-diff-equat-2}, we have

   \begin{equation}\Sigma_{S\d,k}(\Lk{\Hx})=\Sigma_{\nabla,k}(\Lk{M})=\bigcup_{i=0}^{p^n-1}\disf{i}{\frac{\omega}{r_k(y)}}.
   \end{equation}
   If moreover $|c|\omega^{\fra{p^{n-1}}}<r_k(x)$, then
   $|p|^n<\frac{\omega}{r_k(y)}<|p|^{n-1}$. Consequently, the spectrum
   $\Sigma_{S\d,k}(\Lk{\Hx})$ is a disjoint union of $p^n$ disks. For
   the case where $r_k(x)=|c|\omega^{\fra{p^n}}$ with $n\in \NN$, we
   have $r_k(y)=\frac{\omega}{|p|^n}$
   (cf. Properties~\ref{sec:logarithm}).  For all $0\leq i\leq p^n-1$
   and $1\leq l \leq p-1$, we have
   $\disf{i}{|p|^n}=\disf{i+lp^n}{|p|^n}$. Hence, we obtain

   \begin{equation}\Sigma_{S\d,k}(\Lk{\Hx})=\bigcup_{i=0}^{p^{n+1}-1}\disf{i}{|p|^n}=\bigcup_{i=0}^{p^n-1}\disf{i}{|p|^n},
   \end{equation}
   which is obviously a disjoint union.
 \end{proof}

 \begin{cor}\label{sec:case-posit-resid-2}
   Let $x\in\Ak$ be a point of type (2), (3) or (4) not of the form
   $x_{0,r}$. Let $c\in k\setminus\{0\}$ such that $x\in
   \diso{c}{|c|}$. Let $y$ be $\Log{c}(x)$. The spectrum of $S\d$, as an
   element of $\Lk{\Hx}$, is
    \begin{equation}\Sigma_{S\d,k}(\Lk{\Hx})=\bigcup_{i\in
        \NN}\disf{i}{\frac{\omega}{r_k(y)}}.
    \end{equation}  
  \end{cor}
  \begin{proof}
    By Proposition~\ref{sec:case-posit-resid-1}, we have
    $\Sigma_{S\d,k}(\Lk{\Hx})=\bigcup_{i=0}^{p^n-1}\disf{i}{\frac{\omega}{r_k(y)}}
    $ for some $n\in\NN$. Since for all $l\in \NN$ we have $S\d (S(x)^l)-l(S(x)^l)=0$, we
    obtain \begin{equation}\NN\subset\Sigma_{S\d,k}(\Lk{\Hx}).
    \end{equation}
    Therefore, for each $l\in \NN$ there exists $0\leq i_l\leq p^n-1$
    such that
    $\disf{l}{\frac{\omega}{r_k(y)}}=\disf{i_l}{\frac{\omega}{r_k(y)}}$. Consequently,
        we obtain   $\bigcup_{i\in
          \NN}\disf{i}{\frac{\omega}{r_k(y)}}\subset\Sigma_{S\d,k}(\Lk{\Hx})$
        which ends the proof.
  \end{proof}
  \subsubsection{ The case of residue characteristic zero}
  We assume here that $\crk=0$.
\paragraph{The case where $x\in (0,\infty)$}
  \begin{Pro}
    Let $x\in \Ak$ be a point of type (2) of the form
    $x_{0,r}$. The spectrum of $S\d$, as an element of $\Lk{\Hx}$, is
    \begin{equation}\Sigma_{S\d,k}(\Lk{\Hx})=\disf{0}{1}.
    \end{equation}
  \end{Pro}
  \begin{proof}
    We set $d:=S\d$ and $\Sigma_d=\Sigma_{d,k}(\Lk{\Hx})$. Since $\nsp{d}=1$, we have
    $\Sigma_d\subset\disf{0}{1}$. Let $y\in \disf{0}{1}$. We set
    $A_{\Hy}=\Hx\ct_k\Hy$ and $d_{\Hy}=S\d:A_{\Hy}\to A_{\Hy}$. From
    \cite[Lemma~2.5]{Cons} we have a bounded morphism:
    \begin{equation}\Lk{\Hx}\ct_k\Hy\to \LL{\Hy}{A_{\Hy}}.
    \end{equation}
    The image of $d\ot 1$ by this morphism is the derivation
    $d_{\Hy}$. We now show that the image of $d\ot 1-1\ot T(y)$ is not
    invertible in $\LL{\Hy}{A_{\Hy}}$. Let $\alpha$ be an element of $k$
    that corresponds to the class $\tilde{\alpha}$ in $\rk$. We have the following decomposition \begin{equation}\Hx=\widehat{\bigoplus}_{\tilde{\alpha}\in
                    \tilde{k}}\{\sum_{i\in\NN^*} \frac{a_{\alpha
                      i}}{(S+\gamma\alpha)^i}|\; a_{\alpha i}\in
                  k,\; \lim\limits_{i\to +\infty}|a_{\alpha i}|r^{-i}= 0  \}\oplus
      \fdisf{0}{r}  
    \end{equation}
    with $\gamma\in k$ and $|\gamma|=r$
    (cf. \cite[Theorem~2.1.6]{Chr}, \cite[Proposition~2.10]{Cons}). Therefore, we obtain the isometric isomorphism
\begin{equation}\Hx\ct_k\Hy\simeq\widehat{\bigoplus}_{\tilde{\alpha}\in
                    \tilde{k}}\{\sum_{i\in\NN^*} \frac{a_{\alpha
                      i}}{(S+\gamma\alpha)^i}|\; a_{\alpha i}\in
                  \Hy,\; \lim\limits_{i\to +\infty}|a_{\alpha i}|r^{-i}= 0  \}\oplus
                  \fdiscf{\Hy}{0}{r} .
                \end{equation}
                Each Banach space of the completed direct sum is
                stable under $d_{\Hy}-T(y)$. The operator
                $d_{\Hy}-T(y)$ is not surjective. Indeed, let $c:=\gamma\alpha$ with
                $\tilde{\alpha}\in\rk\setminus\{0\}$ and let
                $g=\frac{c}{S-c}$. If there exists $f\in \Hx\ct_k \Hy$
                such that $(d_{\Hy}-T(y))(f)=g$, then we can choose $f$ of the form
                $f=\sum_{i\in\NN\setminus\{0\}}\frac{a_i}{(S-c)^i}$,
                such that for each $i\in \NN\setminus\{0\} $ we
                have
                \begin{equation}
                  a_i=-\frac{(-c)^{i}(i-1)!}{\prod_{j=1}^i(T(y)+j)}.
                \end{equation}
                We observe that $|a_i|r^{-i}\geq1$ for each
              $i\in\NN\setminus\{0\}$. This means that such $f$ does not exist in
                $\Hx\ct_k\Hy$. Hence, $d\ot 1-1\ot T(y)$ is not
                invertible in $\Lk{\Hx}$ and we conclude that
                $\disf{0}{1}\subset \Sigma_d$.  
  \end{proof}

 The proof of the following proposition is almost similar to
 \cite[Proposition~3.7]{Azz21}, but it is in a more general context.
  \begin{Pro}
    Let $x\in \Ak$ be a point of type (3) of the form
    $x_{0,r}$. The spectrum of $S\d$ as an element of $\Lk{\Hx}$ is
    \begin{equation}\Sigma_{S\d,k}(\Lk{\Hx})=\ZZ\cup\{x_{0,1}\}.
    \end{equation}
  \end{Pro}
  \begin{proof}
    We set $d:=S\d$ and $\Sigma_{d-n}:=\Sigma_{d-n,k}(\Lk{\Hx})$. As $\nsp{d}=1$
    (cf. Lemma~\ref{sec:spectr-deriv-sd}), we have
    $\Sigma_d\subset\disf{0}{1}$. Recall that
    \begin{equation}\Hx=\fcouf{0}{r}{r}=\{\sum_{i\in \ZZ}a_i S^i;\;
      \lim\limits_{|i|\to \infty}|a_i|r^i=0\}.
    \end{equation}
    Let $a\in k\cap\disf{0}{1}$. If $a\in \ZZ$,  then we have
    $(d-a)(S^a)=0$. Hence, $d-a$ is not
    injective and $\ZZ\subset \Sigma_d$. As the spectrum is
    compact, we have $\ZZ\cup\{x_{0,1}\}\subset \Sigma_d$. If
    $a\not\in \ZZ$, then $d-a$ is invertible in $\Lk{\Hx}$. Indeed, let
    $g(S)=\sum_{i\in \ZZ}b_i S^i\in \Hx$, if there exists
    $f=\sum_{i\in \ZZ}a_i S^i\in \Hx$ such that $(d-a)f=g$ , then for
    each $i\in \ZZ$ we have
    \begin{equation} a_i= \frac{b_i}{(i-a)}.
    \end{equation}
    If there exists $i_0\in\ZZ$ such that $a\in\diso{i_0}{1}$, then for
    each $i\ne i_0$ we have $|a_i|=|b_i|$ and
    $|a_{i_0}|=\frac{|b_{i_0}|}{|i_0-a|}$. Otherwise, for each
    $i\in\ZZ$ we have $|a_i|=|b_i|$. This means that $f$ is unique
    and converges in $\Hx$. We obtain also $|f|\leq \frac{|g|}{|i_0-a|}$ or
    $|f|= |g|$. Consequently, the set theoretical inverse
    $(d-a)^{-1}$ is bounded. We claim that if $a\in \diso{i_0}{1}$
    then $\nsp{(d-a)\-1}=\fra{|i_0-a|}$, otherwise
    $\nsp{(d-a)\-1}=1$. Indeed, in the first case, similar computations show that
    $\nor{(d-a)^{-n}}\leq \fra{|i_0-a|^n}$. Since
    $(d-a)^{-n}(S^{i_0})=\frac{S^{i_0}}{(i_0-a)^n}$, the equality
    holds and we obtain $\nsp{(d-a)\-1}=\fra{(i_0-a)}$. In the second
    case, by the above computations $(d-a)\-1$ is an isometry. Therefore, we have
    $\nsp{(d-a)\-1}=1$. Setting $R_a:=\inf_{i\in\ZZ}|i-a|$, we have
    $\nsp{(d-a)\-1}=\fra{R_a}$. According to
    Lemma~\ref{sec:spectr-diff-module-2}, we have $\diso{a}{R_a}\subset
    \Ak\setminus\Sigma_d$.

    In order to end the proof, since $\disf{0}{1}=\bigcup_{a\in
      k\setminus \ZZ} \diso{a}{R_a}\cup\bigcup_{n\in \ZZ}[n,x_{0,1}]$, it is enough to show that
    $(n,x_{0,1})\subset\Ak\setminus \Sigma_d$ for all $n\in \ZZ$. Let
    $n\in \ZZ$. Then we have 

\begin{equation}\Hx=k.S^n\oplus
      \widehat{\bigoplus}_{i\in\ZZ\setminus\{n\}}k.S^i.
    \end{equation}
    The operator $(d-n)$
      stabilizes both $k.S^n$ and
      $\widehat{\bigoplus}_{i\in\ZZ\setminus\{n\}}k.S^i$. We set
      $\nabla_1:=(d-n)|_{k.S^n}$ and
      $\nabla_2 :=(d-n)|_{\widehat{\bigoplus}_{i\in\ZZ\setminus\{n\}}k.S^i}$. We
      set  $\Sigma_{\nabla_1}:=\Sigma_{\nabla_1,k}(\Lk{k.S^n})$ and
      $\Sigma_{\nabla_2}:=\Sigma_{\nabla_2,k}(\Lk{\widehat{\bigoplus}_{i\in\ZZ\setminus\{n\}}k.S^i})$. We
      have $\nabla_1=0$. By Lemma~\ref{sec:spectr-vers-youngs-1}, we have:
    \begin{equation}\Sigma_{d-n}=\Sigma_{\nabla_1}\cup\Sigma_{\nabla_2}=\{0\}\cup
        \Sigma_{\nabla_2}.
      \end{equation}
      Now we prove that
      \begin{equation}\diso{0}{1}\cap\Sigma_{\nabla_2}=\emptyset.
      \end{equation}
      The operator $\nabla_2$ is invertible in
      $\Lk{\widehat{\bigoplus}_{i\in\ZZ\setminus\{n\}}k.S^i}$. Indeed,
      let $g(S)=\sum_{i\in \ZZ\setminus\{n\}}b_i S^i\in \widehat{\bigoplus}_{i\in\ZZ\setminus\{n\}}k.S^i$. If there exists
    $f=\sum_{i\in \ZZ\setminus\{n\}}a_i S^i\in \widehat{\bigoplus}_{i\in\ZZ\setminus\{n\}}k.S^i$ such that $\nabla_2(f)=g$ , then for
    each $i\in \ZZ\setminus\{n\}$ we have
    \begin{equation} a_i= \frac{b_i}{(i-n)}.
  \end{equation}
  Since $|a_i|=|b_i|$, the element $f$ exists and it is unique,
    moreover $|f|=|g|$. Hence, $\nabla_2$ is invertible in
    $\Lk{\widehat{\bigoplus}_{i\in\ZZ\setminus\{n\}}k.S^i}$ and as a
    $k$-linear map it is isometric. Therefore, we have
    $\nsp{\nabla_2\-1}=1$. Hence, 
    $\diso{0}{1}\subset\Ak\setminus\Sigma_{\nabla_2}$ by
    Lemma~\ref{sec:spectr-diff-module-2}. Consequently,
    $\diso{0}{1}\cap \Sigma_{d-n}=\{0\}$. As
    $\Sigma_d=\Sigma_{d-n}+n$, we have
      $\diso{n}{1}\cap\Sigma_d=\{n\}$. Therefore, for all $n\in
      \ZZ$ we have
      $(n,x_{0,1})\subset\Ak\setminus \Sigma_d$ and the claim follows. 
  \end{proof}
\paragraph{The case where $x\not\in (0,\infty)$}
  \begin{Pro}
  Let $x\in\Ak$ be a point of type (2), (3) or (4) not of the form
   $x_{0,r}$. Let $c\in k\setminus\{0\}$ such that $x\in
   \diso{c}{|c|}$. The spectrum of $S\d$ as an element of $\Lk{\Hx}$ is
   \begin{equation}\Sigma_{S\d,k}(\Lk{\Hx})=
     \begin{cases}
       \overline{\diso{0}{\frac{|c|}{r_k(x)}}} &  \text{ if } x \text{ is of
         type (4)}\\
       & \\
\disf{0}{\frac{|c|}{r_k(x)}} & \text{otherwise}\\
     \end{cases}.
   \end{equation}
   
\end{Pro}

\begin{proof}
  Let $\Log{c}:\diso{c}{|c|}\to
  \diso{0}{1}$ be the logarithm, we set $y:=\Log{c}(x)$. Since
  $\crk=0$, $\Log{c}$  is an analytic isomorphism and $[\Hx :\Hy]=1$. Therefore, the push-forward of
   $(\Hx,S\d)$ by $\Log{c}$ is isomorphic to $(\Hy,\d)$. Therefore, by
   Propositions~\ref{sec:spectr-diff-equat-2} and
   \ref{sec:push-forw-spectr} we obtain
\begin{equation}\Sigma_{S\d,k}(\Lk{\Hx})=\Sigma_{\d,k}(\Lk{\Hy})=  \begin{cases}
       \overline{\diso{0}{\fra{r_k(y)}}} &  \text{ if } x \text{ is of
         type (4)}\\
       & \\
\disf{0}{\fra{r_k(y)}} & \text{otherwise}\\
\end{cases}.
\end{equation}
Since $r_k(y)=\frac{r_k(x)}{|c|}$, the result follows.
\end{proof}
\subsection{Spectrum of a regular singular differential module}
As mentioned at the beginning of the section, the computation
of the spectrum of a regular differential module follows directly from
the computation of the spectrum done above and
Remark~\ref{sec:spectr-diff-module-3}. In this section,  we will summarize all the
different cases  discussed in the previous section. We will also
discuss the variation of the spectrum.

\begin{nota}
  We denote by $\overline{\ZZ}$ the topological closure of $\ZZ$ in $\Ak$.
\end{nota}

\begin{Theo}\label{sec:spectr-regul-sing-2}
  Assume that $\crk=p>0$. Let $x\in\Ak$ be a point of type (2), (3) or
  (4). Let $(M,\nabla)$ be a regular singular differential module over
  $(\Hx,S\d)$. Let $G$ be the matrix associated to $\nabla$ with constant
  entries (i.e. $G\in\cM_\nu(k)$), and let $\{a_1,\cdots, a_N\}$ be
  the set of eigenvalues of $G$.\\
If $x$ is a point of the form $x_{0,r}$, then we have
 \begin{equation}\Sigma_{\nabla,k}(\Lk{M})=\bigcup_{i=1}^N
   a_i+\overline{\ZZ}.
 \end{equation}
Otherwise, let $c\in k\setminus\{0\}$ such that $x\in
    \diso{c}{|c|}$ and $y:=\Log{c}(x)$. Then we have 

    \begin{equation}\Sigma_{\nabla,k}(\Lk{M})=
      \begin{cases}
        \Bigcup_{j=1}^N\disf{a_j}{\frac{\omega}{r_k(y)}}& \text{ if }
        r_k(x)\in (0, |c|\omega\;]\\
        & \\
        \Bigcup_{j=1}^N\Bigcup_{i=0}^{p^n}\disf{a_j+i}{\frac{\omega}{r_k(y)}}&
        \text{ if } r_k(x)\in (|c|\omega^{\fra{p^{n-1}}},
          |c|\omega^{\fra{p^{n}}}\; ]\\ &\text{ with } n\in\NN\setminus\{0\}.\\
      \end{cases}
    \end{equation}
\end{Theo}

\begin{Theo}\label{sec:spectr-regul-sing-1}
  Assume that $\crk=0$. Let $x\in\Ak$ be a point of type (2), (3) or
  (4). Let $(M,\nabla)$ be a regular singular differential module over
  $(\Hx,S\d)$. Let $G$ be the matrix associated to $\nabla$ with constant
  entries (i.e. $G\in\cM_\nu(k)$), and let $\{a_1,\cdots, a_N\}$ be
  the set of eigenvalues of $G$.
\\
If $x$ is a point of type (2) of the form $x_{0,r}$, then we
    have
\begin{equation}\Sigma_{\nabla,k}(\Lk{M})=\bigcup_{j=1}^N\disf{a_j}{1}.
\end{equation}
If $x$ is a point of type (3) of the form $x_{0,r}$, then we
    have
    \begin{equation}\Sigma_{\nabla,k}(\Lk{M})=\bigcup_{j=1}^N
      a_j+\overline{\ZZ}.
    \end{equation}
Otherwise, let $c\in k\setminus\{0\}$ such that $x\in
    \diso{c}{|c|}$. Then we have
    \begin{equation}\Sigma_{\nabla,k}(\Lk{M})=
     \begin{cases}
       \Bigcup_{j=1}^N\overline{\diso{a_j}{\frac{|c|}{r_k(x)}}} &  \text{ if } x \text{ is of
         type (4),}\\
       & \\
\Bigcup_{j=1}^N\disf{a_j}{\frac{|c|}{r_k(x)}} & \text{otherwise.}\\
     \end{cases}
   \end{equation}
\end{Theo}

\begin{rem}
  Note that from the computation of
  the spectrum of $S\d$, we observe that for a point $x_{c,r}\in
  \Ak\setminus k$ with $c\ne 0$ and $|c|>r$, it is better to choose $(S-c)\d$ than $S\d$.
\end{rem}
\subsubsection{Variation of the spectrum}
Let $(\F,\nabla)$ be a differential equation over
$\Ak\setminus\{0\}$. We fix the derivation $S\d$ over $\Ak$. For each
$x\in \Ak\setminus k$ we set $(M_x,\nabla_x):=(\F_x\ot\Hx,\nabla)$ the
differential module over $(\Hx,S\d)$. We say that $(\F,\nabla)$ is a
differential equation with regular singularities if there exists a matrix $G\in
\cM_n(k)$ such that $G$ is an associated matrix of $(M_x, \nabla_x)$
for each $x$. Note that it is for its own interest to study the
variation of the spectrum of $\nabla_x$.
\begin{nota}
  Let $\cK(\Ak)$ be the set of nonempty compact subsets of
  $\Ak$. We endow $\cK(\Ak)$ with the exponential topology, the
  topology generated by the following family of sets:

\begin{equation} (U,\{U_i\}_{i\in I})=\{\Sigma\in \cK(\Ak);\; \Sigma\subset U,\;
    \Sigma\cap U_i\ne \emptyset\, \forall i\},
  \end{equation}
  where $U$ is an open of $\Ak$ and $\{U_i\}_{i\in I}$ is a finite
  open cover of $U$. In this case, since $\Ak$ is a Hausdorff space,
  then so is for $\cK(\Ak)$.
\end{nota}

\begin{Lem}\label{sec:topology-kct}
  The following function is continuous
\begin{equation}\Fonction{\Upsilon:\cK(\cT)\times
    \cK(\cT)}{\cK(\cT)}{(\Sigma,\Sigma')}{\Sigma\cup\Sigma'}
\end{equation}
\end{Lem}

\begin{proof}
 Let $\Sigma$ and $\Sigma'$ be two non-empty compact subsets of
 $\cT$. Let $(U,\{U_i\}_{i\in I})$ be an open neighbourhood of
 $\Sigma\cup\Sigma'$. We set \begin{equation}J:=\{i\in I|\; \Sigma\cap U_i\ne
 \emptyset\} \text{ and } J':=\{i\in I|\; \Sigma'\cap U_i\ne
 \emptyset\}.\end{equation}
Then $(U,\{U_i\}_{i\in J})$ (resp. $(U,\{U_i\}_{i\in J'})$) is an open
neighbourhood of $\Sigma$ (resp. $\Sigma'$) and we have 
\begin{equation}(U,\{U_i\}_{i\in J})\times (U,\{U_i\}_{i\in J'})\subset \Upsilon\-1(
  (U,\{U_i\}_{i\in I})).\end{equation}
Hence we obtain the result.
\end{proof}
\paragraph{The case of positive residue characteristic}
Assume that $\crk=p>0$. We observe from Theorem~\ref{sec:spectr-regul-sing-2} that, although
the spectrum is 
roughly different from the constant case studied in \cite{Cons}, it satisfies
analogous continuity properties.

\begin{Theo}\label{sec:case-posit-resid-3}
  Let $(\F,\nabla)$ be a differential equation over $\Ak\setminus k$ with regular
  singularities. Let $x\in \Ak\setminus k$. We set:
          \begin{equation}\Fonction{\Psi:
              [x,\infty)}{\cK(\Ak)}{y}{\Sigma_{\nabla_y,k}(\Lk{M_y})}.\end{equation}
         Let $y\in [x,\infty)$, then we have:
          \begin{itemize}
          \item the restriction of
            $\Psi$ to $[x,y]$ is continuous at $y$,
          \item the map $\Psi$ is
            continuous at $y$ if and only if $y$ is of type (3) or of
            the form $x_{0,R}$.
          \end{itemize}
        \end{Theo}

        \begin{proof}
         We identify $[x,\infty)$ with the interval
          $[r(x),\infty)$ by the map $y\mapsto r(y)$ (cf. Definition
          \ref{sec:type-points-a_k}). Let $y\in
          [x, \infty)$. Assume that there exists $x' \in
          [x, \infty)$ such that $[x,x']\cap(0,\infty)=\emptyset$
          and $[x,y]\subset [x,x']$. Let
          $y'\in [x,x']$. By
        Theorem~\ref{sec:spectr-regul-sing-2} and
        Corollary~\ref{sec:case-posit-resid-2}, we have $\Psi(y)=\bigcup_{i=1}^N\disf{a_i}{\phi(y)}$ and
        $\Psi(y') =\bigcup_{i=1}^N\disf{a_i}{\phi(y')}$, where $\phi:[x,y]\to\R+$
        is a decreasing continuous function and $\phi(y)\not\in |k|$
        if $y$ is of type (3). Therefore, the claims:
        \begin{itemize}
        \item $\Psi$ is continuous at $y$ if and only if $y$ is of
          type (3),
        \item the restriction of $\Psi$ to $[x,y]$ is continuous at $y$, 
        \end{itemize}
        holds by the continuity results of \cite[Theorem~5.3]{Cons}.
        
        Now assume that $y\in[x_{0,R},\infty)$. In the case where $y\ne x_{0,R}$,  the
        restriction of $\Psi$ to $[x_{0,R}, \infty)$ is constant
        (cf. Theorem~\ref{sec:spectr-regul-sing-2}). Hence the
        restriction of $\Psi$ to $[x_{0,R}, \infty)$ is
        continuous. Otherwise, i.e. $y=x_{0,R}$, on the one hand the restriction of
        $\Psi$ to $[y,\infty)$ is continuous at $y$. On the
        other hand, since for all $y'\in [x,y]$ we
        have
        $\Psi(y')=\bigcup_{i=1}^N\alpha_i+\Sigma_{S\d,k}(\Lk{\h{y'}})$,
        it is enough to show the result for the differential module
        $(\hh{y},S\d)$ (cf. Lemma~\ref{sec:topology-kct}). Hence, we reduce to the case where
        $\Psi(y)=\ZZ_p$ and
        $\Psi(y')=\bigcup_{i\in\NN}\disf{i}{\phi(y')}$, with $\phi:
        [x,y)\to \R+$ a
         decreasing continuous function and $\lim\limits_{y'\to
          y}\phi(y')=0$ (cf. Theorem~\ref{sec:spectr-regul-sing-2} and
        Corollary~\ref{sec:case-posit-resid-2}). Let $(U,\{U_i\}_{i\in
          I})$ be an open neighbourhood of $\Psi(y)$. Since $\Psi(y)$
        is a set of points of type (1), we can assume that $U_i$
        is an open disk for all $i\in I$. Since $\ZZ_p$ is the topological closure of
        $\NN$ in $\Ak$, for all $i\in I$ we have $\NN\cap
        U_i\ne\emptyset$. Therefore, for all $i\in I$ we have
        $\Psi(y')\cap U_i\ne\emptyset$. We now prove that there exists
        $x'\in[x,y)$ such that for all $y'\in (x',y)$ we have
        $\Psi(y')\subset U$. Let $L$ be the smallest radius of the disks
        $U_i$. Since $\phi$ is a decreasing continuous function, there exists
        $y_L$ such that for all $y'\in(y_L,y)$ we have
        $\phi(y')<L$. Therefore, since $\NN\subset U=\bigcup_{i\in
          I}U_i$, for all $j\in\NN$, there exists $i\in
        I$ such that $\disf{j}{\phi(y')}\subset U_i$. Consequently, we have
        $\Psi(y')\subset U$ and $\Psi(y')\in(U,\{U_i\}_{i\in I})$.      
        \end{proof}

\paragraph{The case of residue characteristic zero}
Assume that $\crk=0$. We observe from Theorem~\ref{sec:spectr-regul-sing-1} that the
spectrum behaves differently from the case where
$\crk=p>0$. In the special case where $k$ is not trivially valued and $|k|\ne\R+$, the map
\begin{equation}\Fonction{\Psi:(0,\infty)}{\cK(\Ak)}{y}{\Sigma_{\nabla_y,k}(\Lk{M_y})}\end{equation}
is not continuous at all. Indeed, let $y\in (0,\infty)$ be a point of
type (2). Assume that $(\F,\nabla)=(\cO_{\Ak\setminus\{0\}},d_{\Ak\setminus\{0\}/k})$. Then we have $\Psi(y)=\disf{0}{1}$. Let $U$ be an open
neighbourhood of $\Psi(y)$ in $\Ak$. Let $a\in(\disf{0}{1}\cap
k)\setminus \ZZ$  and let $0<r<1$ such that $\diso{a}{r}\cap\ZZ=\emptyset$. For any $y'\in (0,\infty)$ of type (3) we have
$\Psi(y')=\ZZ\cup\{x_{0,1}\}$, hence $\Psi(y')\cap
\diso{a}{r}=\emptyset$. Therefore, $\Psi(y')\not\in
(U,\{U,\disf{a}{r}\})$. 

If $k$ is trivially valued, the only point where there
is no continuity is $x_{0,1}$. For the other points of $(0,\infty)$,
since $\Psi$ is constant, it is continuous on
$(0,\infty)\setminus\{x_{0,1}\}$.

For branches $(c,x_{0,|c|}\;]$ with $c\in k\setminus\{0\}$, the map

\begin{equation}\Fonction{\Psi:(c,x_{0,|c|}\;]}{\cK(\Ak)}{y}{\Sigma_{\nabla_y,k}(\Lk{M_y})}\end{equation}
satisfies the same continuity properties as those of
\cite[Theorem~5.3]{Cons}. Indeed, for any $y\in (c,x_{c,|c|}]$
we have $\Psi(y)=\bigcup_{i=1}^N\disf{a_i}{\phi(y)}$ with $\phi:
(c,x_{0,|c|}]\to \R+$ a decreasing continuous function and
$\phi(y)\not\in |k|$ if $y$ is of type (3).

We have the following results:

\begin{Theo}
 Assume that $|k|=\R+$. Let $(\F,\nabla)$ be a differential equation over $\Ak\setminus k$ with regular
  singularities. Let $x\in \Ak\setminus k$. We set:
          \begin{equation}\Fonction{\Psi:
              [x,\infty)}{\cK(\Ak)}{y}{\Sigma_{\nabla_y,k}(\Lk{M_y})}.\end{equation}
         Let $y\in [x,\infty)$, then we have:
          \begin{itemize}
          \item the restriction of
            $\Psi$ to $[x,y]$ is continuous at $y$,

            \item the map $\Psi$ is
            continuous at $y$ if and only if $y$ is of type (4) or of
            the form $x_{0,R}$.
              
          \end{itemize}
        \end{Theo}

        \begin{proof}
          The proof is analogous to the proof of
          Theorem~\ref{sec:case-posit-resid-3}.
        \end{proof}

\section{Spectral version of Young's theorem}\label{sec:spectr-vers-youngs-10}
        In this part we give a spectral version of
        Young's theorem \cite{Young}, \cite[Theorem~6.5.3]{Ked},
        \cite[Theorem~6.2]{CM02}, which states the following.
\begin{hyp} We still assume that $k$ is algebraically closed.
\end{hyp}
        \begin{Theo}[Young]
          Let $x\in \Ak$ be a point of type (2), (3) or (4). Let
          $\cL=\sum_{i=0}^ng_{n-i}\d^i$ with $g_0=1$ and $g_i\in\Hx$,
          and let $(M,\nabla)$ be the associated differential module
          over $(\Hx, \d)$. We set $|\cL|_{\mr{Sp}}=\Max_{0\leq i\leq
            n}|g_i|^{\fra{i}}$. If $|\cL|_{\mr{Sp}}>\nor{\d}$ then $\nsp{\nabla}=|\cL|_{\mr{Sp}}$.
        \end{Theo}

In order to state and prove the main statement of the section, we will
need the following additional results.

\begin{Defi}
      Let $E$ be a $k$-Banach algebra and $B$ a commutative
     $k$-subalgebra of $E$. We say that $B$ is a maximal commutative subalgebra of
     $E$, if for any commutative $k$-subalgebra $B'$ of $E$ we have the following
     property:
     \begin{equation}(B\subset B'\subset E)\Leftrightarrow(B'=B).\end{equation} 
   \end{Defi}

   \begin{rem}
     A maximal subalgebra $B$ is necessarily closed in $E$, hence a
     $k$-Banach algebra.
   \end{rem}
   
   \begin{Pro}[{\cite[Proposition~7.2.4]{Ber}}]\label{sec:berk-spectr-theory-5}
     Let $E$ be a $k$-Banach algebra. For any
     maximal commutative subalgebra $B$ of $E$, we have:
     \begin{equation}\forall f\in B,\qqq \Sigma_f(B)=\Sigma_f(E).\end{equation}
   \end{Pro}

   \begin{Defi}
     Let $E$ be a $k$-algebra and let $B$ be a $k$-subalgebra of
     $E$. If any element of $B$ invertible in $E$ is also invertible in
     $B$, we say that $B$ is a {\em saturated} subalgebra of $E$. 
   \end{Defi}

   \begin{Pro}[{\cite[Proposition~7.2.4]{Ber}}]\label{sec:defin-first-prop-1}
     Assume that $k$ is not trivially valued. Let $E$ be a $k$-Banach
     algebra and let $B$ be a saturated $k$-Banach subalgebra of
     $E$. Then we have:
     \begin{equation}\forall f\in B,\qqq \Sigma_f(B)=\Sigma_f(E).\end{equation}
   \end{Pro}

\begin{Lem}
  Let $\Omega\in E(k)$, let $E$ be an $\Omega$-Banach algebra and
  $f\in E$. Then we have
  \begin{equation}\Sigma_{f,k}(E)=\pik{\Omega}(\Sigma_{f,\Omega}(E)),\end{equation}
  where $\pik{\Omega}:\A{\Omega}\to \Ak$ is the canonical projection. 
  
\end{Lem}

\begin{proof}
  Let $B$ be a maximal commutative $\Omega$-subalgebra of $E$
  containing $f$. Let $B'$ be a commutative $k$-subalgebra of $E$ such
  that $B\subset B'$. Then $B'$ is also an $\Omega$-subalgebra of $E$.
  Therefore $B$ is also  maximal
  as a commutative $k$-subalgebra of $E$. Let $x\in\Ak$, we have an
  isometric isomorphism $B\ct_k\Hx\simeq B\ct_\Omega(\Omega\ct_k\Hx)$
  (cf. \cite[Section~2.1, Proposition 7]{Bosc}). For each $y\in
  \pik{\Omega}\-1(x)$, we have a contracting map $\Omega\ct_k\Hx\to
  \Hy$. Therefore, the induced map $B\ct_k\Hx\to B\ct_\Omega\Hy$ is
  contracting too.  Hence, if $f\ot 1-1\ot T(y)$ is not invertible in
  $B\ct_\Omega\Hy$, then $f\ot 1-1\ot T(x)$ is not invertible too in
  $B\ct_k \Hx$. Therefore,
  $\pik{\Omega}(\Sigma_{f,\Omega}(B))\subset\Sigma_{f,k}(B)$. Now let 
  $x\in \Sigma_{f,k}(B)$. 
  Since $f\ot 1-1\ot T(x)$ is not invertible in $B\ct_k\Hx\simeq
  B\ct_\Omega(\Omega\ct_k\Hx)$, according to
  Lemma~\ref{sec:topology-cma-1}, there exists $y\in
  \cM(\Omega\ct_k\Hx)=\pik{\Omega}\-1(x)$ such that $f\ot 1-1\ot T(y)$
  is not invertible in $B\ct_\Omega\Hy$. Therefore,
  $\Sigma_{f,k}(B)\subset\pik{\Omega}(\Sigma_{f,\Omega}(B))$. Hence, by
  Proposition~\ref{sec:berk-spectr-theory-5} we obtain
   \begin{equation}\Sigma_{f,k}(E)=\pik{\Omega}(\Sigma_{f,\Omega}(E)). \qedhere \end{equation}
 \end{proof}

 \begin{Defi}
   Let $\Omega\in E(k)$ and $f\in \cM_n(\Omega)$. Let $\{a_1,\cdots
   a_N\}$ be the set of eigenvalues of $f$ in
   $\wac$. Let us call $\piro{\wac}{\Omega}(\{a_1,\cdots, a_N\})$
   the set of eigenvalues of $f$ in $\A{\Omega}$.
 \end{Defi}

 \begin{cor}\label{sec:defin-first-prop-2}
   Let $\Omega\in E(k)$ and $f\in \cM_n(\Omega)$. Let
   $\{a_1,\cdots,a_N\}$ be the set of eigenvalues of $f$ in $\A{\Omega}$. Then we have
\begin{equation}\Sigma_{f,k}(\cM_n(\Omega))=\pik{\Omega}(\{a_1,\cdots, a_N\}).\end{equation}
\end{cor}

 \begin{cor}\label{sec:spectr-vers-youngs-3}
   Suppose that $k$ is not trivially valued. Let $\Omega\in E(k)$ and $f\in \cM_n(\Omega)$. Let
   $\{a_1,\cdots,a_N\}$ be the set of rigid points of $\A{\Omega}$
that correspond to the eigenvalues
   of $f$ in some finite extensions of $\Omega$. Then we have
\begin{equation}\Sigma_{f,k}(\Lk{\Omega^n})=\pik{\Omega}(\{a_1,\cdots, a_N\}).\end{equation}
 \end{cor}
 \begin{proof}
   Since $\cM_n(\Omega)$ is a saturated subalgebra of $\Lk{\Omega^n}$,
   according to Proposition~\ref{sec:defin-first-prop-1}, we have
   $\Sigma_{f,k}(\cM_n(\Omega))=\Sigma_{f,k}(\Lk{\Omega^n})$. The
   result follows by
   Corollary~\ref{sec:defin-first-prop-2}.
 \end{proof}
 \begin{rem}
 In the case where $k$ is trivially valued, we have at least the
 inclusion $\Sigma_{f,k}(\Lk{\Omega^n})\subset\pik{\Omega}(\{a_1,\cdots, a_N\}) $. 
 \end{rem}

 \begin{Lem}\label{sec:spectr-vers-youngs-5}
   Let $(A,|.|)$ be a commutative $k$-Banach algebra. Let
   $P(T)=\sum_{i=0}^n a_i T^i\in A[T]$ with $a_n=1$. Let $G\in\cM_n(A)$ such that

   \begin{equation}G=\left(
     \raisebox{0.5\depth}{%
       \xymatrixcolsep{1ex}%
       \xymatrixrowsep{1ex}%
       \xymatrix{0\ar@{.}[rrr]& & & 0&-a_0\ar@{.}[dddd] \\
        1\ar@{.}[rrrddd]& 0\ar@{.}[rr]\ar@{.}[rrdd]& &0\ar@{.}[dd]& \\
        0\ar@{.}[dd]\ar@{.}[rrdd]& &  &  &\\
        & &  &0& \\
        0\ar@{.}[rr]& &0&1& -a_{n-1}\\ }%
        }
       \right)
       .\end{equation}
    If $a_0$ is invertible in $A$, then $G$ is invertible and we have
    $\nor{G\-1}=\max\limits_{0\leq i\leq n}|a_ia_0\-1|$.
     
   \end{Lem}
   \begin{proof}
   Since $\det(G)=(-1)^na_0$ then $G$ is invertible if and only if
   $a_0$ is invertible in $A$. Assume now that $a_0$ is
   invertible. Then we have
    \begin{equation}G\-1=\fra{a_0}\left(
     \raisebox{0.5\depth}{%
       \xymatrixcolsep{1ex}%
       \xymatrixrowsep{1ex}%
       \xymatrix{-a_1\ar@{.}[ddd]&a_0\ar@{.}[rrrddd]& 0\ar@{.}[rr]\ar@{.}[rrdd]&&0\ar@{.}[dd]\\
         &0\ar@{.}[rrrddd]\ar@{.}[ddd]& &&\\
         & &  &&0\\
         -a_{n-1}&& & & a_0\\
         -1& 0\ar@{.}[rrr] &&&0
         \ }%
        }
       \right)
       .\end{equation}
     Hence, $\nor{G\-1}=\max\limits_{0\leq i\leq n}|a_ia_0\-1|$.
   \end{proof}
   \begin{Lem}[{\cite[Proposition~2.1.8]{Bosc}}]\label{sec:spectr-vers-youngs-4}
     Let $\Omega,\Omega'\in E(k)$. Then
     $$(\cM_n(\Omega\ct_k\Omega'),\nor{.})\simeq
     (\cM_n(\Omega)\ct_k\Omega',\nor{.}'),$$ where $\nor{.}$ is the
     maximum norm and $\nor{.}'$ is the tensor product norm.
   \end{Lem}

   \begin{Theo}[Spectral version of Young's theorem (weak version)]\label{sec:spectr-vers-youngs-6}
   Let $\Omega\in E(k)$ and $d:\Omega\to \Omega$ be a $k$-linear bounded
    derivation. Let $(M,\nabla)$ be a differential module over
    $(\Omega,d)$ with $(M,\nabla)\simeq (\DD_\Omega/\DD_\Omega\cdot
    P(d), d)$ and $P(d)=\sum_{i=0}^{n}a_i d^i$ with $a_n=1$. Let $\{z_1,\cdots,
    z_n\}\subset \Omega^{alg}$ be the multiset of roots of $P(T)$ (the
    commutative polynomial associated to $P(d)$). If $ \min_i
    r_k(\pik{\wac}(z_i))>\nor{d}$, then
    \begin{equation}\Sigma_{\nabla,k}(\Lk{M})\subset
      \pik{\wac}(\{z_1,\cdots,z_n\}).\end{equation}
    In particular if
    $\pik{\wac}(\{z_1,\cdots,z_n\})=\{z\}$, we have $\Sigma_{\nabla,k}(\Lk{M})=\{z\}$.
  \end{Theo}

  \begin{proof}
     Since $(M,\nabla)\simeq (\DD_\Omega/\DD_\Omega\cdot
    P(d), d)$, there exists a cyclic basis $\{m,\nabla(m),\cdots,
    \nabla^{n-1}(m)\}$ such that:
     \begin{equation}\nabla \left(
     \raisebox{0.5\depth}{%
       \xymatrixcolsep{1ex}%
       \xymatrixrowsep{1ex}%
       \xymatrix{f_0\ar@{.}[ddddd]\\\\
        \\
 \\
\\
    f_{n-1}\\ }%
        }
       \right)=\left(
     \raisebox{0.5\depth}{%
       \xymatrixcolsep{1ex}%
       \xymatrixrowsep{1ex}%
       \xymatrix{df_0\ar@{.}[ddddd]\\\\ 
        \\
 \\
\\
    df_{n-1}\\ }%
        }
       \right)+\left(
     \raisebox{0.5\depth}{%
       \xymatrixcolsep{1ex}%
       \xymatrixrowsep{1ex}%
       \xymatrix{0\ar@{.}[rrr]& & & 0&-a_0\ar@{.}[dddd] \\
        1\ar@{.}[rrrddd]& 0\ar@{.}[rr]\ar@{.}[rrdd]& &0\ar@{.}[dd]& \\
        0\ar@{.}[dd]\ar@{.}[rrdd]& &  &  &\\
        & &  &0& \\
        0\ar@{.}[rr]& &0&1& -a_{n-1}\\ }%
        }
       \right) \left(
     \raisebox{0.5\depth}{%
       \xymatrixcolsep{1ex}%
       \xymatrixrowsep{1ex}%
       \xymatrix{f_0\ar@{.}[ddddd]\\\\
        \\
 \\
\\
    f_{n-1}\\ }%
        }
       \right).\end{equation}
     As first step we will assume that
     $k$ is not trivially valued. Then there exists $\alpha\in k$ such
     that $|\alpha|\cdot \min_i r_k(\pik{\wac}(z_i))\geq 1$ and $\nor{\alpha d}<1$. In
     order to prove the statement, it is enough to show that
     $\Sigma_{\alpha\nabla,k}(\Lk{M})\subset \pik{\wac}(\{\alpha
     z_1,\cdots,\alpha z_n\})$. In
     the basis $\{m,\alpha \nabla(m),\cdots,
     \alpha^{n-1}\nabla^{n-1}(m)\}$ we have:
     \begin{equation}\alpha\nabla \left(
     \raisebox{0.5\depth}{%
       \xymatrixcolsep{1ex}%
       \xymatrixrowsep{1ex}%
       \xymatrix{f_0\ar@{.}[ddddd]\\
         \\
        \\
 \\
\\
    f_{n-1}\\ }%
        }
       \right)=\left(
     \raisebox{0.5\depth}{%
       \xymatrixcolsep{1ex}%
       \xymatrixrowsep{1ex}%
       \xymatrix{\alpha df_0\ar@{.}[ddddd]\\
         \\
        \\
 \\
\\
    \alpha df_{n-1}\\ }%
        }
      \right)+
      \left(
     \raisebox{0.5\depth}{%
       \xymatrixcolsep{1ex}%
       \xymatrixrowsep{1ex}%
       \xymatrix{0\ar@{.}[rrr]& & & 0&-a_0\alpha^n\ar@{.}[dddd] \\
        1\ar@{.}[rrrddd]& 0\ar@{.}[rr]\ar@{.}[rrdd]& &0\ar@{.}[dd]& \\
        0\ar@{.}[dd]\ar@{.}[rrdd]& &  &  &\\
        & &  &0& \\
        0\ar@{.}[rr]& &0&1& -a_{n-1}\alpha\\ }%
        }
       \right)
      \left(
     \raisebox{0.5\depth}{%
       \xymatrixcolsep{1ex}%
       \xymatrixrowsep{1ex}%
       \xymatrix{f_0\ar@{.}[ddddd]\\
         \\
        \\
 \\
\\
    f_{n-1}\\ }%
        }
      \right).\end{equation}
    We set
    \begin{equation}G:= \left(
     \raisebox{0.5\depth}{%
       \xymatrixcolsep{1ex}%
       \xymatrixrowsep{1ex}%
       \xymatrix{0\ar@{.}[rrr]& & & 0&-b_0\ar@{.}[dddd] \\
        1\ar@{.}[rrrddd]& 0\ar@{.}[rr]\ar@{.}[rrdd]& &0\ar@{.}[dd]& \\
        0\ar@{.}[dd]\ar@{.}[rrdd]& &  &  &\\
        & &  &0& \\
        0\ar@{.}[rr]& &0&1& -b_{n-1}\\ }%
        }
      \right),\end{equation}
     with $b_i=a_i\alpha^{n-i}$. We set $w_i:=\alpha z_i$ and
     $\Delta:=\alpha \nabla-G$. Then on the one
     hand, it is easy to see that $\{w_1,\cdots
    w_n\}$ are the eigenvalues of $G$ in $\A{\Omega}$. On the other hand,
    we have $\nor{\Delta}=\nor{\alpha d}<1$ and $\min_i
    r_k(\pik{\wac}(w_i))\geq 1$. By Corollary~\ref{sec:spectr-vers-youngs-3}, $G\ot
     1-1\ot T(y)$ is invertible for
     all $y\in \Ak\setminus \pik{\wac}(\{w_1,\cdots,w_n\})$.
     
     Let $y\in \Ak\setminus \pik{\wac}(\{w_1,\cdots,w_n\})$, now we show that $\alpha \nabla \ot 1- 1\ot T(y)$ is invertible.
     Since we have  $\cM_n(\Omega\ct_k\Hy)\simeq
     \cM_n(\Omega)\ct_k\Hy$
     (cf. Lemme~\ref{sec:spectr-vers-youngs-4}), then we can write
     \begin{equation}G\ot 1-1\ot T(y)=\left(
     \raisebox{0.5\depth}{%
       \xymatrixcolsep{1ex}%
       \xymatrixrowsep{1ex}%
       \xymatrix{-1\ot T(y)\ar@{.}[rrr]& & & 0&-b_0\ot 1\ar@{.}[dddd] \\
        1\ar@{.}[rrrddd]& -1\ot T(y) \ar@{.}[rr]\ar@{.}[rrdd]& &0\ar@{.}[dd]& \\
        0\ar@{.}[dd]\ar@{.}[rrdd]& &  &  &\\
        & &  &-1\ot T(y)& \\
        0\ar@{.}[rr]& &0&1& -b_{n-1}\ot 1-1\ot T(y)\\\ }%
        }
      \right).\end{equation}
Let $Q_y(T):=\sum_i b_i^yT^i+T^n\in \Omega\ct_k\Hy[T]$ such that
$Q_y(T)=P(T+1\ot T(y))$. For 

\begin{equation}U=\left(
     \raisebox{0.5\depth}{%
       \xymatrixcolsep{1ex}%
       \xymatrixrowsep{1ex}%
       \xymatrix{1& 1\ot (-T(y))& 1\ot T(y)^2 \ar@{.}[rr]& &1\ot (-T(y))^{n-1}\\
         0 \ar@{.}[rrrddd]\ar@{.}[ddd]&1 \ar@{.}[rrrddd]&2\ot (-T(y)) \ar@{.}[rr]\ar@{.}[rrdd]  & &\binom{n-1}{1}\ot(-T(y))^{n-2}\ar@{.}[dd]\\
         && & & \\
         & &  &&\binom{n-1}{n-2}\ot(-T(y))\\
         0\ar@{.}[rrr]&&&0&1\\
         \ }%
        }
      \right),\end{equation}
    then, it is easy to see that $U$ is invertible in $\cM_n(\Omega\ct_k\Hy)$ and we have:
    \begin{equation}G_y:=U^{-1}(G\ot 1-1\ot T(y))U=\left(
     \raisebox{0.5\depth}{%
       \xymatrixcolsep{1ex}%
       \xymatrixrowsep{1ex}%
       \xymatrix{0\ar@{.}[rrr]& & & 0&-b_0^y\ar@{.}[dddd] \\
        1\ar@{.}[rrrddd]& 0\ar@{.}[rr]\ar@{.}[rrdd]& &0\ar@{.}[dd]& \\
        0\ar@{.}[dd]\ar@{.}[rrdd]& &  &  &\\
        & &  &0& \\
        0\ar@{.}[rr]& &0&1& -b_{n-1}^y\\ }%
        }
      \right).\end{equation}
    Since $U\in \cM_n(k\ct_k\Hy)$, we have $U^{-1}\Delta
    U=U^{-1}U\Delta=\Delta$. Then $U^{-1}(\alpha\nabla\ot 1-1\ot
    T(y))U=\Delta+G^y$. Consequently, in order to proof that
    $\alpha\nabla\ot 1-1\ot T(y)$ is invertible, since $\nor{\Delta}<1$, it is enough to show
    that
    \begin{equation}1\leq \nor{G_y^{-1}}^{-1}.\end{equation}
Since $k$ is algebraically closed, $\Omega\ct_k\Hy$ is a multiplicative
$k$-algebra (cf. Proposition~\ref{sec:univ-points-fiber-2}). By
Lemma~\ref{sec:spectr-vers-youngs-5}, we have
$\nor{G_y\-1}=|b_0^y|\-1\nor{G_y}$. Now we show that
$\nor{G_y\-1}=1$, i.e. $\nor{G_y}=|b_0|$. 
Since $\Omega\hookrightarrow \wac$ is an isometry, so is
$\Omega\ct_k\Hy\hookrightarrow \wac\ct_k\Hy$
(cf. \cite[Lemme~3.1]{poi}). We have $Q_y(T)=P(T+1\ot
T(y))=\prod_{i=0}^n(T-(w_i\ot 1-1\ot T(y)))$, hence
\begin{equation}b_{n-i}^y=\sum_j \prod_{l=1}^i(w_{j_l}\ot 1-1\ot T(y)).\end{equation}
Therefore, if we show that $|w_i\ot 1-1\ot T(y)|\geq 1$ for all $i$,
then we have automatically

\begin{equation}\nor{G_y}=\max_{1\leq i\leq n}|b_i^y|=|b_0^y|.\end{equation}

Now let us show that $|w_i\ot 1-1\ot T(y)|\geq 1$ for all $i$. To avoid
any confusion, we fix another coordinate function $S$ on $\Ak$. Note that
we have an isometric embedding $\h{\pik{\wac}(w_i)}\hookrightarrow
\wac$, that assigns $w_i$ to $S(\pik{\wac}(w_i))$. By above
argument, we have an isometric embedding of $k$-algebra $\h{\pik{\wac}(w_i)}\ct_k\Hy\hookrightarrow
\wac\ct_k \Hy$, and we have $|w_i\ot 1-1\ot T(y)|=|S(\pik{\wac}(w_i))\ot
1-1\ot T(y)|$.
The natural map $\h{\pik{\wac}(w_i)}\ct_k\Hy\to
\h{\uni{\Hy}(\pik{\wac}(w_i))}$ is an isometry (see
Definition~\ref{sec:univ-points-fiber-4}), mapping $S(\pik{\wac}(w_i))\ot
1-1\ot T(y)$ to $S(\uni{\Hy}(\pik{\wac}(w_i)))-T(y)$. By
Lemma~\ref{sec:berkovich-line}, we have
$r_{\Hy}(\uni{\Hy}(\pik{\wac}(w_i)))=r_k(\pik{\wac}(w_i))$. Consequently,
we obtain 

\begin{equation}1\leq r_k(\pik{\wac}(w_i))\leq |S(\uni{\Hy}(\pik{\wac}(w_i)))-T(y)|=|w_i\ot 1-1\ot T(y)|.\end{equation}
Hence, we obtain $\nor{G_y\-1}=1$ and conclude that $\alpha \nabla\ot 1-1\ot
T(y)$ is invertible for all $y\in \Ak\setminus\{w_1,\cdots, w_n\}$.

Now we assume that $k$ is trivially valued. Let $k'\in E(k)$ be
algebraically closed, such that
there exists $\alpha\in k'$, with $|\alpha|\min_ir_k(z_i)\geq 1$ and
$|\alpha|\nor{d}<1$. We know that $\Sigma_{\nabla\ot 1,k'}(\Lk{M}\ct_k
k')=\pik{k'}\-1(\Sigma_{\nabla,k}(\Lk{M})$
(cf. \cite[Proposition~7.1.6]{Ber}). Therefore, in order to show that
$\Sigma_{\nabla,k}(\Lk{M})\subset\pik{\wac}(\{z_1,\cdots, z_n\})$, it
is enough to show that $\Sigma_{\nabla\ot 1,k'}(\Lk{M}\ct_k
k')\subset \pik{k'}\-1(\pik{\wac}(\{z_1,\cdots, z_n\}))$. For that we
will consider $\nabla\ot \alpha$ and show, as for
the non trivial case, that $\Sigma_{\nabla\ot \alpha,k'}(\Lk{M}\ct_k
k')\subset \alpha\pik{k'}\-1(\pik{\wac}(\{z_1,\cdots, z_n\}))$.

If in addition we assume that  $\pik{\wac}(\{z_1,\cdots,z_n\})=\{z\}$,
then since the spectrum is not empty \cite[Theorem~7.1.2]{Ber} we must
have $\Sigma_{\nabla,k}(\Lk{M})=\{z\}$.
\end{proof}

\begin{rem}\label{sec:spectr-vers-youngs-11}
Set the notations as in Theorem~\ref{sec:spectr-vers-youngs-6}. If $P(T)$ is irreducible as a commutative polynomial,
then it is easy to see that $|z_1-a|=\cdots=|z_n-a|$ for all $a\in
k$. Hence $\Sigma_{\nabla,k}(\Lk{M})=\pik{\wac}(\{z_1,\cdots, z_n\}=\{z\}$.   
\end{rem}

We need the
following results to prove that the same result (as in Remark~\ref{sec:spectr-vers-youngs-11}) hold for a monic irreducible
differential polynomial $P(d)$.

\begin{Lem}[{\cite[Theorem~6.4.4]{Ked}}]\label{sec:spectr-vers-youngs}
  Let $\Omega\in E(k)$ and $d:\Omega\to \Omega$ be a $k$-linear bounded
  derivation. Let $P(d)\in \DD_\Omega$ be a monic differential
  polynomial. Let $\{z_1,\cdots, z_n\}$ be the
  multiset\footnote{Counted with multiplicity.} of roots of $P(T)$
  (the commutative polynomial associated to $P(d)$), with $|z_1|\leq \cdots \leq|z_n|$. Let $r>\R+^*$. If $\nor{d}<r$ and for some
  $i_0$ we have $|z_{i_0}|<r<|z_{i_0+1}|$, then there exists a unique factorization $P(d)=Q(d)R(d)$ such that $\{\omega_1,\cdots,\omega_{i_0}\}$
  (resp. $\{\omega_{i_0+1},\cdots,\omega_{n}\}$) is the multiset of roots of
  $R(T)$ (resp. $Q(T)$) with $|\omega_i|<r$
  (resp. 
  $|\omega_i|=|z_i|$) for $i\leq i_0$ (resp. $i>i_0$), where $R(T)$
  (resp. $Q(T)$) is the commutative polynomial associated to $R(d)$
  (resp. $Q(d)$). If moreover we
  have $|z_i|>\nor{d}$ for each $i$ then $|\omega_i|=|z_i|$ for each $i$.
\end{Lem}

\begin{Lem}
  Let $\Omega\in E(k)$ and $d:\Omega\to \Omega$ be a $k$-linear bounded
  derivation. Let $P(d)\in \DD_\Omega$ be a monic irreducible differential
  polynomial. Let $\{z_1,\cdots, z_n\}$ be the multiset of roots of $P(T)$ (the
  commutative polynomial associated to $P(d)$). If  $\Min_i r_k(\pik{\wac}(z_i))>\nor{d}$,
  then $\pik{\wac}(z_1,\cdots, z_n)=\{z\}$.
\end{Lem}

\begin{proof}
  Since $P(d)$ is irreducible, than so is for $P_a(d)=P(d+a)$ for all
  $a\in k$. Let $a\in k$, we have $P_a(T)=P(T+a)$, hence the roots of $P_a$ as a
  commutative polynomial are exactly $\{z_1-a,\cdots,z_n-a\}$.
  Since
  \begin{equation}\forall i,\quad |z_i-a|\geq\Min_i r_k(\pik{\wac}(z_i))>\nor{d},\end{equation}
  then by Lemme~\ref{sec:spectr-vers-youngs} we have
  \begin{equation}|z_1-a|=\cdots=|z_n-a|.\end{equation}
  Consequently, since $k$ is algebraically closed, we have
  \begin{equation}\pik{\wac}(z_1)=\cdots=\pik{\wac}(z_n)=\{z\}\end{equation}
  for some $z\in \Ak\setminus k$.
\end{proof}

\begin{cor}
  Let $\Omega\in E(k)$ and $d:\Omega\to \Omega$ be a $k$-linear bounded
    derivation. Let $(M,\nabla)$ be a differential module over
    $(\Omega,d)$ with $(M,\nabla)\simeq (\DD_\Omega/\DD_\Omega\cdot
    P(d), d)$ and $P(d)=\sum_ia_i d^i+d^n$. We assume that $P(d)$ is
    an irreducible differential polynomial. Let $\{z_1,\cdots,
    z_n\}\subset \Omega^{alg}$ be the multiset
    of roots of $P(T)$ (the commutative polynomial associated to $P(T)$). If $ \min_i
    r_k(\pik{\wac}(z_i))>\nor{d}$, then
    \begin{equation}\Sigma_{\nabla,k}(\Lk{M})=\pik{\wac}(\{z_1,\cdots,z_n\})=\{z\}.\end{equation} 
  \end{cor}

  The first spectral version of Young's theorem can be
  refined as follows. Before, we need the following
  results. 
  \begin{Lem}\label{sec:spectr-vers-youngs-8}
    Let $\{\omega_1,\cdots,\omega_n\}\subset\Ak$ and assume that
    $r_k(\omega_1)\leq\cdots\leq r_k(\omega_n)$. Then for
    each $\omega_i$ there exists an $a_i\in k$ such that for all $j>i$ we have $|T(\omega_j)-a_i|>|T(\omega_i)-a_i|$. 
  \end{Lem}

  \begin{proof}
    If $\omega_i$ is not a point of type (4), then there
    exists $a_i$ such that $r_k(\omega_i)=|T(\omega_i)-a_i|$, and if
    $j>i$, then we have $r_k(\omega_i)\leq r(\omega_j)\leq
    |T(\omega_j)-a_i|$. If $r_k(\omega_i)< r_k(\omega_j)$ then it is
    clear that  $|T(\omega_j)-a_i|>|T(\omega_i)-a_i|$. If
    $r_k(\omega_i)=r_k(\omega_j)$, then we must have
    $r_k(\omega_j)<|T(\omega_j)-a_i|$, otherwise we get
    $\omega_i=\omega_j$ which contradicts the hypothesis. Hence we
    have $|T(\omega_j)-a_i|>|T(\omega_i)-a_i|$.
    Now if $\omega_i$ is a point of type (4), then we choose $a_i$ such
    that $\omega_i\in\disf{a_i}{r_k(\omega_i)+\epsilon}$ with
    $\epsilon>0$ and for all $j\ne i$ we have $\omega_j\notin
    \disf{a_i}{r_k(\omega_i)+\epsilon}$. Then we have
    \begin{equation}|T(\omega_i)-a_i|\leq r_k(\omega_i)+\epsilon <
      |T(\omega_j)-a_i|,\end{equation}
    which prove the statement. 
  \end{proof}

  \begin{Theo}\label{sec:spectr-vers-youngs-7}
  Let $\Omega\in E(k)$ and $d:\Omega\to \Omega$ be a $k$-linear bounded
  derivation. Let $P(d)\in \DD_\Omega$ be a monic differential
  polynomial. Let $\{z_1,\cdots, z_n\}$ be the multiset of the roots
  of $P(T)$ (the commutative polynomial associated to $P(d)$). Assume that $\Min_ir_k(\pik{\wac}(z_i))>\nor{d}$. Let $\{\omega_1,\cdots,\omega_\nu\}:=\pik{\wac}(\{z_1,\cdots,
  z_n\})$ with $\omega_i\ne \omega_j$ if $i\ne j$ and 
    $r_k(\omega_1)\leq\cdots\leq r_{k}(\omega_\nu)$. Then for each
  $\omega_i$ there exists a commutative monic polynomial $P_{\omega_i}(T)$ such that the
  projection of its roots in $\Ak$ are equal to $\omega_i$, and we
  have $P(d)=P_{\omega_{\nu}}(d)\cdots P_{\omega_{1}}(d)$, where
  $P_{\omega_i}(d)$ are the differential polynomials associated to $P_{\omega_i}(T)$.
  \end{Theo}
  \begin{proof}
    For each differential polynomial $Q(d)$, we denote by $\Sigma_Q$
    the spectrum of the differential module
    $(\DD_\Omega/\DD_\Omega.Q(d),d)$. Let $S$ be the multiset of roots of $P(T)$,
    in particular we have $\car(S)=\Deg(P)$, and let
    $S(\omega_i):=\{z_{i_1},\cdots,z_{i_{n_i}}\}$ be the multiset of
    all the elements of $S$
    such that $\pik{\wac}(S(\omega_i))=\{\omega_i\}$ . 
    
We will prove by induction that there exist $Q_i\in \Omega[T]$ where
$S_i$ is the multiset
    of its roots, and $P_{\omega_j}\in\Omega[T]$ where $S_{\omega_j}
    $ is the multiset of its roots for each $j\in \{1,\cdots, i\}$,  such that $P(d)=Q_i(d).P_{\omega_i}(d)\cdots
    P_{\omega_1}(d)$, $\Sigma_{Q_i}\subset \{\omega_{i+1},\cdots,
    \omega_\nu\}$, $\pik{\wac}(S_{\omega_j})=\{\omega_j\}$ and
    $\Deg(P_{\omega_j})=\car(S_{\omega_j})=\car(S(\omega_j))$ for each
    $j\in \{1,\cdots,i\}$. First of all, for each $\omega_i$ we will associate an
    $a_i\in k$ as in Lemma~\ref{sec:spectr-vers-youngs-8}, and we set $r_i=|T(\omega_i)-a_i|$.

    Now we prove the induction hypothesis for $i=1$. For each $z_l\in S$ with
    $\pik{\wac}(z_l)\ne \omega_1$, we have $|z_l-a_1|>r_1$. Since $r_1>\nor{d}$, then by
    Lemma~\ref{sec:spectr-vers-youngs} there exists a unique $Q,R\in
    \Omega[T]$ such that $P(d+a_1)=Q(d).R(d)$ and
    $\{\alpha_1,\cdots,\alpha_{n_1}\}$
    (resp. $\{\alpha_{n_1+1},\cdots, \alpha_n\}$) the multiset of
    roots of $R(T)$ (resp. $Q(T)$) with $|\alpha_l|=r_1$ for all
    $l\leq n_1$ and $|\alpha_l|>r_1$ for all $l>n_1$. We set
    $Q_1(T)=Q(T-a_1)$ and $P_{\omega_1}(T)=R(T-a_1)$ and $S_1$
    (resp. $S_{\omega_1}$) the multiset of roots of $Q_1(T)$ (resp. $P_{\omega_1}(T)$). We clearly have
    $\car(S_{\omega_1})=\car(S(\omega_1))$, it remains to prove that
    $\pik{\wac}(S_{\omega_1})=\{\omega_1\}$ and $\Sigma_{Q_1}\subset \{\omega_2,\cdots,\omega_\nu\}$. It is clear that
$\pik{\wac}(S_1)\cap \pik{\wac}(S_{\omega_1})=\emptyset$. Moreover if $z$ is a root of $Q_1(T)$
    or $P_{\omega_1}(T)$ then we must have $r_k(z)>\nor{d}$. On
    one hand we have the short exact sequence
    \begin{equation}0\to (\DD_\Omega/\DD_\Omega.Q_1,d)\to (\DD_\Omega/\DD_\Omega.P(d),d)\to (\DD_\Omega/\DD_\Omega.P_{\omega_1}(d),d)\to 0.\end{equation}
    On the other hand, by Theorem~\ref{sec:spectr-vers-youngs-6} we have
    $\Sigma_{Q_1}\subset \pik{\wac}(S_1)$,
    $\Sigma_{P_{\omega_1}}\subset \pik{\wac}(S_{\omega_1})$ and
    $\Sigma_P\subset \pik{\wac}(S)$. Since $\pik{\wac}(S_1)\cap
    \pik{\wac}(S_{\omega_1})=\emptyset$, we must have $\Sigma_{Q_1}\cap
    \Sigma_{P_{\omega_1}}=\emptyset$. By Lemma~\ref{sec:spectr-vers-youngs-1}, we
    have $\Sigma_{Q_1}\cup\Sigma_{P_{\omega_1}}=\Sigma_P$. Hence
    $\Sigma_{Q_1},\,\Sigma_{P_{\omega_1}} \subset \{\omega_1,\cdots,\omega_\nu\}$. Since
    $\omega_1$ is the only point such that $|T(\omega_1)-a_1|=r_1$,
    then we must have $\Sigma_{P_{\omega_1}}=\{\omega_1\}$ and $\Sigma_{Q_1}\subset\{\omega_2,\cdots,\omega_\nu\}$. If we
    suppose that there exists $\omega\in
    \pik{\wac}(S_{\omega_1})\setminus\{\omega_1\}$, then there exists $a\in k$ such
    that $|T(\omega)-a|<|T(\omega)-a_1|=r_1$. In particular we have
    $|a-a_1|=r_1$. This means that there exists a root $z$ of $P(T)$
    such that $\pik{\wac}(z)\ne \omega_1$ and $|z-a|< r_1$, on the other
    hand we have
    \begin{equation}|z-a|=|z-a_1+a_1-a|=|z-a_1|>r_1,\end{equation}
    which is
    impossible. Consequently, we have
    $\pik{\wac}(S_{\omega_1})=\{\omega_1\}$.

    Assume now that we have $P(d)=Q_i(d)\cdot P_{\omega_i}(d)\cdots P_{\omega_1}(d)$
    that satisfies the induction hypothesis for $i<\nu$. Let $S_{r_{i+1}}$
    be the multiset of all the roots $z$ of $P(T)$ such that
    $|z-a_{i+1}|=r_{i+1}$. There exist
    $\car(S_{r_{i+1}})-\car(S(\omega_{i+1}))$ roots $z\in S_{r_{i+1}}$
    such that $\pik{\wac}(z)=\omega_j$ with $j\leq i$, i.e. $|z-a_j|=r(z)=r_j$. This means
    that if $R(d)=P_{\omega_i}(d)\cdots P_{\omega_1}(d)$, then $R(T)$ admits
    exactly  $\car(S_{r_{i+1}})-\car(S(\omega_{i+1}))$ roots $z$
    counted with multiplicity such that $|z-a_{i+1}|=r_{i+1}$. Therefore $Q_i$ admits exactly $\car(S(\omega_{i+1}))$ roots $z$ counted
    with multiplicity such that $|z-a_{i+1}|=r_{i+1}$. On the other hand
    if $z\in S\setminus (S_{r_{i+1}}\cup \bigcup_{j\leq
      i}S(\omega_j))$, than we have $|z-a_{i+1}|> r_{i+1}$. Hence, if
    $z\in S_i$ than $|z-a_{i+1}|\geq r_{i+1}$.  By
    Lemma~\ref{sec:spectr-vers-youngs}, there exists $Q_{i+1}$
    (resp. $P_{\omega_{i+1}}$) with $S_{i+1}$ (resp. $S_{\omega_{i+1}}$) the multiset of its roots, such that $Q_i(d)=Q_{i+1}(d)\cdot
    P_{\omega_{i+1}}(d)$, $\forall \alpha\in S_{\omega_{i+1}}$ we have
    $|\alpha-a_{i+1}|=r_{i+1}$, $\forall \alpha\in S_{i+1}$ we have
    $|\alpha-a_{i+1}|>r_{i+1}$ and
    $\car(S_{\omega_{i+1}})=\car(S(\omega_{i+1}))$. It remains to
    prove that $\pik{\wac}(S_{\omega_{i+1}})=\{\omega_{i+1}\}$ and
    $\Sigma_{Q_{i+1}}\subset \{\omega_{i+2},\cdots,\omega_\nu\}$. It
    is clear that
    $\pik{\wac}(S_{i+1})\cap\pik{\wac}(S_{\omega_{i+1}})=\emptyset$. Moreover if $\alpha$ is a root of $Q_{i+1}(T)$
    or $P_{\omega_{i+1}}(T)$ then we must have $r_k(z)>\nor{d}$. On the
    one hand we have the short exact sequence
    \begin{equation}0\to (\DD_\Omega/\DD_\Omega.Q_{i+1}(d),d)\to (\DD_\Omega/\DD_\Omega.Q_i(d),d)\to (\DD_\Omega/\DD_\Omega.P_{\omega_{i+1}}(d),d)\to 0.\end{equation}
    On the other hand, by Theorem~\ref{sec:spectr-vers-youngs-6} we have
    $\Sigma_{Q_{i+1}}\subset \pik{\wac}(S_{i+1})$,
    $\Sigma_{P_{\omega_{i+1}}}\subset \pik{\wac}(S_{\omega_{i+1}})$ and
    $\Sigma_P\subset \pik{\wac}(S)$. Since $\pik{\wac}(S_1)\cap
    \pik{\wac}(S_{\omega_{i+1}})=\emptyset$, we must have $\Sigma_{Q_{i+1}}\cap
    \Sigma_{P_{\omega_{i+1}}}=\emptyset$. By Lemma~\ref{sec:spectr-vers-youngs-1}, we
    have $\Sigma_{Q_{i+1}}\cup\Sigma_{P_{\omega_{i+1}}}=\Sigma_{Q_i}$. Hence
    $\Sigma_{Q_i},\,\Sigma_{P_{\omega_1}} \subset \{\omega_{i+1},\cdots,\omega_\nu\}$. Since
    $\omega_{i+1}$ is the only point such that $|T(\omega_{i+1})-a_{i+1}|=r_{i+1}$,
    then we must have $\Sigma_{P_{\omega_{i+1}}}=\{\omega_{i+1}\}$ and $\Sigma_{Q_{i+1}}\subset\{\omega_{i+2},\cdots,\omega_\nu\}$. If we
    suppose that there exists $\omega\in
    \pik{\wac}(S_{\omega_{i+1}})\setminus\{\omega_{i+1}\}$, then there exists $a\in k$ such
    that $|T(\omega)-a|<|T(\omega)-a_{i+1}|=r_{i+1}$. By
    Lemme~\ref{sec:spectr-vers-youngs}, there exists $P_1$ and $P_2$
    such that $P_{\omega_{i+1}}(d)=P_1(d).P_2(d)$ with the roots $z$
    of $P_1(d)$ (resp. $P_2(d)$) satisfying $|z-a|\geq r_{i+1}$
    (resp. $|z-a|<r_{i+1}$). By the same argument as above we have
    $\Sigma_{P_1}\cap\Sigma_{P_2}=\emptyset$ and
    $\Sigma_{P_1}\cup\Sigma_{P_2}=\Sigma_{P_{\omega_{i+1}}}=\{\omega_{i+1}\}$,
    which is impossible. Therefore we have
    $\pik{\wac}(S_{\omega_{i+1}})=\{\omega_{i+1}\}$, which proves the
    induction for $i+1$.

    Let $Q_{\nu-1}\in \Omega[T]$ with $S_{\nu-1}$ the multiset
    of its roots and $P_{\omega_j}\in\Omega[T]$ with $S_{\omega_j}
    $ the multiset of its roots,
    for each $j\in \{1,\cdots, \nu-1\}$,  such that $P(d)=Q_{\nu-1}(d).P_{\omega_i}(d)\cdots
    P_{\omega_1}(d)$, $\Sigma_{Q_{\nu-1}}\subset \{
    \omega_\nu\}$, $\pik{\wac}(S_{\omega_j})=\{\omega_j\}$ and
    $\Deg(P_{\omega_j})=\car(S_{\omega_j})=\car(S(\omega_j))$ for each
    $j\in \{1,\cdots,\nu-1\}$. Then we must have
    $\Deg(Q_{\nu-1})=\car(S(\omega_{\nu}))$. Since
    $\Sigma_{Q_{\nu-1}}\ne\emptyset$, we have $\Sigma_{Q_{\nu-1}}=\{
    \omega_\nu\}$. We can prove as above that we must have
    $\pik{\wac}(S_{\nu-1})=\{\omega_{\nu}\}$. We set
    $P_{\omega_{\nu}}:=Q_{\nu+1}$ and we get our decomposition $P(d)=P_{\omega_{\nu}}(d)\cdots P_{\omega_{1}}(d)$.
  \end{proof}

  \begin{rem}
    This result can be seen as a strong version of the decomposition
    of a differential polynomial by its slopes. 
  \end{rem}

  \begin{Theo}[Spectral version of Young's theorem (strong version)]\label{sec:spectr-vers-youngs-9}
   Let $\Omega\in E(k)$ and $d:\Omega\to \Omega$ be a $k$-linear bounded
    derivation. Let $(M,\nabla)$ be a differential module over
    $(\Omega,d)$ with $(M,\nabla)\simeq (\DD_\Omega/\DD_\Omega\cdot
    P(d), d)$ and $P(d)=\sum_ia_i d^i+d^n$. Let $\{z_1,\cdots,
    z_n\}\subset \Omega^{alg}$ be the multiset of the roots of $P(T)$
    (the commutative polynomial associated to $P(d)$). If $ \min_i
    r_k(\pik{\wac}(z_i))>\nor{d}$, then
    \begin{equation}\Sigma_{\nabla,k}(\Lk{M})=
      \pik{\wac}(\{z_1,\cdots,z_n\}).\end{equation}
  \end{Theo}

  \begin{proof}
    We set $\pik{\wac}(\{z_1,\cdots,z_n\})=\{\omega_1,\cdots,
    \omega_\nu\}$ such that $r_k(\omega_1)\leq \cdots\leq r_k(\omega_\nu)$. By Theorem~\ref{sec:spectr-vers-youngs-7} we have
    $P(d)=P_{\omega_\nu}(d)\cdots P_{\omega_1}(d)$. We set
    $P_{\omega_{1,\cdots,j}}(d):=P_{\omega_{j}}(d)\cdots
    P_{\omega_{1}}(d)$ and
    $(M_{\omega_{1,\cdots,j}},\nabla_{\omega_{1,\cdots,j}})=(\DD_\Omega/\DD_\Omega.P_{\omega_{i_1,\cdots,i_j}},d)$.
    By
    Theorem~\ref{sec:spectr-vers-youngs-6}, we have
    $\Sigma_{\nabla_{\omega_i}}(\Lk{M_{\omega_i}})=\{\omega_i\}$. By
    induction and using Lemma~\ref{sec:spectr-vers-youngs-1} we obtain
    the result. 
  \end{proof}

  \begin{rem}
    Assume that $k$ is trivially valued. Let $r<1$. If
    $(M,\nabla)$ is a differential module over $(\h{x_{0,r}},S\d)$ with
    irregular singularities, the use of Turrittin's theorem in 
    \cite{Azz21} was unavoidable to determine the
    spectrum of $\nabla$. However, by using only
    Theorem~\ref{sec:spectr-vers-youngs-9} we can obtain the result directly. 
  \end{rem}

  \section{Spectrum of differential module}\label{sec:spectr-diff-module-5}
  The aim of this section is to determine the spectrum of a differential
module $(M,\nabla)$ over $(\Hx, (S-c)\d)$, where $x\in (c,\infty)$ and
$c\in k$. Note that we can reduce the computation only for $c=0$.

The section is divided into four parts, the first one is to recall the
definitions of subsidiary spectral radii of
  convergence. In the special case where they are small, we
  establish the link between 
  the spectral radii and the
  spectrum. In the second, we determine the spectrum of the pull-back
  by the Frobenius of a differential module having small radii, and we
  establish the link between the spectrum and the spectral radii. In
  third part we announce and prove the main result of the paper. In
  the last part, we explain how we can deduce from the main result the shape of the spectrum of a
differential module $(M,\nabla)$ over $(\Hx,d)$, where $x$ is a point of
a quasi-smooth curve of type (2) or (3), and $d$ is well chosen
$k$-linear bounded derivation.
\begin{hyp} We still assume that $k$ is algebraically closed.
\end{hyp}
 
\subsection{Link between the radius of convergence and the
  spectrum when the radii are small} 

\paragraph{Subsidiary spectral radii of
  convergence} Let $x\in \Ak\setminus k$. Let $(M,\nabla)$ be a differential module over $(\Hx,
\d)$ with rank equal to $n$. We set

\begin{equation}\label{eq:1}
\Rd{(M,\nabla)}_1(x):=\frac{\omega}{\nsp{\nabla}}
\end{equation}

Consider the following Jordan-Hölder sequence of $(M,\nabla)$
\begin{equation}
  \label{eq:7}
  0=M_0\subset M_1\subset\cdots \subset M_\nu=M.
\end{equation}
This means that for all
$i$, $N_i:=M_i/M_{i-1}$ has non trivial strict differential
sub-modules.
Let $n_i$ be the rank of $N_i$, and let $R_i:=\Rd{(N_i,\nabla_i)}_1(x)$. Perform a
permutation of the indexes  to get $R_1\leq \cdots \leq
R_\nu$. Let $\Rd{(M,\nabla)} (x): \Rd{(M,\nabla)}_1(x)\leq\cdots\leq \Rd{(M,\nabla)}_n(x)$ be the sequence
obtained from $R_1\leq\cdots\leq R_\nu$ by counting the value $R_i$
with multiplicity $n_i$, i.e.
\begin{equation}
  \label{eq:11}
  \Rm(x):\; \underbrace{R_1=\cdots =R_1}_{n_1\,\text{times}}\leq \underbrace{R_2=\cdots =R_2}_{n_2\,\text{times}}\leq\cdots\leq \underbrace{R_\nu=\cdots =R_\nu}_{n_\nu\,\text{times}}.
\end{equation}
The values $\Rd{(M,\nabla)}_i(x)$ are called the {\it subsidiary spectral radii of
  convergence} of $(M,\nabla)$. We will just denote by
$\Rm_i(x)$ when no confusion is possible.
\begin{Defi}
  For a differential module $(M,\nabla)$ over $(\Hx,g\d)$ with
  $g\in\Hx\setminus\{0\}$, the {\it subsidiary spectral radii of convergence} of
  $(M,\nabla)$, that we still denote by $\Rd{(M,\nabla)}_i(x)$, are the subsidiary
  radii of $(M,g\-1\nabla)$ as a differential module over $(\Hx,\d)$. If $\Rd{(M,\nabla)}_1(x)=\cdots=\Rd{(M,\nabla)}_n(x)$, we say that $(M,\nabla)$ is {\it pure}.
\end{Defi}

\begin{Lem}[{\cite[Lemma~6.2.3]{Ked}}]
Let $(M,\nabla)$ be a differential module over $(\Hx,\d)$. We have
$\Rm_i(x)\leq r(x)$.
\end{Lem}
\begin{Defi}
  Let $(M,\nabla)$ be a differential module over $(\Hx,\d)$. If $\Rm_i(x)=r(x)$ we say
that $\Rm_i(x)$ is {\it solvable}.
\end{Defi}

\subsubsection{Differential polynomial and radii of convergence}

Let $\Omega\in E(k)$ and let $d:\Omega\to \Omega$ be a $k$-linear bounded derivation. Let $\cL:=\sum_{i=0}^nf_{n-i}\cdot d^i$ be a differential operator
with $f_0=1$. Let $\{\lambda_1,\cdots, \lambda_n\}$ be the multiset
of roots of the commutative polynomial $\tilde{\cL}=\sum_{i=0}^nf_{n-i}\cdot T^i\in \Omega[T]$, and assume that
$|\lambda_n|\leq\cdots \leq |\lambda_1|$. We set:
\begin{equation}
  \label{eq:12}
  \cR_i^{\cL,d}:=\frac{\omega}{\max(\nor{d},|\lambda_{i}|)}.
\end{equation}

\begin{Lem}\label{sec:bf-diff-polyn}
  Let $(M,\nabla)$ be a differential module over $(\Omega,d)$. Let
  $\cL$ and $\cL'$ be two attached differential operator of
  $(M,\nabla)$. Then for each $i$ we have
  \begin{equation}
    \label{eq:29}
    \cR_i^{\cL,d}=\cR_i^{\cL',d}
  \end{equation}
\end{Lem}

\begin{proof}
  See \cite[Corollary~6.5.4]{Ked}.
\end{proof}

\begin{Theo}[{\cite{Young},\cite[Theorem~6.5.3]{Ked},
    \cite[Theorem~6.2]{CM02}}]\label{sec:bf-diff-polyn-1}
  Let $x\in \Ak\setminus k$, and
  $\cL:=\sum_{i=0}^{n-1}f_{n-i} \d ^i+\d^n$ be a differential operator
  with coefficients in $\Hx$. Let
  $(M,\nabla)$ be the
  differential module over $(\Hx,\d)$ attached to $\cL$. Then
  $\cR_i^{\cL,\d}<\omega \cdot r(x)$ if and only if
  $\Rm_i(x)<\omega\cdot r(x)$, and in this case we have
  \begin{equation}
    \label{eq:13}
    \Rm_i(x)=\cR_i^{\cL,\d}.
  \end{equation}
  
\end{Theo}

\subsubsection{Link between the subsidiary radii and the spectrum, the
case of small radii}
Let $x\in \Ak\setminus k$, and let
$(M,\nabla)$ be a differential module over $(\Hx,\d)$. The starting point of our motivation to study the spectrum is the
interesting relation between $\Rm_1(x)$ and the spectrum of
$\nabla$. Indeed, the smallest closed disk centered at zero and
containing $\Sigma_\nabla$ has radius
equal to $\frac{\omega}{\Rm_1(x)}$. In our work \cite{Cons} we 
prove, in the case of  constant coefficients, that if
$\Sigma_\nabla=\bigcup_i\Sigma_i$ where $\Sigma_i$ are the connected component
of the spectrum, then for each $\Rm_j(x)$ there exists $\Sigma_j$ such
that the smallest closed disk centered at zero and containing this
component has radius equal to $\frac{\omega}{\Rm_i(x)}$, and conversely. In this paper we
prove that this can be generalized to the case where the radii are
small. Indeed from Theorem~\ref{sec:spectr-vers-youngs-9} and
\ref{sec:bf-diff-polyn-1} we have the following result.

\begin{Pro}
 Let $x\in \Ak$ be a point of type (2), (3) or (4), and let
$(M,\nabla)$ be a differential module over $(\Hx,\d)$. For each
$a\in k$ we set $(M_a,\nabla_a):=(M,\nabla-a)$. Suppose that
$\Rd{M_a}_i<\omega.r(x)$ for each $a$ and all $i$. Then we have for each
$\Rm_i$ there exists $x_i\in \Sigma_\nabla$ such that
\begin{equation}
  \label{eq:40a}
  |T(x_i)|=\frac{\omega}{\Rm_i(x)},
\end{equation}
and conversely, the same holds for each $x_i\in \Sigma_\nabla$.
\end{Pro}

If $\crk=0$, except if the radii are
solvable, the radii are small and we can compute the spectrum of
$\nabla$. However, if $\crk=p>0$, there are many cases where the
radii are neither small nor solvable. In such situation the use of
the push-forward by $\forp$ (cf. Section~\ref{sec:frobenius-map}) is primordial to compute the radii, but it
is not that easy to use the push-forward by $\forp$ to determine the
spectrum. Indeed, in order to use
Proposition~\ref{sec:push-forw-spectr}, we need to find $g\d$ such
that $\forp^*(g\d)=\d$ which is impossible. That is why it is more
convenient to use the derivation $S\d$, in fact we have
$\forp^*(pS\d)=S\d$.

In the following, we will show how we can recover the data of radii in
the spectrum of a differential module $(M,\nabla)$ over $(\Hx, p^lS\d)$.

\begin{nota}
  We set $s(i,j)$ and $S(i,j)$ to be the numbers satisfying the
  following identities in $\NN[T]$:
  \begin{equation}
    \label{eq:16}
    T(T-1)\cdots (T-i+1)=\sum_{j=0}^is(i,j) T^j,
  \end{equation}
  \begin{equation}
    \label{eq:17}
    T^i=\sum_{j=0}^iS(i,j) T(T-1)\cdots(T-j+1).
  \end{equation}
\end{nota}

\begin{Lem}\label{sec:bf-like-between}
  Let $x=x_{0,r}\in \Ak\setminus k$. Let
  $(M,\nabla_M)$ (resp. $(N,\nabla_N)$) be a differential module over
  $(\Hx,\d)$ (resp. $(\Hx,p^lS\d)$ with $l\in\NN$) such that
  $\cL_{\nabla_M}:=\sum_{i=0}^{n}f_{n-i}(\d)^i$ with $f_0=1$
  (resp. $\cL_{\nabla_N}:=\sum_{i=0}^{n}g_{n-i}(p^lS\d)^i$ with $g_0=1$) is an attached
  differential operator of $(M,\nabla_M)$ (resp. $(N,\nabla_N)$). Then:
  \begin{equation}
    \label{eq:15}
    \cL_{p^lS\nabla_M}:=\sum_{i=0}^{n}p^{(n-i)l}\left(\sum_{j=i}^nf_{n-j}S^{n-j}s(j,i)\right)
    (p^lS\d)^i
  \end{equation}
  is an attached differential operator of $(M,p^lS\nabla_M)$
  (differential module over $(\Hx,p^lS\d)$), and
  \begin{equation}
    \label{eq:18}
    \cL_{p^{-l}S\-1\nabla_N}:=\sum_{i=0}^nS^{(i-n)}\left(\sum^n_{j=i}g_{n-j}p^{l(j-n)}S(j,i)\right)(\d)^i
  \end{equation}
  is an attached differential operator of $(N, p^{-l}S\-1\nabla_N)$
  (differential module over $(\Hx,\d)$).
\end{Lem}
\begin{proof}
We prove by induction that:

  \begin{equation}
    \label{eq:20}
  (\nabla_M)^i=(S)^{-i}\sum_{j=1}^ip^{l(-j)}s(i,j) (p^lS\cdot \nabla_M)^j,
  \end{equation}
  \begin{equation}
    \label{eq:24}
   (\nabla_N)^i=p^{li}\sum_{j=1}^iS(i,j)S^j ((p^lS)\-1\cdot \nabla_N)^j.  
 \end{equation}

 Let $c$ be a cyclic vector of $(M,\nabla_M)$. Then:

   \begin{align}
       \label{eq:26}
 \sum_{i=0}^nf_{n-i}\nabla_M^i(c)  =&\sum_{i=0}^nf_{n-i}(S)^{-i}\sum_{j=1}^ip^{l(-j)}s(i,j)
   (p^lS\cdot\nabla_M)^j(c)
                                       \\
                                       =&\sum_{i=0}^np^{l(-i)}\left(\sum_{j=i}^nf_{n-j}S^{-j}s(j,i)\right)(p^lS\cdot
         \nabla_M)^i(c)
         \\
         =& 0. \end{align}

 Hence, multiplying  the equality by $p^{ln}S^n$ we obtain $\sum_{i=0}^np^{l(n-i)}\left(\sum_{j=i}^nf_{n-j}S^{n-j}s(j,i)\right)(p^lS\cdot
         \nabla_M)^i(c)=0$, and deduce that $\cL_{p^lS\nabla_M}$ is
         an attached differential operator of $(M,p^lS\nabla_M)$.

         Let $e$ be a cyclic vector of $(N,\nabla_N)$. Then:
         \begin{align}
           \label{eq:27}
           \sum_{i=0}^ng_{n-i}\nabla_N^i(e)=&\sum_{i=0}^n
                                              g_{n-i}p^{li}\sum_{j=1}^iS(i,j)S^j
                                              ((p^lS)\-1\cdot
                                              \nabla_N)^j(e)\\
           =&\sum_{i=0}^nS^i\left(\sum_{j=i}^ng_{n-j}p^{lj}S(j,i)\right)
              ((p^lS)\-1\cdot \nabla_N)^i(e)\\
           =&0.
         \end{align}
         Hence, by multiplying by $p^{-nl}S^{-n}$ we obtain
         $\sum_{i=0}^nS^{n-i}\left(\sum_{j=i}^ng_{n-j}p^{l(j-n)}S(j,i)\right)
         ((p^lS)\-1\cdot \nabla_N)^i(e)=0$, and deduce that
         $\cL_{p^{-l}S\-1\nabla_N}$ is an attached differential module
         of $(N, p^{-l}S\-1\nabla_N)$.
       \end{proof}

       \begin{rem}
         We have a bijection
         \begin{equation}
           \label{eq:30}
           \Fonction{\Lambda:\DD_{\Hx,\d}}{\DD_{\Hx,p^lS\d}}{P(\d)}{\Lambda(P)(p^lS\d)},
         \end{equation}
         such that $\Lambda(P)(p^lS\d)$ is the differential polynomial
         obtained from $P(\d)$ using the formula \eqref{eq:15}, and
         $\Lambda\-1(Q)(\d)$ is the differential polynomial obtained
         from $Q(p^lS\d)$ using the formula \eqref{eq:18}.
       \end{rem}

       \begin{Defi}
         Let $(\Omega,|.|)$ be an ultrametric complete field. Let
         $P(T)=\sum_{i=0}^na_iT^i\in \Omega[T]$. For $r>0$, let us denote by $W_r(P)$
          the {\it width } of $P$
         under the $r$-Gauss norm $|.|_r$ (i.e
         $|P|_r:=\max_i|a_i|r^i$), as the difference between the
         maximum and minimum values of $i$ for which $\max_i|a_i|r^i$
         is achieved. 
       \end{Defi}
       \begin{Lem}\label{sec:bf-like-between-2}
      Let $(\Omega,|.|)$ be an ultrametric complete field. Let
      $P(T)=\sum_{i=0}^na_iT^i\in \Omega[T]$ and let $r>0$. If $W_r(P)=l$,
      then $P$ admits exactly $l$ roots, counted with multiplicity,
      whose absolute value are equal to $r$. 
    \end{Lem}
    \begin{proof}
      See \cite[Section~2.1]{Ked}.
    \end{proof}

    \begin{Lem}\label{sec:bf-like-between-1}
      Let $x=x_{0,r}\in \Ak\setminus k$. Let
  $(M,\nabla_M)$ be a differential module over
  $(\Hx,\d)$ such that
  $\cL_{\nabla}:=\sum_{i=0}^{n}f_{n-i}(\d)^i$, with $f_0=1$,
  is an attached
  differential operator of $(M,\nabla)$. Let 
    $\cL_{p^lS\nabla}$ be an attached differential operator of
    $(M,p^lS\nabla)$. Suppose that $\cL_{\nabla}=P_2(\d)\cdot P_1(\d)$
   . Then there exists a decomposition of
    $\cL_{p^lS\nabla}=Q_{2}(p^lS\d)\cdot Q_1(p^lS\d)$ such that
    \begin{equation}
      \label{eq:31}
      \cR^{Q_j,p^lS\d}_i=\cR^{\Lambda(P_j),p^lS\d}_i.
    \end{equation}
  \end{Lem}

  \begin{proof}
   Let $(M_1,\nabla_1)=(\DD_{\Hx,\d}/\DD_{\Hx,\d}\cdot P_1,\d)$ and
   $(M_1,\nabla_2)=(\DD_{\Hx,\d}/\DD_{\Hx,\d}\cdot P_2,\d)$. Then
   \begin{equation}
     \label{eq:1a}
     0\to (M_2,\nabla_2)\to (M,\nabla)\to (M_1,\nabla_1)\to 0,
   \end{equation}
   therefore
   \begin{equation}
     \label{eq:32}
     0\to (M_2,p^lS\nabla_2)\to (M,p^lS\nabla)\to (M_1,p^lS\nabla_1)\to 0.
   \end{equation}
   Consequently, on one hand there exist $Q_1(p^lS\d)$ and
   $Q_2(p^lS\d)$ such that $\cL_{p^lS\nabla}=Q_{2}(p^lS\d)\cdot
   Q_1(p^lS\d)$, $(M_1,p^lS\nabla_1)=(\DD_{\Hx,p^lS\d}/\DD_{\Hx,p^lS\d}\cdot Q_1,p^lS\d)$ and
   $(M_1,p^lS\nabla_2)=(\DD_{\Hx,p^lS\d}/\DD_{\Hx,p^lS\d}\cdot Q_2,p^lS\d)$.
   On the other hand, by Lemma~\ref{sec:bf-like-between} we have $(M_1,p^lS\nabla_1)=(\DD_{\Hx,p^lS\d}/\DD_{\Hx,p^lS\d}\cdot \Lambda(P_1),p^lS\d)$ and
   $(M_1,p^lS\nabla_2)=(\DD_{\Hx,p^lS\d}/\DD_{\Hx,p^lS\d}\cdot
   \Lambda(P_2),p^lS\d)$. By Lemma~\ref{sec:bf-diff-polyn}, we
   conclude that $\cR^{Q_j,p^lS\d}_i=\cR^{\Lambda(P_j),p^lS\d}_i$ for
   each $i$ and $j$.
  \end{proof}

       \begin{Pro}\label{sec:bf-like-between-3}
          Let $x=x_{0,r}\in \Ak\setminus k$. Let
  $(M,\nabla_M)$ be a differential module over
  $(\Hx,\d)$ such that
  $\cL_{\nabla_M}:=\sum_{i=0}^{n}f_{n-i}(\d)^i$, with $f_0=1$,
  is an attached
  differential operator of $(M,\nabla)$. Let 
    $\cL_{p^lS\nabla}$ be an attached differential operator of $(M,p^lS\nabla)$. Then
    $\cR^{\cL_{\nabla},\d}_i<\omega\cdot r(x)$ if and only if
    $\cR^{\cL_{p^lS\nabla},p^lS\d}_i<\omega\cdot |p|^{-l}$ and we have
    \begin{equation}
      \label{eq:28}
      \cR_i^{\cL_{\nabla},\d}=|p|^l r(x)\cR_i^{\cL_{p^lS\nabla},p^lS\d}.
    \end{equation}
  \end{Pro}
  
  \begin{proof}
By
    Lemma~\ref{sec:spectr-vers-youngs} there exists a unique
    decomposition such that $\cL_{\nabla}=P_\mu.P_{\mu-1}\cdots P_0$
    with $\cR^{P_0,\d}_j={\omega r(x)}$ for each $j$, and for each
    $P_i$ ($i\ne 0$), there exists $r_i<\omega r(x)$ such that
    $\cR^{P_i,\d}_j=r_i$. By induction and using
    Lemma~\ref{sec:bf-like-between-1} there exists a decomposition of
    $\cL_{p^lS\nabla_M}=Q_{\mu}.Q_{\mu-1}\cdots Q_0$ such that
    $\cR^{Q_j,p^lS\d}_i=\cR^{\Lambda(P_j),p^lS\d}_i$. Then in order to
    prove the statement it is enough to show that
    $\cR^{P_j,\d}_i=p^lr(x) \cR^{\Lambda(P_j),p^lS\d}_i$. Therefore, we can reduce to the case where
    $\cR^{\cL_{\nabla},\d}_i=\rho$ for each $i$. We choose
    $\cL_{p^lS\nabla,p^lS\d}=\Lambda(\cL_{\nabla,\d})$, we set
    \begin{equation}
      \label{eq:33}
      g_{n-i}=p^{(n-i)l}\sum_{j=i}^nf_{n-j}S^{n-j}s(j,i),
    \end{equation}
    then we get
    $\cL_{p^lS\nabla,p^lS\d}=\sum_{i=0}^ng_{n-i}(p^lS\d)^i$. We set $P(T):=\sum_{i=0}^ng_{n-i}T^i$. Suppose that
    $\rho=\omega.r(x)$ in order to prove that
    $\cR_i^{\cL_{p^lS\nabla},p^lS\d}=\omega|p|^{-l}$ for all $i$, we
    need to prove that the commutative polynomial
    $P$ does not admit a root $\lambda$ with $|\lambda|>|p|^l$. Since
    $\rho= \omega\cdot r(x)$, then
    \begin{equation}
      \label{eq:35}
      |f_{n-i}S^{n-i}|\leq 1.
    \end{equation}
    Therefore we have
     \begin{equation}
      \label{eq:34a}
      |g_{n-i}|\leq |p|^{(n-i)l}
    \end{equation}
    Let
    $\alpha>|p|^{l}$ then we have

    \begin{equation}
      \label{eq:34}
      \forall i,\quad |g_{n-i}|\alpha^i< |g_0|\alpha^n.
    \end{equation}
    Hence, $W_\alpha(P)=0$ and we conclude that $P$ does not admit a root
    with absolute value greater than $|p^l|$
    (cf. Lemma~\ref{sec:bf-like-between-2}). Consequently we obtain
    $\cR_i^{\cL_{p^lS\nabla},p^lS\d}=\frac{\omega}{\nor{p^lS\d}}=\omega|p|^{-l}$.
    Now suppose that $\rho< \omega.r(x)$. Then we have
    \begin{equation}
      \label{eq:36}
      |f_{n-i}S^{n-i}|\leq \left(\frac{\omega.r(x)}{\rho}\right)^{n-i}.
    \end{equation}
    Consequently,
    \begin{equation}
      \label{eq:37}
      |g_{n-i}|\leq|p|^{(n-i)l}\left(\frac{\omega.r(x)}{\rho}\right)^{n-i}.
    \end{equation}
    In order to prove that
    $\cR_i^{\cL_{p^lS\nabla},p^lS\d}=\frac{\rho}{|p|^{l}r(x)}$, it is
    enough to prove that $W_{|p|^l
      \left(\frac{\omega.r(x)}{\rho}\right)}(P)=n$.
    We have
    \begin{equation}
      \label{eq:38}
        |g_{n-i}||p|^{li}
      \left(\frac{\omega.r(x)}{\rho}\right)^i\leq|p|^{(n)l}\left(\frac{\omega.r(x)}{\rho}\right)^{n}
    \end{equation}
    with $|g_n|=|p^{nl}f_n
    S^n|=|p|^{(n)l}\left(\frac{\omega.r(x)}{\rho}\right)^{n}$ and
    $|g_0||p|^{(n)l}\left(\frac{\omega.r(x)}{\rho}\right)^{n}=|p|^{(n)l}\left(\frac{\omega.r(x)}{\rho}\right)^{n}$. Consequently,
    \begin{equation}
      \label{eq:39}
    W_{|p|^l
      \left(\frac{\omega.r(x)}{\rho}\right)}(P)=n,
  \end{equation}
which completes the proof.
\end{proof}
Now we can establish the link between the spectrum of a differential
module over $(\Hx,p^lS\d)$ and the radii of convergence.  
\begin{Pro}\label{sec:bf-like-between-4}
 Let $x=x_{0,r}\in \Ak\setminus k$, and let
$(M,\nabla)$ be a differential module over $(\Hx,p^lS\d)$. For each
$a\in k$ we set $(M_a,\nabla_a):=(M,\nabla-a)$. Suppose that
$\Rd{M_a}_i<\omega.r(x)$ for each $a$ and $i$. Then for each
$\Rm_i$, there exists $x_i\in \Sigma_\nabla$ such that
\begin{equation}
  \label{eq:40}
  \Rm_i(x)=\frac{\omega\cdot |p|^l r(x)}{|T(x_i)|},
\end{equation}
and conversely, the same holds for each $x_i\in \Sigma_\nabla$.
\end{Pro}

\begin{proof}
  The result holds directly from
  Theorem~\ref{sec:spectr-vers-youngs-9}, \ref{sec:bf-diff-polyn-1}
  and Proposition~\ref{sec:bf-like-between-3}.
\end{proof}

\begin{cor}\label{sec:bf-like-between-5}
  Let $x=x_{0,r}\in \Ak\setminus k$, and let
$(M,\nabla)$ be a differential module over $(\Hx,p^lS\d)$. For each
$a\in k$ we set $(M_a,\nabla_a):=(M,\nabla-a)$. If $(M_a,\nabla_a)$
is pure and with small radii for each $a\in k$, then there exists
$z\in \Ak\setminus k$ such that $\Sigma_\nabla=\{z\}$.
\end{cor}

\subsection{Frobenius and spectrum}\label{sec:bf-frob-spectr-5}
As we explained before if $\crk=0$, except the case where the radii are
solvable, the condition of Theorem~\ref{sec:spectr-vers-youngs-9} are
satisfied, which allows to determine the spectrum. However this is not
the case when
$\crk=p>0$. That is why in this part we will focus on this case and
assume from now on that $\crk=p>0$.

Let $x=x_{0,r}$ and $y:=\forp^l(x)$. Let $(M,\nabla)$ be a
differential module over $(\Hy, p^lS\d)$. According to formula
\eqref{eq:14}, we have $(\forp^l)^*(p^lS\d)=S\d$. To avoid confusion,
we set $p^lS(y)\dy:\Hy\to \Hy$ and $S(x)\dx: \Hx\to \Hx$. Recall from Lemma~\ref{sec:sheaf-diff-etale} that,
we have
\begin{equation}
  \label{eq:1b}
  \Hx=\bigoplus_{i=0}^{p^l-1}\Hy\cdot S(x)^i.
\end{equation}
Let
$\{e_1,\cdots,e_n\}$ be a basis of $(M,\nabla)$ and let $G$ be the
associated matrix in this basis. Then $\{e_1\ot 1,\cdots,e_n\ot 1\}$
is a basis of $(\forp^l)^*(M,\nabla)$ for which the associated matrix
is $G$. Moreover, the following inequality holds \cite[Lemma~10.3.2]{Ked}
\begin{equation}
  \label{eq:47}
  \Rd{(\forp^l)^*M}_1(x)\geq \min(\Rm_1(y)^{\fra{p}},p\Rm_1(y)).
\end{equation}

We set
$(M_{p^l},\nabla_{p^l}):=(\forp^l)_*(\forp^l)^*(M,\nabla)$. The
associated matrix of
$(M_{p^l},\nabla_{p^l})$ in the basis $\{e_1\ot 1,\cdots,e_n\ot 1,
e_1\ot S(x),\cdots,e_n\ot S(x),\cdots, e_1\ot
S(x)^{p^l-1},\cdots,e_n\ot S(x)^{p^l-1}\}$ is:
\begin{equation}
  \label{eq:41}
  \left(
     \raisebox{0.5\depth}{%
       \xymatrixcolsep{1ex}%
       \xymatrixrowsep{1ex}%
   \xymatrix{G& 0\ar@{.}[rr]\ar@{.}[ddrr] &    &0\ar@{.}[dd]\\
                    0\ar@{.}[ddrr]\ar@{.}[dd]&G+I_n\ar@{.}[ddrr]                 &    & \\
                      &                   &    &0\\
                    0\ar@{.}[rr]&                   &0
                    &G+(p^l-1)\cdot I_n\\}
}
   \right).
 \end{equation}
 Therefore we have the following isomorphism:
 \begin{equation}
   \label{eq:42}
   (M_{p^l},\nabla_{p^l})\simeq \bigoplus_{i=0}^{p^l-1}(M,\nabla+i).
 \end{equation}
 By Proposition~\ref{sec:push-forw-spectr} and
 Remark~\ref{sec:spectr-diff-module} we have
 \begin{equation}
   \label{eq:43}
   \Sigma_{(\forp^l)^*\nabla,k}(\Lk{(\forp^l)^*M})=\Sigma_{\nabla_{p^l},k}(\Lk{M_{p^l}})=\bigcup_{i=0}^{p^l-1}(\Sigma_{\nabla,k}(\Lk{M})+i).
 \end{equation}

 \begin{Theo}\label{sec:bf-frob-spectr}
   Let $x=x_{0,r}\in \Ak\setminus k$, $y:=\forp^l(x)$ and  consider
   the embedding $\Hy\hookrightarrow \Hx$ induced by $\forp^l$. Let
   $(M,\nabla)$ be a differential module over $(\Hx,S\d)$. Assume that
   $(M,\nabla)\simeq (\DD_{\Hx}/\DD_{\Hx}\cdot P(S\d),S\d) $ with
   $P(S\d)=\sum_i a_iS\d+S\d^n$ and $a_i\in \Hy$. Let
   $\{z_1,\cdots,z_n\}\subset\Hx^{alg}$ be the multiset of the roots
   of $P(T)$. If $\min_i r_k(\pik{\widehat{\Hx^{alg}}}(z_i))>|p|^l$,
   then
   \begin{equation}
     \label{eq:44}
     \Sigma_{\nabla,k}(\Lk{M})=\bigcup_{i=0}^{p^l-1}(\pik{\widehat{\Hx^{alg}}}(\{z_1,\cdots,z_n\})+i).
   \end{equation}
 \end{Theo}

 \begin{proof}
   Let $(N,\nabla_N)\simeq (\DD_{\Hy}/\DD_{\Hy}\cdot
   P(p^lS\d),p^lS\d)$. Then we have $(M,\nabla)\simeq
   (\forp^l)^*(N,\nabla_N)$. From \eqref{eq:43} we have:
   \begin{equation}
     \label{eq:45}
     \Sigma_{\nabla,k}(\Lk{M})=\bigcup_{i=0}^{p^l-1}(\Sigma_{\nabla_N,k}(\Lk{N})+i).
   \end{equation}
   It remains to prove that
   $\Sigma_{\nabla_N,k}(\Lk{N})=\pik{\widehat{\Hx^{alg}}}(\{z_1,\cdots,z_n\})$. Since
   $\Hx$ is an algebraic finite extension of $\Hy$ then we have $\widehat{\Hx^{alg}}\simeq\widehat{\Hy^{alg}}$. Hence, we should
   have
   $\pik{\widehat{\Hx^{alg}}}(\{z_1,\cdots,z_n\})=\pik{\widehat{\Hy^{alg}}}(\{z_1,\cdots,z_n\})$. Recall
   that we have $\nor{p^lS\d}=|p|^l$
   (cf. Lemma~\ref{sec:spectr-deriv-sd}). Then by
   Theorem~\ref{sec:spectr-vers-youngs-9} we have
   $\Sigma_{\nabla_N,k}(\Lk{N})=\pik{\widehat{\Hx^{alg}}}(\{z_1,\cdots,z_n\})$,
   which ends the proof.
 \end{proof}
 Now we need to establish, on the one hand, the link between the spectrum
 and the spectral radii of convergence. On the other hand, we are looking for
 conditions such that
 the hypothesis of Theorem~\ref{sec:bf-frob-spectr} hold. For that we
 need the following result.

 \begin{Theo}[after Christol and Dwork {\cite[Theorem~10.4.2]{Ked}}]\label{sec:bf-frob-spectr-2}
   Let $x=x_{0,r}\in \Ak\setminus k$, $y:=\forp(x)$. Let
   $(M,\nabla)$ be a finite differential module over $(\Hx,S\d)$ such
   that $\Rm_1(x)>\omega.r(x)$. Then there exists a
   unique (up to isomorphism) $(N,\nabla_N)$ such that
   $(M,\nabla)\simeq (\forp)^*(N,\nabla_N)$ and
   $\Rd{N}_1(y)>\omega^p\cdot r(y)$. Moreover,
   \begin{equation}
     \label{eq:48}
     \Rd{N}_i(y)=\Rm_i(x)^{p}.
   \end{equation}
 \end{Theo}
 \begin{rem}
   The differential module $(N,\nabla_N)$ in
   Theorem~\ref{sec:bf-frob-spectr-2} is called the {\it Frobenius
     antecedent } of $(M,\nabla)$.
 \end{rem}
 \begin{rem}\label{sec:bf-frob-spectr-6}
   By induction we can prove that, if moreover $\Rm_1(x)>\omega^{\fra{p^{l-1}}}.r(x)$, then there exists a
   unique (up to isomorphism) $(N,\nabla_N)$ such that
   $(M,\nabla)\simeq (\forp^l)^*(N,\nabla_N)$ and
   $\Rd{N}_1(y)>\omega^p\cdot r(y)$. Moreover, we have
   $\Rd{N}_i(y)=\Rm_i(x)^{p^l}$. We will call it the {\it l-Frobenius
     antecedent } of $(M,\nabla)$.
 \end{rem}

 \begin{rem}\label{sec:bf-frob-spectr-1}
   In particular if $(M,\nabla)=(\forp^l)^*(N,\nabla)$ and
   $\Rd{N}_i(y)>\omega^p\cdot r(y)$, then we must have $\Rd{N}_i(y)=\Rm_i(x)^{p^l}$.
 \end{rem}

 Therefore, the following proposition shows the link between the
 spectrum and the radius of convergence.

 \begin{Pro}\label{sec:bf-frob-spectr-3}
 Let $x=x_{0,r}\in \Ak\setminus k$, $y:=\forp^l(x)$ and we consider
   the embedding $\Hy\hookrightarrow \Hx$ induced by $\forp^l$. Let
   $(M,\nabla)$ be a differential module over $(\Hx,S\d)$. Assume that
   $(M,\nabla)\simeq (\DD_{\Hx}/\DD_{\Hx}\cdot P(S\d),S\d) $ with
   $P(S\d)=\sum_i a_iS\d+S\d^n$ and $a_i\in \Hy$. Let
   $\{z_1,\cdots,z_n\}\subset\Hx^{alg}$ be the multiset of the roots
   of $P(T)$. If $\min_i r_k(\pik{\widehat{\Hx^{alg}}}(z_i))>|p|^l$
   and $\max_i|T(z_i)|<|p|^{l-1}$,
   then for each $\Rm_i(x)$ there exists $z_i$ such that:
   \begin{equation}
     \label{eq:46}
     \Rm_i(x)=\frac{\omega^{\fra{p^l}}.|p|^{\frac{l}{p^l}}.r(x)}{|T(z_i)|^{\fra{p^l}}},
   \end{equation}
   and conversely, the same holds for each $z_i$.
 \end{Pro}
 
 \begin{proof}
Let $(N,\nabla_N)\simeq (\DD_{\Hy}/\DD_{\Hy}\cdot
   P(p^lS\d),p^lS\d)$. Then we have $(M,\nabla)\simeq
   (\forp^l)^*(N,\nabla_N)$. By Theorem~\ref
{sec:spectr-vers-youngs-9} and Proposition~\ref{sec:bf-like-between-4}
for each $\Rm_i(y)$ there exists $z_i$ such that:
\begin{equation}
  \label{eq:51}
    \Rd{N}_i(y)=\frac{\omega\cdot |p|^l r(y)}{|T(z_i)|}=\frac{\omega\cdot |p|^l r(x)^{p^l}}{|T(z_i)|}>\omega^p\cdot
    r(x)^{p^l}=\omega\cdot r(y),
  \end{equation}
  By Remark~\ref{sec:bf-frob-spectr-1} we obtain the result.
 \end{proof}

 \begin{rem}
   The condition $\max_i|T(z_i)|<|p|^{l-1} $ of Proposition~\ref{sec:bf-frob-spectr-3} means that
   if $(M,\nabla)\simeq(\forp^l)^*(N,\nabla_N)$, then $(N,\nabla_N)$
   does not admit a Frobenius antecedent. 
 \end{rem}

 Note that in practice we can compute the spectrum in a more general
 case. However, it is not that easy to recover the link between the
 spectrum and the spectral radii of convergence as in formula \eqref{eq:40} and
 \eqref{eq:46}. Indeed, let $x=x_{0,r}$ with $r>0$, given a differential module
 $(M,\nabla)$ over $(\Hx,S\d)$, the following results show how to compute the
 radii of ${\forp}_*(M,\nabla)$. 
 \begin{Pro}[{\cite[Theorem~10.5.1]{Ked}}]\label{sec:bf-frob-spectr-4}
   Let $x=x_{0,r}\in \Ak\setminus k$. Let $(M,\nabla)$ be a
   differential module over $(\Hx, S\d)$ of rank $n$ with subsidiary
   spectral radii $\Rm_1(x)\leq\cdots\leq \Rm_n(x)$. Then the multiset
   of spectral subsidiary radii
   of ${\forp}_*(M,\nabla)$ is   \begin{equation}
     \label{eq:49}
     \bigcup_{i=1}^n
     \begin{cases}
       \{\Rm_i(x)^p,\, \omega^p r(x)^p \, (p-1\text{ times})\}&
       \Rm_i(x)\geq\omega r(x),\\
       \{|p|r(x)^{p-1}\Rm_i(x) \, (p\text{ times})\}&
       \Rm_i(x)\leq\omega r(x)
     \end{cases}.
   \end{equation}
 \end{Pro}
 In practice, by induction, we can prove the following result.
 \begin{cor}\label{H2}
   Let $x=x_{0,r}\in \Ak\setminus k$. Let $(M,\nabla)$ be a
   differential module over $(\Hx,S\d)$ of rank $n$ with subsidiary
   radii $\Rm_1(x)\leq \cdots \leq \Rm_m(x)$ . Suppose that
   $(M,\nabla)$ is pure and $\omega^{\frac{1}{p^{l-1}}} r(x)\leq
   \Rm_1(x)\leq \omega^{\frac{1}{p^l}} r(x)$, with
   $l\in\NN\setminus\{0\}$. Then the multiset of subsidiary spectral radii of $(\forp)^l_*(M,\nabla)$ is
   \begin{equation}
   \label{eq:H1}
       \bigcup_{i=2}^l\{|p|^i \omega r(x)^{p^l}\, (n(p-1)p^{i-1}\text{ times})\}\cup \{|p|\omega r(x)^{p^l}\, (n(p-1) \text{ times})\}\cup \{\Rm_1(x)^{p^l}\, (n \text{ times})\}.
   \end{equation}
 \end{cor}
 \begin{rem}\label{H3}
   We keep the assumption of Corollary~\ref{H2}. We can observe from
   equation~\eqref{eq:H1} that the greatest spectral subsidiary
   spectral radius of $(\forp)^l_*(M,\nabla)$ is equal to $\Rm_1(x)^{p^l}$.
This criterion allows to recover easily the radii of convergence of $(M,\nabla)$. \end{rem}
\subsection{The main result}We assume here that $\crk=p>0$. Now we will determine
the spectrum of any differential equation and establish the link with the subsidiary radii.

\begin{Theo}[{Robba, \cite[Corollary~3.6.9]{np3}, \cite[Theorem~10.6.2]{Ked}}]\label{sec:bf-spectr-comp}
Let $x=x_{0,r}\in\Ak\setminus k$. Let $(M,\nabla)$ be a differential
module over $(\Hx,S\d)$. Then there exists a unique decomposition
\begin{equation}
  \label{eq:52}
  (M,\nabla)=\bigoplus_{\rho\leq r(x)}(M_\rho,\nabla_\rho)
\end{equation}
of differential modules, such that every sub-quotient $(N,\nabla_N)$of
$(M_\rho,\nabla_\rho)$ satisfies $\Rd{N}_{i}(x)=\rho$
. This decomposition is called \emph{the spectral decomposition}.
\end{Theo}
\begin{cor}\label{H1}
  The spectral decomposition \eqref{eq:52} can be refined as
  follows:
  \begin{equation}
      (M,\nabla)=\bigoplus_{i=1}^{\nu}(M_{i},\nabla_{i})
  \end{equation}
  such that for each $i$ and for all $a\in k$, every sub-quotient of
  $(M_{i},\nabla_{i}-a)$ is pure with radius equal to $\Rd{(M_{i},\nabla_{i}-a)}$.
\end{cor}

\begin{proof}
  We will proceed by contradiction and suppose that a decomposition
  \begin{equation}
    \label{eq:61}
    (M,\nabla)=\bigoplus_{i=1}^{\nu}(M_{i},\nabla_{i})
  \end{equation}
 such that for each $i$ and for all $a\in k$ the differential module $(M_{i},\nabla_{i}-a)$
 is pure does not exist. 
  Then we prove by
  induction that there exists a family of finite sets $(I_l)_{l\in
    \NN}$ such that $\text{Card}(I_l)<\text{Card}(I_{l+1})$ such that
  \[(M,\nabla)=\bigoplus_{i\in I_l} (M_i,\nabla_i),\]
  with $(M_i,\nabla_i)\ne 0$. We set $I_0$ such that
  \[(M,\nabla)=\bigoplus_{i\in I_0} (M_i,\nabla_i),\]
  is the decomposition induced by
  Theorem~\ref{sec:bf-spectr-comp}. We suppose now that the induction
  hypothesis is true until $l$, i.e there exist $(I_j)_{j\leq l}$ with $\text{Card}(I_j)<\text{Card}(I_{j+1})$ such that
  \[(M,\nabla)=\bigoplus_{i\in I_j} (M_i,\nabla_i),\]
   with $(M_i,\nabla_i)\ne 0$. By the contradiction hypothesis there exists
   $i_0\in I_l$ and $a_l\in k$ such that $(M_{i_0},\nabla_{i_0}-a)$ is
   not pure, then there exists a finite set $J_l$ with
   $\text{Card}(J_l)\geq 2$, such that
   \[(M_{i_0},\nabla_{i_0}-a)=\bigoplus_{j\in
       J_l}(M_j,\nabla_j-a_l),\]
   with $(M_j,\nabla_j-a_l)\ne 0$. We set $I_{l+1}=(I_l\setminus
   \{i_0\})\coprod J_l$, then we get
   \[(M,\nabla)=\bigoplus_{i\in I_{l+1}} (M_i,\nabla_i),\]
with $(M_i,\nabla_i)\ne 0$. Since for all $l$ and all $i\in I_l$ we
have $(M_i,\nabla_i)\ne 0$, we should have for all $l$,
$\text{Card}(I_l)\leq \text{dim}(M)=n$. Hence we obtain a strictly
increasing sequence of bounded integers, which is absurd. We 
consider now a decomposition like in \eqref{eq:61}, then by
Theorem~\ref{sec:bf-spectr-comp} every sub-quotient of
  $(M_{i},\nabla_{i}-a)$ is pure with radius equal to $\Rd{(M_{i},\nabla_{i}-a)}$. 
\end{proof}

\begin{Theo}[(Christol-Dwork) {\cite[Theorem~6.5.3]{Ked}}]\label{H4}
Let $\Omega\in E(k)$ and $d:\Omega\to \Omega$ be a bounded derivation. Let $\cL\in \DD_\Omega\setminus k$, we set $(M,\nabla):=(\DD_\Omega/\DD_\Omega\cdot \cL,d)$. Then we have 
\begin{equation}
    \label{eq:H3}
    \max(\nor{d},\nsp{\nabla})=\frac{\omega}{\cR_1^{\cL,d}}
\end{equation}
\end{Theo}

\begin{rem}\label{H6}
 From Theorem~\ref{H4} we can deduce the following. Let $x\in \Ak\setminus k$. For a differential module $(M,\nabla)$ over $(\Hx,g\d)$, with $g\in\Hx$, if $\Rm_1(x)\geq \omega.r(x)$, then we have
 \begin{equation}
 \label{eq:H4}
     \nsp{\nabla}\leq \nor{g\d}.
 \end{equation}
\end{rem}

\begin{Lem}[{\cite[Corollary~10.6.3]{Ked}}]\label{H5}
  Let $x=x_{0,r}\in \Ak\setminus k$. Let $(M,\nabla)$ be a
  differential module over $(\Hx,S\d)$. Suppose that all the spectral radii of $(M,\nabla)$ are not solvable, then $\nabla$ is a bijective operator.
\end{Lem}

\begin{Lem}\label{sec:bf-spectr-comp-1}
Let
$x=x_{0,r}\in\Ak\setminus k$ and we set $x^{p^l}:=\forp^l(x)$. Let $(M,\nabla)$ be a differential
module over $(\Hx,S\d)$. For each $a\in k$ we set
$(M_a,\nabla_a):=(M,\nabla-a)$. Assume that $\Rd{M_a}_i(x)< r$ for all
$i$ and $a\in k$. There
exists $l\in \NN$ such that all the radii of $((\forp^l)_*M,(\forp^l)_*(\nabla)-a)$
are strictly less then $\omega. r(x)^{p^l}$ for each $a\in k$. 
\end{Lem}

\begin{proof}
First of all, since $\Rd{M_a}_i(x)< r$ for all
$i$ and $a\in k$, by Lemma~\ref{H5} $\nabla-a$ is bijective. By Open
Mapping Theorem (cf. \cite[Section~2.8.1]{Bosc}), $\nabla-a$ is
invertible in $\Lk{M}$ for all $a\in k$. Hence, $\Sigma_\nabla\subset
\Ak\setminus k$. By Corollary~\ref{H1} we can reduce to the case where
$(M,\nabla-a)$ is pure for all $a\in k$. If $(M,\nabla)$ satisfies the
hypothesis of Theorem~\ref{sec:spectr-vers-youngs-9}, it is clear that
for each $l\in \NN$, all the radii of $(\forp^l)_*(M,\nabla-a)$ are
strictly less than $\omega.r(x^{p^l})$ for each $a\in k$. 

Now suppose that there exists $a\in k$ such that $\Rd{M_a}_1(x)\geq
\omega\cdot r(x)$. We prove the statement by contradiction. We assume
that for each $l\in \NN$ there exists $a_l\in k$ such that
$\Rd{M_{a_l}}\geq\omega^{\fra{p^l}} r(x)>\omega^{\fra{p^{l-1}}}r(x)$. By
Remark~\ref{sec:bf-frob-spectr-6}, there exists a unique (up to an
isomorphism) differential module $(N_{p^{l}},\nabla_{p^{l}})$ over $(\h{x^{p^{l}}},p^{l}S\d)$ such that $(M_{a_l},\nabla_{a_l})=(\forp^{l})^*(N_{p^{l}},\nabla_{p^{l}})$. Hence, we have $\Rd{N_{p^{l}}}_1=(\Rd{M_{a_l}}_1)^{p^l}\geq\omega r(x^{p^{l}})>\omega^pr(x^{p^l})$. By Remark~\ref{H6}, we have 
\begin{equation}\label{eq:H5}
    \nsp{\nabla_{p^{l}}}\leq \nor{p^{l}S\d}=|p|^{l}.
\end{equation}
Therefore we obtain
\begin{equation}\label{eq:H6}
\Sigma_{\nabla_{p^{l}}}\subset \disf{0}{|p|^{l}}.    
\end{equation}
Hence, by formula \eqref{eq:43}, for all $l\in \NN$ we have
\begin{equation}
    \label{H7}
    \Sigma_\nabla\subset \bigcup_{i=0}^{p^{l}-1}\disf{a_l+i}{|p|^{l}}.
\end{equation}
Consequently, we have
\begin{equation}
    \label{H8}
    \Sigma_\nabla\subset \bigcap_{l\in \NN}\bigcup_{i=0}^{p^{l}-1}\disf{a_l+i}{|p|^{l}}.
\end{equation}
This means in particular that $\Sigma_\nabla\subset k$, which contradict the fact that  $\Sigma_\nabla\subset \Ak\setminus k$.    
\end{proof}
\begin{rem}\label{sec:bf-spectr-comp-5}
 Lemma~\ref{sec:bf-spectr-comp-1} shows in particular that, if $(M_a,\nabla_a)$ is not solvable for all $a\in k$, then there exists $l\in \NN$ such that for all $a\in k$ and all $i$ we have $\Rd{M_a}_i(x)< \omega^{\fra{p^l}}r(x)$.
\end{rem}

\begin{Defi}
  Let $\Omega\in E(k)$. We define the following map
  \begin{equation}
    \label{eq:62}
    \Fonction{\delta_\Omega:\Omega}{\R+}{z}{\inf_{n\in \ZZ}|z-n|.}
  \end{equation}
  We will drop the subscription $\Omega$ when no confusion is possible.
\end{Defi}

\begin{Lem}\label{sec:bf-spectr-comp-6}
  Let $x=x_{0,r}$. Let $a\in k$, and $\cR_a(x)$ be the spectral radius of
  convergence of $(\Hx,S\d-a)$. We have
  \begin{equation}
    \label{eq:50}
    \cR_a(x):=
    \begin{cases}
      r& \text{ if } a\in \ZZ_p\\
      (\frac{|p|^l\omega}{\delta(a)})^{\fra{p^l}}r(x)& \text{
        if } |p|^{l}<\delta (a)\leq |p|^{l-1}, \; l\in \NN\setminus\{0\}\\
      \frac{\omega r(x)}{\delta(a)}& \text{ if } \delta(a)>1.
    \end{cases}
  \end{equation}
\end{Lem}
\begin{proof}
  Let $a\in k$. If $\delta(a)>1$ then $\delta(a)=|a|$. By
  Theorem~\ref{sec:bf-diff-polyn-1} and
  Proposition~\ref{sec:bf-like-between-3} we have
  $\cR_a(x)=\frac{\omega r(x)}{\delta(a)}$. Let
  $(M,\nabla)=(\forp)^l_*(\Hx,S\d-a)$ with $l\in \NN$. Then we have
  $(M,\nabla)=\bigoplus_{i=0}^{p^l}(\h{x^{p^l}},p^lS\d-a+i)$. If we assume that
  $a\in \ZZ_p$, then we have $\delta(a)=0$. This means that for each
  $l\in \NN$, there exists $i_0<p^l$ such that $|a-i_0|\leq
  |p|^l$. Since $|a-i_0|\leq
  |p|^l$, then $\Rd{(\h{x^{p^l}},p^lS\d-a+i_0)}_1(x^{p^l})\geq \omega
  r(x^{p^l})$. Therefore, $\max_i\Rm_i(x^{p^l})\geq\omega
  r(x^{p^l})$. Hence, for all $l\in \NN$ we have $\cR_a(x)\geq
  \omega^{\fra{p^l}}r(x)$ (cf. Corollary~\ref{H2}), and we deduce that
  $\cR_a(x)=r$. Now if we suppose that $|p|^{l}<\delta (a)\leq
  |p|^{l-1}$ for some $l\in\NN\setminus\{0\}$. Then there exists
  $i_0<p^l$ such that $\delta(a)=|a-i_0|$. On the other hand we have
  $\max_i\Rm_i(x^{p^l})=\Rd{(\h{x^{p^l}},p^lS\d-a+i_0)}_1(x^{p^l})=\frac{|p|^l\omega
    r(x)^{p^l}}{\delta(a)}<\omega r(x^{p^l})$ (cf. Theorem~\ref{sec:bf-diff-polyn-1} and
  Proposition~\ref{sec:bf-like-between-3}). Hence, by
  Corollary~\ref{H2} we obtain $\cR_a(x)=   (\frac{|p|^l\omega r(x)^{p^l}}{\delta(a)})^{\fra{p^l}}$.
\end{proof}

\begin{Pro}\label{sec:bf-spectr-comp-9}
Let $x=x_{0,r}\in\Ak\setminus k$. Let $(M,\nabla)$ be a  differential module over
  $(\Hx,S\d)$. For each $a\in k$ we set
  $(M_a,\nabla_a):=(M,\nabla-a)$. Assume that $(M_a,\nabla_a)$ is pure
  and non-solvable for
  all $a\in k$. Then there exists $z:=x_{c,\rho}\in \Ak\setminus k$
  such that:
  \begin{enumerate}
  \item $\Sigma_\nabla=\{z\}+\ZZ_p$;
  \item if $\rho>1$, then $\Sigma_\nabla=\{z\}$, and we have
    $\Rd{M_a}_1(x)=\min(\frac{\omega}{\rho}r(x),\cR_{c-a}(x))$;
  \item if $|p|^l<\rho<|p|^{l-1}$, then
    $\Sigma_\nabla=\{x_{c,\rho},\cdots, x_{c+p^l-1,\rho}\}$, and we
    have $\Rd{M_a}_1(x)=\min\left(\left(
        \frac{|p|^l\omega}{\rho}\right)^{\fra{p^l}}r(x),\cR_{c-a}(x)\right)$;
  \item if $\rho=|p|^l$, then $\Sigma_{\nabla}=\{x_{c,\rho},\cdots,
    x_{c+p^l-1,\rho}\}$, and we have
    $\Rd{M_a}_1(x)=\min(\omega^{\fra{p^l}}r(x),\cR_{c-a}(x))$;
  \item If $|p|^l<\rho\leq |p|^{l-1}$ (resp. $\rho>1$), let $P(d)$ be
    a differential polynomial associated to ${\forp^l}_*(M,\nabla)$
    (resp. $(M,\nabla)$), with $d=p^lS\d$. Let $P(T)$ be the
    commutative polynomial associated to $P(d)$, then the image by
    $\pik{\widehat{\h{x^{p^l}}}^{alg}}$ (resp. $\pik{\widehat{\Hx^{alg}}}$) of all roots
    of $P(T)$ is the set $\Sigma_\nabla$.
  \end{enumerate}
\end{Pro}

\begin{rem}
The values of the radius of convergence
in (2), (3) and (4) of Proposition~\ref{sec:bf-spectr-comp-9} can be
resumed by one formula as follows:
  \begin{equation}
      \label{eq:H2}
     \Rd{M_a}_1=\left(\frac{|p|^{l_a}\omega}{\delta(T(z)-a)}\right)^{\fra{p^{l_a}}}r(x),
   \end{equation}
 where $l_a:=\max(0,\lfloor
   \frac{\log (\delta(T(z)-a))}{\log |p|} \rfloor+1)$.
\end{rem}
\begin{proof}[Proof of proposition~\ref{sec:bf-spectr-comp-9}]
We set $x^{p^l}:=\forp^l(x)$. By Remark~\ref{sec:bf-spectr-comp-5} there exists $l\in \NN$ such
  that for all $a\in k$ we have $\Rd{M_a}_1(x)\leq
  \omega^{\fra{p^l}}r(x)$, we assume that $l$ is the smallest one
  verifying this inequality. Let $a\in k$ such that
  $\omega^{\fra{p^{l-1}}}r(x)<\Rd{M_a}_1(x)\leq
  \omega^{\fra{p^l}}r(x)$, without loss of generality we can suppose
  that $a=0$. This means that there exists a unique (up to
  isomorphism) $(N_{p^l},\nabla_{p^l})$
  such that $(M,\nabla)=(\forp^l)^*(N_{p^l},\nabla_{p^l})$ and
  $\Rm_1(x)^{p^l}=\Rd{N_{p^l}}_1(x^{p^l})$. Since $(M,\nabla)$ is pure
  then so is for $(N_{p^l},\nabla_{p^l})$.

  The idea is to first prove that $\Sigma_{\nabla_{p^l}}=\{z\}$, where
  $z\in\Ak\setminus k$.  Firstly, we need to prove that for each $a\in k$ the differential module
  $(N_{p^l},\nabla_{p^l}-a)$ is pure. Let $a\in k$, we proceed
  case by case.
  \begin{enumerate}
  \item {\bf Case where $\mathbf{\delta(a)\geq |p|^{l-1}}$}. In this case, the radius of
    $(\h{x^{p^l}},p^lS\d-a)$ is equal to
    $\frac{\omega r(x^{p^l})}{\delta(\frac{a}{p^l})}$
    (cf. Lemma~\ref{sec:bf-spectr-comp-6}). In particular, we have
    $\frac{\omega r(x^{p^l})}{\delta(\frac{a}{p^l})}\leq |p|\omega
    r(x)^{p^l}$. Consequently, all the radii of
    $(N_{p^l},\nabla_{p^l}-a)$ should be equal to
    $\frac{\omega r(x^{p^l})}{\delta(\frac{a}{p^l})}$. Hence,
    $(N_{p^l},\nabla_{p^l}-a)$ is pure with small radii.
    \item {\bf Case where $\mathbf{\delta(a)<|p|^{l-1}}$}. Since here the radius of $(\Hx,S\d-a)$ is
    strictly greater than $\omega^{\fra{p^{l-1}}}r(x)$
    (cf. Lemma~\ref{sec:bf-spectr-comp-6}), then
    $\omega^{\fra{p^{l-1}}}r(x)<\Rd{M_a}_1(x)\leq\omega^{\fra{p^{l}}}r(x)$.
    \begin{enumerate}
      \item{\bf If $\delta(a)=0$}, then there exists $i_0\in \NN$ such
        that $i_0<p^l$ such that $a-i_0$ is divisible by
        $p^l$. Hence, $\delta(\frac{a-i_0}{p^l})=0$, then $(N_{p^l},
        \nabla_{p^l}-a+i_0)$ has the same radii as
        $(N_{p^l},\nabla_{p^l})$, i.e. strictly greater than
        $|p|\omega r(x)^{p^l}$. Therefore,
        $(N_{p^l},\nabla_{p^l}-a+i_0)$ is the $l$-Frobenius antecedent
        of $(M_a,\nabla_a)$. Since
        $\delta(\frac{i_0}{p^l})=\frac{|i_0|}{|p|^l}\geq |p|\-1$. The
        radii of $(\h{x^{p^l}},p^lS\d+i_0)$ are less or equal to
        $|p|\omega r(x)^{p^l}$. Therefore, $(N_{p^l},\nabla_{p^l}-a)$ is
        pure with radii equal to $\cR_{\frac{i_0}{p^l}}(x^{p^l})$. 
    \item {\bf If $\delta(a)=|a-i_0|$, with $i_0\in\NN$ and $\mathbf{|i_0|\leq |p|^l}$}, then
      $\delta(\frac{a}{p^l})=\frac{\delta(a)}{|p|^l}$. Therefore the
      radius of $(\h{x^{p^l}},p^lS\d-a)$ is strictly greater than
      $|p|\omega r(x)^{p^l}$. Hence, the smallest radius of
      $(N_{p^l}\nabla_{p^l}-a)$ is strictly greater than
      $|p|\omega r(x)^{p^l}$. Consequently, $(N_{p^l},\nabla_{p^l}-a)$ is  the $l$-Frobenius
      antecedent of $(M_a,\nabla_a)$, which implies that it is
      pure.

      \item{\bf If $\delta(a)=|a-i_0|$, with $i_0\in\NN$ and $\mathbf{|i_0|>|p|^l}$}, then the $l$-Frobenius antecedent of
      $(M_a,\nabla_a)$ is $(N_{p^l},\nabla_{p^l}-a+i_0)$. This means
      that $(N_{p^l},\nabla_{p^l}-a+i_0)$ is pure. Since
      $(N_{p^l},\nabla_{p^l}-a)=(N_{p^l},(\nabla_{p^l}-a+i_0)-i_0)$,
      by the same argument as above all the radii of
      $(N_{p^l},\nabla_{p^l}-a)$ are small and should be equal to
      $\frac{\omega r(x^{p^l})}{\delta(\frac{i_0}{p^l})}$. 
    \end{enumerate}

  \end{enumerate}
  Then we
      conclude that for each $a\in k$, the differential module
      $(N_{p^l},\nabla_{p^l}-a)$ is pure, moreover 
      except the case (2), (b), the differential module has small radii.

  If we assume that for all $a\in k$, we have $\Rd{M_a}_1(x)<\omega^{\fra{p^l}}r(x)$, then for each $a\in k$, the differential module
      $(N_{p^l},\nabla_{p^l}-a)$ is pure with small radii. By
      Proposition~\ref{sec:bf-like-between-4} and Corollary~\ref{sec:bf-like-between-5} we have
      $\Sigma_{\nabla_{p^l}}=\{z\}$ with
      \begin{equation}
        \label{eq:63}
        \Rd{(N_{p^l},\nabla_{p^l}-a)}_1 (x^{p^l})=\frac{|p|^l\omega r(x)^{p^l}}{|T(z)-a|}.
      \end{equation}
In this case  we have $r(z)>|p|^l$.

  Now if we assume that there exists $a_0\in k$ such that 
  $\Rd{M_{a_0}}_1(x)=\omega^{\fra{p^l}}r(x)$. We set
  $(N,\nabla_N)=(\forp)_*(N_{p^l},\nabla_{p^l})$. Since
  $(N_{p^l},\nabla_{p^l}-a)$ is pure with radii less than or equal to
  $\omega r(x^{p^l})$ for each $a\in k$, $(N,\nabla_N-a)$ is pure with
  small radii for each $a\in k$
  (cf. Proposition~\ref{sec:bf-frob-spectr-4}). Hence, by
  Proposition~\ref{sec:bf-like-between-4},
  Corollary~\ref{sec:bf-like-between-5} and Proposition~\ref{sec:push-forw-spectr}, we have
  $\Sigma_{\nabla_{p^l}}=\Sigma_{\nabla_N}=\{z\}$ with
  \begin{equation}
    \label{eq:64}
       \Rd{(N,\nabla_N-a)}_1 (x^{p^l})=\frac{|p|^{l+1}\omega r(x)^{p^{l+1}}}{|T(z)-a|}.
     \end{equation}
     and by Proposition~\ref{sec:bf-frob-spectr-4} we get
     \begin{equation}
       \label{eq:65}
        \Rd{(N_{p^l},\nabla_{p^l}-a)}_1 (x^{p^l})=\frac{|p|^l\omega r(x)^{p^l}}{|T(z)-a|}.
      \end{equation}
Since
     $\Rd{(N_{p^l},\nabla_{p^l}-a)}_1(x^{p^l})\leq \omega r(x)^{p^l}$ for all
     $a\in k$, and  $\Rd{M_{a_0}}_1(x)=\omega^{\fra{p^l}}r(x)$ then we have $r(z)=|p|^l$. By formula~\eqref{eq:43}
     we obtain
     \begin{equation}
       \label{eq:68}
       \Sigma_\nabla=\{z+1,\cdots, z+p^l-1\}=\{z\}+\ZZ_p.
     \end{equation}
In the first cases i.e $r(z)>|p|^l$ , if $P(p^{l}S\d)$ is a differential polynomial of
$(\forp^{l})_*(M,\nabla)$ and $P(T)$ its associated commutative polynomial,
then $\Sigma_\nabla$ is the image of all the roots of $P(T)$ by
$\pik{\widehat{\h{x^{p^l}}^{alg}}}$. In the case where $r(z)=|p|^l$, if $P(p^{l+1}S\d)$ is a differential polynomial of
$(\forp^{{l+1}})_*(M,\nabla)$ and $P(T)$ its associated commutative polynomial,
then $\Sigma_\nabla$ is the image of all the roots of $P(T)$ by
$\pik{\widehat{\h{x^{p^{l+1}}}^{alg}}}$.

Since $r(z)\geq |p|^l$, then
$\delta(T(z)-a)=min\{|T(z)-a|,|T(z)-a+1|,\cdots,|T(z)-a+p^l-1|\}$. We have $(\forp^l)_*(M,\nabla-a)=\bigoplus_{i=0}^{p^l-1}(N_{p^l},\nabla_{p^l}-a+i)$
     (cf. \eqref{eq:42}) with maximal radius of convergence equal to
     \begin{equation}
       \label{eq:66}
       \cR=\frac{|p|^l\omega r(x)^{p^l}}{\delta(T(z)-a)}
     \end{equation}
     If $|p|^{i+1}\omega r(x)^{p^l}<\cR\leq |p|^{i}\omega r(x)^{p^l}$
     with $i\in \NN$,
     on the one hand, we have $|p|^{l-i}\leq
     \delta(T(z)-a)<|p|^{l-i-1}$. We set $l_a=\lfloor
   \frac{\log (\delta(T(z)-a))}{\log |p|} \rfloor+1=l-i$. On the other
   hand, 
     by induction, Proposition~\ref{sec:bf-frob-spectr-4} and Remark~\ref{H5}, we have
     \begin{equation}
       \label{eq:67}
       \Rm_1(x)^{p^{l_a}}=\frac{\cR}{|p|^ir(x)^{p^l-p^{l_a}}}=\frac{|p|^{l_a}\omega
         r(x)^{p^{l_a}}}{\delta(T(z)-a)},
     \end{equation}
     hence we obtain the desired result.
\end{proof}
\begin{cor}\label{sec:bf-spectr-comp-8}
  Let $x=x_{0,r}\in\Ak\setminus k$. Let $(M,\nabla)$ be a differential
  module over $(\Hx,S\d)$. For each $a\in k$ we set
  $(M_a,\nabla_a):=(M,\nabla-a)$. Assume that $(M_a,\nabla_a)$ is
  non-solvable for each $a\in k$. If there exists $z\in
  \Ak\setminus k$ such that $\Sigma_{\nabla}=\{z\}+\ZZ_p$, then
  $(M_a,\nabla_a)$ is pure for each $a\in k$ and we have  $\Rd{M_a}_1(x)=\left(\frac{|p|^{l_a}\omega
     }{\delta(T(z)-a)}\right)^{\fra{p^{l_a}}}r(x)$ with $l_a:=\max(0,\lfloor
   \frac{\log (\delta(T(z)-a))}{\log |p|} \rfloor+1)$.
  
\end{cor}

\begin{rem}
  Note that the condition $\Rd{M_a}_i(x)< r$ for all
$i$ and $a\in k$, excludes the case of a differential module with
regular singularities. However this does not mean that if $(M,\nabla)$
is not with regular singularities then $\Rd{M_a}_i(x)< r$. Indeed, for
example for
$x=x_{0,\omega}$, the differential module $(\Hx,\nabla)$ with $\nabla:=S\d-S$ is
solvable. 
\end{rem}

Now we treat the remaining case, the case where a differential module
(or its translate) admits a solvable radii.

\begin{Pro}\label{sec:bf-spectr-comp-4}
Let $x=x_{0,r}\in \Ak\setminus k$. Let $(M,\nabla)$ be a differential
module over $(\Hx,S\d)$ such that $\Rm_1(x)=r(x)$. Then we have
\begin{equation}
  \label{eq:54}
  \Sigma_{\nabla,k}(\Lk{M})=\ZZ_p.
\end{equation}
\end{Pro}

\begin{proof}
  We set $x^{p^l}=\forp^l(x)$. The proof is slightly like the proof of
  Proposition~\ref{sec:bf-case-positive}. Since  $\Rm_1(x)=r(x)$ by
  Remark~\ref{sec:bf-frob-spectr-6}, for each $l\in \NN$, there exists
  $(M_{p^l},\nabla_{p^l})$ a differential module over
  $(\h{x^{p^l}},p^lS\d)$ such that
  $(M,\nabla)=(\forp^l)^*(M_{p^l},\nabla_{p^l})$, moreover $\Rd{M_{p^l}}_1(x^{p^l})=r(x)^{p^{l}}=r(x^{p^l})$. By Formula
  \eqref{eq:43},
  we have
  \begin{equation}
    \label{eq:55}
    \Sigma_{\nabla,k}(\Lk{M})=\bigcup_{i=0}^{p^l-1}(\Sigma_{\nabla_{p^l},k}(\Lk{M_{p^l}})+i).
  \end{equation}
  Since $(M_{p^l},\nabla_{p^l})$ is a solvable differential module over
  $(\h{x^{p^l}},p^l S\d)$ and $\nor{S\d}=|p|^l$, by
 Remark~\ref{H6}, we have
  \begin{equation}
    \label{eq:56}
    \nsp{\nabla} \leq |p|^l.
  \end{equation}
Consequently, we deduce that
$\Sigma_{\nabla_{p^l},k}(\Lk{M_{p^l}})\subset \disf{0}{|p|^l}$. Then we
obtain
\begin{equation}
  \label{eq:57}
\Sigma_{\nabla,k}(\Lk{M})\subset\bigcap_{l\in \NN}\bigcup_{i=0}^{p^l-1} \disf{i}{|p|^l}=\ZZ_p.
\end{equation}

As $\Sigma_{\nabla,k}(\Lk{M})\ne \emptyset$, there exists $a\in
\ZZ_p\cap \Sigma_{\nabla,k}(\Lk{M})$. Then we have $a+\ZZ\subset
\Sigma_{\nabla,k}(\Lk{M})$. Since the spectrum is compact, we have
$a+\ZZ_p=\ZZ_p\subset \Sigma_{\nabla,k}(\Lk{M})$. Then the result follows. 
\end{proof}

\begin{cor}\label{sec:bf-spectr-comp-7}
 Let $x=x_{0,r}\in \Ak\setminus k$. Let $(M,\nabla)$ be a differential
module over $(\Hx,S\d)$. If $\Sigma_\nabla=a+\ZZ_p$ with $a\in k$ then $(M,\nabla)$
is pure with radius equal to $\cR_a(x)$. 
\end{cor}

\begin{cor}\label{sec:main-result}
   Let $x=x_{0,r}\in \Ak\setminus k$. Let $(M,\nabla)$ be a differential
module over $(\Hx,S\d)$. If for all $a\in k$, $(M,\nabla-a)$ is pure,
then any sub-quotient of $(M,\nabla)$ has the same spectrum as $(M,\nabla)$.
\end{cor}

\begin{nota}
  Let $S\subset\Ak$, we denote by $S/\ZZ_p$ the quotient set of $S$ by the
  equivalence relation:
  \begin{equation}
    \label{eq:69}
    x\sim x'\iff\exists n\in \ZZ_p;\; x=x'+n.
  \end{equation}
\end{nota}

\begin{rem}
  Note that if $z,\,z'\in\Ak$ and $z\sim z'$, then we have
  $\delta(T(z)-a)=\delta(T(z')-a)$ for all $a\in k$.
\end{rem}

Now we can announce the main theorem of the paper. 

\begin{Theo}\label{sec:bf-spectr-comp-3}

  Assume that $\crk=p>0$ and $x:=x_{0,r}\in \Ak\setminus k$. Let
  $(M,\nabla)$ be a differential module $(\Hx, S\d)$. We denote $\forp^l(x)$ by $x^{p^l}$.
\begin{itemize}

\item There exist $z_1,\cdots,z_{\nu}\in \Ak\setminus k$ and
  $a_1,\cdots, a_\mu\in k$, such that
  \[\Sigma_{\nabla,k}(\Lk{M})=\{z_1,\cdots,z_\nu, a_1,\cdots,
    a_\mu\}+\ZZ_p,\]
 where $z_i$ has the same type as $x$, and $(\nu,\mu)$ is not equal to
 $(0,0)$.
\item We can choose $z_i$ and $a_j$ such that the set $\{z_1,\cdots,z_\nu, a_1,\cdots,
    a_\mu\}$ has minimal cardinality. Indeed it is enough to keep only
    $z_i$ and $a_j$, for which we have $\{z_i\}+\ZZ_p\cap
    \{z_{i'}\}+\ZZ_p=\emptyset$ and $\{a_j\}+\ZZ_p\cap \{a_{j'}\}+\ZZ_p=\emptyset$
    for $i\ne i'$ and $j\ne j'$.
  
  \item We choose $\{z_1,\cdots,z_\nu, a_1,\cdots,
    a_\mu\}$ to be minimal. Then we have a unique (up to an isomorphism) decomposition
    \begin{equation}\label{eq:58}
      (M,\nabla)=\bigoplus_{i=1}^{\nu}(M_{z_i},\nabla_{z_i})\oplus \bigoplus_{j=1}^{\mu}(M_{a_j},\nabla_{a_j}),
    \end{equation}
    such that, $\Sigma_{\nabla_{z_i},k}(\Lk{M_{z_i}})=\{z_i\}+\ZZ_p$
    and $\Sigma_{\nabla_{a_j},k}(\Lk{M_{a_j}})=\{a_j\}+\ZZ_p$.
    \item Let $c_i\in k$ and $r_i>0$ such that $z_i=x_{c_i,r_i}$. If
    $|p|^{l}\leq r_i< |p|^{l-1}$, with $l\in \NN\setminus\{0\}$, then $\car(\{z_i\}+\ZZ_p)=p^l$ and $\{z_i\}+\ZZ_p=\{x_{c_i,r_i},
    x_{c_i+1,r_i},\cdots, x_{c_i+p^l-1,r_i}\}$. If $r_i\geq1$ then we
    have $\car(\{z_i\}+\ZZ_p)=1$ and
    $\{z_i\}+\ZZ_p=\{x_{c_i,r_i}\}$.
  \item If $r_i>1$, let $P_{z_i}(S\d)$ be a differential polynomial
    associated to $(M_{z_i},\nabla_{z_i})$. Then the image by
    $\pi_{\widehat{\Hx^{alg}}/k}$ of all roots of $P_{z_i}(T)$ (the
    commutative polynomial associated to $P_{z_i}(S\d)$) is equal to $z_i$. 
  \item If $|p|^l<r_i\leq|p|^{l-1}$, let $P_{z_i}(p^lS\d)$ be a differential
    polynomial associated to $(\forp^l)_*(M,\nabla)$ (as a differential
    module over $(\h{x^{p^l}},p^lS\d)$). Then the image by
    $\pi_{\widehat{\h{x^{p^l}}^{alg}}/k}$ of all roots of $P_{z_i}(T)$
    (the commutative polynomial associated to $P_{z_i}(p^lS\d)$) is equal to $\{x_{c_i,r_i},
    x_{c_i+1,r_i},\cdots, x_{c_i+p^l-1,r_i}\}$. In the special case
    where $r_i=|p|^{l-1}$ we have $\{x_{c_i,r_i},
    x_{c_i+1,r_i},\cdots, x_{c_i+p^l-1,r_i}\}=\{x_{c_i,r_i},
    x_{c_i+1,r_i},\cdots, x_{c_i+p^{l-1}-1,r_i}\}$.

  \item If $r_i\geq 1$. For all $a\in k$, the differential module $(M_{z_i},\nabla_{z_i}-a)$
    is pure. For $a\in \disf{c_i}{r_i}\cap k$ we have
    \begin{equation}\label{eq:70}
      \Rd{(M_{z_i},\nabla_{z_i}-a)}_1(x)=\frac{\omega}{r_i},
    \end{equation}
    and for all $a\in k\setminus \disf{c_i}{r_i}$
    \begin{equation}\label{eq:71}
      \Rd{(M_{z_i},\nabla_{z_i}-a)}_1(x)=\frac{\omega}{|a-c_i|}.
    \end{equation}
 \item If $|p|^l\leq r_i<|p|^{l-1}$. For all $a\in k$, the differential module $(M_{z_i},\nabla_{z_i}-a)$
    is pure. We have, for all $a\in \bigcup_{j=0}^{p^l-1} \disf{c_i+j}{r_i}\cap k$,
    \begin{equation}\label{eq:75}
      \Rd{(M_{z_i},\nabla_{z_i}-a)}_1(x)=\left(\frac{|p|^l\omega}{r_i}\right)^{\fra{p^l}},
    \end{equation}
    and for all $a\in k\setminus \bigcup_{j=0}^{p^l-1} \disf{c_i+j}{r_i}$
    \begin{equation}\label{eq:76}
      \Rd{(M_{z_i},\nabla_{z_i}-a)}_1(x)=\cR_{a-c_i}(x).
    \end{equation}
  \item For all $a\in k$, the differential module
    $(M_{a_i},\nabla_{a_i}-a)$ is pure. More precisely, for all $a\in
    \{a_i\}+\ZZ_p$, $(M_{a_i},\nabla_{a_i}-a)$ is solvable, and for all $a\in
    k\setminus \{a_i\}+\ZZ_p$, we have $\Rd{(M_{a_i},\nabla_{a_i}-a)}_1(x)=\cR_{a-a_i}(x)$.  
\end{itemize}
\end{Theo}

\begin{proof}
By Propositions~\ref{sec:bf-spectr-comp-9},
\ref{sec:bf-spectr-comp-4}, Corollaries~\ref{H1},
\ref{sec:bf-spectr-comp-8} and \ref{sec:bf-spectr-comp-7}, there
exists a decomposition 
   \begin{equation}
     (M,\nabla)=\bigoplus_{i=1}^\nu(M'_{w_i},\nabla'_{w_i}),
   \end{equation}
   with $w_1,\cdots,
  w_\mu\in\Ak$, such that 
   $(M'_{w_i},\nabla'_{w_i}-a)$ is pure for
   each $a\in k$ and 
   $\Sigma_{\nabla'_{w_i}}(\Lk{M'_{w_i}})=\{w_i\}+\ZZ_p$. Also, for
   each $\omega_i$ , $\Rd{(M'_{w_i},\nabla'_{w_i}-a)}_1(x)$ satisfies equations \eqref{eq:70}, \eqref{eq:71}, \eqref{eq:75} and \eqref{eq:76}. We set
   $(M_{w_i},\nabla_{w_i}):=\bigoplus_{w_j\sim
     w_i}(M'_{w_j},\nabla'_{w_j})$ (cf. \eqref{eq:69}). Let
   $\{z_1,\cdots, z_\mu\}\subset \{w_1,\cdots, w_\nu\}$ be the maximal
   subset such that $(\{z_i\}+\ZZ_p)\cap
   (\{z_j\}+\ZZ_p)=\emptyset$. Then we have $
   \Sigma_{\nabla,k}(\Lk{M})= \{z_1,\cdots, z_\mu\}+\ZZ_p$ and the
   decomposition
   \begin{equation}\label{eq:74}
   (M,\nabla)=\bigoplus_{i=1}^\mu(M_{z_i},\nabla_{z_i}),
 \end{equation}
   satisfies the properties of the theorem. It remains to prove the
   uniqueness up to isomorphisms of such decomposition. Let us proof
   that $(M_{z_i},\nabla_{z_i})$ are uniquely determined up to
   isomorphisms. Suppose that there exists another decomposition satisfying
   the same properties of the theorem
   \[ (M,\nabla)=\bigoplus_{i=1}^\mu(M'_{z_i},\nabla'_{z_i}),\]
Let $\imath_{z_i}:M_{z_i}\to M$ (resp. $\imath'_{z_i}:M'_{z_i}\to M$) be
the canonical injection and $\pi_{z_i}:M\to M_{z_i}$
(resp. $\pi'_{z_i}:M\to M'_{z_i}$) be the canonical projection. Then
for all $i\ne j$, we have $\imath_{z_i}\circ\pi'_{z_j}=0$ and
$\imath'_{z_i}\circ\pi_{z_j}=0$, because otherwise we will get  a
sub-quotient of $M_{z_i}$ (resp. $M'_{z_j}$) with spectrum equal to
$\{z_j\}+\ZZ_p$ (resp. $\{z_i\}+\ZZ_p$) which is absurd (cf. Corollary~\ref{sec:main-result}). This means that  $\imath_{z_i}\circ\pi'_{z_i}$ and
$\imath'_{z_i}\circ\pi_{z_i}$ are injective, hence
$(M_{z_i},\nabla_{z_i})\simeq (M'_{z_i},\nabla'_{z_i})$.
\end{proof}

\begin{cor}\label{sec:bf-spectr-comp-11}
Let $x=x_{0,r}\in\Ak\setminus k$. Let $(M,\nabla_M)$  and $(N,\nabla_N)$ be a differential module over
  $(\Hx,S\d)$. Let $
  (M,\nabla_M)=\bigoplus_{i=1}^\mu(M_{z_i},\nabla_{z_i})$ (resp.
  $(N,\nabla_N)=\bigoplus_{i=1}^\mu(N_{w_i},\nabla_{w_i})$) be a
  decomposition as in \eqref{eq:71}. Then for $(\{z_i\}+\ZZ_p)\cap
  (\{w_j\}+\ZZ_p)=\emptyset$ we have
  $\text{\rm Hom}_{\DD_{\Hx}}(M_{z_i},N_{w_j})=0$.  
\end{cor}

\begin{cor}\label{sec:bf-spectr-comp-10}
Let $x=x_{0,r}\in\Ak\setminus k$. Let $(M,\nabla)$ be a differential module over
  $(\Hx,S\d)$. If $(M_1,\nabla_1)$ and $(M_2,\nabla_2)$ are two
  differential modules over $(\Hx,S\d)$ such that
  \begin{equation}
    \label{eq:72}
    0\to (M_1,\nabla_1) \to (M,\nabla) \to (M_2,\nabla_2)\to 0,
  \end{equation}
  then we have $\Sigma_\nabla(\Lk{M})=\Sigma_{\nabla_1}(\Lk{M_1})\cup \Sigma_{\nabla_2}(\Lk{M_2})$.
\end{cor}

\begin{proof}
By Corollary~\ref{sec:bf-spectr-comp-11}, we have
    $\Sigma_{\nabla_1}(\Lk{M_1})\subset \Sigma_{\nabla}(\Lk{M})$. The
    result follows by
    Lemma~\ref{sec:spectr-vers-youngs-1}.
  \end{proof}
\begin{rem}
  As a direct consequence of Corollary~\ref{sec:bf-spectr-comp-10}, if
  $(N_i,\nabla_i)$ are the quotient differential module of a
  Jordan-Hölder sequence of  $(M,\nabla)$, then we have
  \begin{equation}
    \label{eq:73}
    \Sigma_\nabla(\Lk{M})=\bigcup_i\Sigma_{\nabla_i}(\Lk{N_i}).
  \end{equation}
\end{rem}

\subsection{Spectrum of a differential equation at a point of a
  quasi-smooth curve}\label{sec:spectr-diff-equat}

Now let us say some words concerning the spectrum of a differential
module defined over a point of a quasi-smooth curve. Consider a
quasi-smooth curve $\cC$ and a differential
equation $(\F,\nabla)$ (i.e $\F$ is a locally free $\cO_\cC$-module of finite rank together with a
 connection $\nabla:\F\to \F\ot_{\cO_\cC}\Omega_{\cC/k}$). Let $x$ be
 a point of type $2$ or $3$ of $\cC$. In order to compute the spectrum
 of $(\F_x\otimes_{\cO_{\cC}} \Hx,\nabla)$ we need to fix a bounded
 derivation $d$ defined in a neighbourhood of $x$. Recall that there
 exists an affinoid neighbourhood $Y$ of $x$ and an affinoid domain
 $X\subset \Ak$ such that there exists an étale map $\phi: Y\to X$,
 where we can assume that $\phi(x)=x_{0,r}$ and $0\notin X$ (cf. \cite[Theorem~3.12]{np2}). Since $\phi$ is
 étale we have the isomorphism of sheaves $\phi^*\Omega_{X/k}\simeq
 \Omega_{Y/k}$ (cf. \cite[Proposition~3.5.3]{Ber}). We fix a coordinate
 function $S$ on $X$. Then one of the most suitable derivation $d:
 \cO_Y\to \cO_Y$ is
 the one corresponding to the $\cO_Y$-morphism $\Omega_{Y/k}\to
 \cO_Y$, $S\-1 dS\otimes 1 \mapsto 1$, which by construction extends
 $S\d: \cO_X\to \cO_X$. Therefore, $(\Hx,d)$ is an extension of
 $(\h{x_{0,r}},S\d)$ as a differential  field. Then we consider $(\F_x\otimes_{\cO_{\cC}}
 \Hx,\nabla)$ as a differential module over $(\Hx,d)$. By a
 restriction of scalars, we can see $(\F_x\otimes_{\cO_{\cC}}
 \Hx,\nabla)$ as a differential module over $(\h{x_{0,r}},S\d)$, which by
 Proposition~\ref{sec:push-forw-spectr}, does not affect the spectrum of
 $\nabla$. Hence, we can use our previous results to compute the
 spectrum of $\nabla$.
\begin{Theo}\label{sec:spectr-diff-equat-1}
  Assume that $\crk=p>0$. Let $\cC$ be a quasi-smooth curve and $x\in
  \cC$ of type (2) or (3). Let $(M,\nabla)$ be a differential module over
  $(\Hx,d)$, where $d=\psi^*(S\d)$, $\psi$ is a finite étale morphism $\psi$ from a
  nighbourhood of $x$ to $\Pk$, with $\psi(x)=x_{0,r}$. Then there exist $z_1,\cdots,
  z_\mu\in\Ak$, with $(\{z_i\}+\ZZ_p)\cap
  (\{z_j\}+\ZZ_p)=\emptyset$ for $i\ne j$, such that: 
  \begin{equation}
     \Sigma_{\nabla,k}(\Lk{M})= \{z_1,\cdots,
     z_\mu\}+\ZZ_p.
   \end{equation}
 \end{Theo}

\section{Appendix}
  \subsection{Analytic spectrum of a Banach ring}

In this part, we refer the reader to \cite[Chapter 1]{Ber} for more
details and proofs. To a commutative Banach ring $(A,\nor{.})$, we can associate
an analytic spectrum, denoted by $\cM(A)$, as the set of multiplicative
semi-norms bounded by $\nor{.}$. This set is not empty if $A$ is a
nonzero ring.  We  endow $\cM(A)$ with the initial topology with respect
        to the map 
        \begin{equation}\Fonction{\psi:\q \cM(A)}  {\R+^A} {x}{ (|f(x)|)_{f\in
              A}}.\end{equation}
For this topology, $\cM(A)$ is a compact Hausdorff space.        

To a point $x$ of $\cM(A)$ we can associate a residue field as
follows. The point $x$ is associated to a multiplicative semi-norm
$|.|_x$. The set:
\begin{equation}\mfp_x=\{f\in A;\; |f|_x=0\}\end{equation}
is a prime ideal of $A$. Therefore, the semi-norm $|.|_x$ extends to a
multiplicative norm on the fraction field of $\text{Frac}(A/\mfp_x)$. We
will denote by $\Hx$ the completion of $\text{Frac}(A/\mfp_x)$ with respect
to $|.|_x$, and by $|.|$ the valuation on $\Hx$
induced by $|.|_x$.

We have the
natural bounded ring morphism:
\begin{equation}\chi_x: (A,\nor{.})\to (\Hx,|.|).\end{equation}
For all $f\in A$, write $f(x)$ instead of $\chi_x(f)$.

\begin{rem}\label{sec:analyt-spectr-ring}
          The map $x\mapsto \chi_x$ identifies $\cM(A)$ with the set of equivalence
          classes of characters of $A$.
        \end{rem}

 \begin{cor}\label{sec:topology-cma-3}
          An element $f\in A$ is invertible if and only if $f(x)\not=
          0$ for all $x\in \cM(A)$.
        \end{cor}

        Let $(A,\nor{.})$ and $(B,\nor{.}')$ be two commutative Banach rings, let
$\varphi:\q (A,\nor{.}) \to (B,\nor{.}')$ be a bounded morphism of
rings. Then $\phi$ induces a continuous map 
defined as follows: 
\begin{equation}
  \label{eq:4a}
  \Fonction{\phi^*:\cM(B)}{\cM(A)}{x}{f\mapsto|\phi(f)(x)|.}
\end{equation}

\begin{Defi}
  Let $(A, \nor{.})$ be a normed ring. The {\em spectral semi-norm} associated to
  $\nor{.}$ is the map:
   \begin{equation}
    \label{eq:5}
    \Fonction{\Nsp{A}{.}:\q A}{ \R+} {f}{
      \Lim_{n\to +\infty}\nor{f^n}^{\frac{1}{n}}}.
  \end{equation}
  In the case where for all $f\in A$ we have $\nor{f^n}=\nor{f}^n$, (i.e. $\nor{.}=\nsp{.}$), we say that $(A,\nor{.})$ is {\em uniform}. 
\end{Defi}
        The spectral semi-norm defined on a commutative Banach ring $(A,\nor{.})$
       (c.f \eqref{eq:5}) satisfies the following property:

        \begin{prop}
          For all element $f$ in $A$ we have:
          \begin{equation}\Nsp{A}{f}=\max_{x\in \cM(A)}|f(x)|. \end{equation}
        \end{prop}

        \begin{cor}\label{sec:topology-cma-2}
        Let $A$ be a commutative Banach ring. Then the spectral
        semi-norm satisfies:
        \begin{itemize}
        \item $\forall f$, $g\in A$; $\nsp{fg}\leq \nsp{f}\cdot\nsp{g}$.
        \item $\forall f$, $g\in A$; $\nsp{f+g}\leq \nsp{f}+\nsp{g}$.
        \end{itemize}
        \end{cor}

        \begin{Lem}\label{sec:topology-cma-1}
          Let $(A,\nor{.})$ be a Banach ring, let $(B,\nor{.}')$ and
          $(C,\nor{.}'')$ be two Banach $A$-algebras. Let $f\in
          B\ct_AC$. Then $f$ is not invertible in $B\ct_AC$ if and only if
          there exists $x\in \cM(C)$ such that the image of $f$ by the
          natural map
          $B\ct_AC\to B\ct_A\Hx$ is not invertible. 
        \end{Lem}

        \begin{proof}
          It is obvious that if the image of $f$ is not invertible in $B\ct_A\Hx$
          for some $x\in \cM(C)$, then so is for $f$ in $B\ct_A C$. We suppose now
          that $f$ is not invertible in $B\ct_A C$. By Corollary~\ref{sec:topology-cma-3} there exists $z\in
          \cM(B\ct_AC)$ such that $f(z)=0$ in $\Hz$. We have the following
         commutative diagram:

\begin{equation}\xymatrix{ & B\ct_A C\ar[dr]& \\  C\ar[ru]\ar[rr]& & \Hz\\ }\end{equation}
By Remark \ref{sec:analyt-spectr-ring} there exists $x\in \cM(C)$
such that we have the following diagram:

\begin{equation} \xymatrix{ C\ar[r]\ar[d]& B\ct_A C\ar[d]\\ \Hx\ar[r] &\Hz \\}\end{equation} 
Therefore, we obtain the commutative diagram
\begin{equation}\xymatrix{ & B\ct_A C\ar[rd]\ar[rrd]& & \\ C\ar[ru]\ar[rd] & &
    B\ct_A\Hx\ar[r]&\Hz\\ &\Hx\ar[ru]\ar[rru]& & \\}.\end{equation}
Then,  $f(z)=0$ implies that the image of $f$ in $B\ot_A\Hx$ is not
invertible.
\end{proof}

\begin{cor}

  Let $(A,\nor{.})$ be a Banach ring, let $(B,\nor{.}')$ and
          $(C,\nor{.}'')$ be two Banach $A$-algebras. Let $f\in
          B\ct_AC$. If $f$ is not invertible in $B\ct_AC$ then
          there exists $x\in \cM(C)$ and $y\in \cM(B)$ such that the image of $f$ by the
          natural map
          $B\ct_AC\to \Hy\ct_A\Hx$ is not invertible. 
        \end{cor}

\subsection{Description of some \'etale morphisms}
\label{App}
        
        This part is devoted to summarize the properties of the following
        analytic morphisms: logarithm and power map.

        \subsubsection{\bf Logarithm}\label{sec:logarithm-2}

        Let $a\in k\setminus\{0\}$. We define the logarithm function  $\loga: \diso{a}{|a|} \to
        \Ak$ to be the analytic map associated to the ring morphism:

        \begin{equation}\Fonction{k[T]}{\fdiso{a}{|a|}}{ T}{ \sum_{n\in
              \NN\setminus\{0\}}\frac{(-1)^{n-1}}{a^n\cdot n}(T-a)^n}.\end{equation}
      We define the exponential function $\exp_a:\diso{0}{\omega}\to
      \diso{a}{|a|\omega}$ to be the analytic map associated to the
      ring morphism:
      \begin{equation}\Fonction{\fdiso{a}{|a|\omega}}{\fdiso{0}{\omega}}{\frac{T-a}{a}}{\sum_{n\in \NN}\frac{T^n}{n!}},\end{equation}
where \begin{equation} \omega= \begin{cases}
	|p|^{\frac{1}{p-1}} &\text{if char}(\tilde{k})=p, \\ 1 &
          \text{if char}(\tilde{k})=0.
	\end{cases}\end{equation}
      
      \begin{Lem}\label{sec:logarithm-3}
        Let $b\in \diso{a}{|a|}\cap k$. Then we have $\Log{b}=\Log{a}-\loga(b)$.
      \end{Lem}
      \begin{proof}
        We have $\Log{a}(T)=\Log{1}(\frac{T}{a})$, in particular
        $\Log{1}(\frac{b}{a})$ is well defined. Therefore,
        \begin{equation}\Log{b}(T)=\Log{1}\left(\frac{T}{b}\right)=\Log{1}\left(\frac{\frac{T}{a}}{\frac{b}{a}}\right)=\Log{1}\left(\frac{T}{a}\right)-\Log{1}\left(\frac{b}{a}\right)=\Log{a}(T)-\loga(b). \qedhere
        \end{equation}
      \end{proof}
      
      \begin{Lem}\label{sec:logarithm-4}
The logarithm function $\loga$ induces an
        analytic isomorphism $\diso{a}{|a|\omega}\to
        \diso{0}{\omega}$, whose reverse isomorphism is $\exp_a$. 
      \end{Lem}
      \begin{proof}
        Since $\exp_a:\diso{0}{\omega}\to \diso{a}{|a|\omega}$ is surjective, we
        obtain the isomorphism.
      \end{proof}

      \begin{Lem}\label{sec:logarithm-5}
        Assume that $k$ is algebraically closed and $\crk=p>0$. Let
        $\zeta_{p^n}$ be a $p^n$th root of unity. Then we have
        \begin{itemize}
        \item $|\zeta_{p^n}-1|=\omega^{\fra{p^{n-1}}}$;
        \item if $x\in \diso{a}{|a|}\cap k$, then $\loga(x)=0$ if and
          only if $x=a\zeta_{p^n}$.
        \end{itemize}
      \end{Lem}

      \begin{proof}
        It is easy to see that $\loga(a\zeta_{p^n})=0$. Indeed,
        \begin{equation}\loga(a\zeta_{p^n})=\Log{1}(\zeta_{p^n})=\fra{p^n}\Log{1}(1)=0.\end{equation}
        Since $\loga(T)=\Log{1}(\frac{T}{a})$, it is enough to show
        that $\Log{1}(x)=0$ if and only if $ x=\zeta_{p^n}$. By
        Lemma~\ref{sec:logarithm-4}, we have  $\Log{1}(x)=0$ if and
        only if $x=1$, if $x\in \diso{1}{\omega}\cap k$. Now let $x\in \diso{1}{1}\cap k$,
        we have
        \begin{equation}x^p-1=(x-1)^p+p(x-1)\sum_{i=1}^{p-1}p\-1\binom{p}{i}(x-1)^{i-1}.\end{equation}
        Therefore,
        \begin{equation}|x^p-1|\leq \max(|p||x-1|,|x-1|^p).\end{equation}
        Hence, there exists $n\in \NN$ such that
        $|x^{p^n}-1|<\omega$. Since $\Log{1}(x^{p^n})=p^n\Log{1}(x)$,
        if $\Log{1}(x)=0$ then $x^{p^n}=1$. The result follows.  
      \end{proof}
        \begin{prop}\label{sec:logarithm}
          Assume that $k$ is algebraically closed and $\crk=p>0$. Let $\disf{b}{r}$ be a  closed sub-disk of  $\diso{a}{|a|}$. Then:
          \begin{itemize}
          \item The logarithm function induces an étale cover $\diso{a}{|a|}\to \Ak$. 
          \item $\loga(\disf{b}{r})=\disf{\loga(b)}{\varphi(r)}$ where
            \begin{equation}\Fonction{\varphi:[0,|a|)}{\R+}{r}{|\Log{b}(x_{b,r})|}.\end{equation} 
          \item The function $\phi$ depends only on the choice of the radius of the
            disk $\disf{b}{r}$. In particular, it is an increasing
            continuous function and piecewise
            logarithmically affine on $[0,|a|)$ and 
          \item  $\varphi(|a|\omega^{\frac{1}{p^{n}}})=\frac{\omega}{|p|^{n}}$, where $n\in \NN$.
          \item If $|a|\omega^{\frac{1}{p^{n-1}}}\leq r
            <|a|\omega^{\frac{1}{p^{n}}}$, then
            $\loga\-1(\loga(b))\cap\disf{b}{r}=\{b\zeta_{p^n}^i| 0\leq i\leq p^n-1\}$ where
            $\zeta_{p^n}$ is a $p^n$th root of unity.
            
          \end{itemize}
        \end{prop}

        \begin{proof}~\\
          \begin{itemize}
          \item Since $\dT\Log{a}(T)=\frac{1}{T}$ is invertible in
            $\fdiso{a}{|a|}$, by Remark~\ref{sec:sheaf-diff-etale-2}
            $\loga$ is locally étale. Hence, it is an étale cover.
      \item We know that the image of the disk $\disf{b}{r}$ by the analytic
        map $\loga$ is the disk $\disf{\loga(b)}{\phi(r)}$,
        with radius equal to
        $\phi(r)=|\loga(x_{b,r})-\loga(b)|$. By
        Lemma~\ref{sec:logarithm-3} we obtain
        $\varphi(r)=|\Log{b}(x_{b,r})|$.
      \item Since $|b|=|a|$ and $\varphi(r)=|\Log{b}(x_{b,r})|$, by construction it depends only on the
        value $r$. Since $\Log{b}$ is an analytic map well defined on
        $(b,x_{a,|a|})$, the map $\phi$ is an increasing continuous
        function piecewise logarithmically affine on $[0,|a|)$.
      \item We have $\phi(|a|\omega^{\fra{p^n}})=\max_{i\in \NN\setminus\{0\}}|i|\-1\omega^{\frac{i}{p^n}}=\frac{\omega}{|p^n|}$.
        
      \item Since $\Log{b}=\loga-\loga(b)$, we conclude by Lemma~\ref{sec:logarithm-5}.\qedhere 
          \end{itemize}
        \end{proof}

        \begin{Pro}\label{sec:logarithm-1}

          Assume that $\crk=p>0$. Let $y\in\diso{a}{|a|}$ and $x=\loga(y)$,
          then we have:
          \begin{itemize}
            \item If $0<r_k(y)<|a|\omega$, then $[\Hy :\Hx]=1$
            \item If $|a|\omega^{\frac{1}{p^{n-1}}}\leq r_k(y)
              <|a|\omega^{\frac{1}{p^{n}}}$ with
              $n\in\NN\setminus\{0\}$, then $[\Hy:\Hx]=p^n$.              
            \end{itemize}
          \end{Pro}

          \begin{proof}
            It is a consequence of Propositions~\ref{sec:logarithm} and \ref{sec:sheaf-diff-etale-1}.  
          \end{proof}

      \subsubsection{\bf Power map}\label{sec:power-map}

       For the details of this part we refer the reader for example to
       \cite[Section 5]{and} and \cite[Chapter 10]{Ked}. We define the
       $n$th power map $\Delta_n: \Ak\to \Ak$ to be the analytic map
      associated to the ring morphism:
       \begin{equation}\Fonction{ k[T]}{k[T]}{ T}{T^n }\end{equation}

      \paragraph{\bf Frobenius map}\label{sec:frobenius-map}
      We assume here that $\crk=p$,  with
      $p>0$.\\
      We define the Frobenius map to be the $p$th power map. We will
      denote it by $\forp$.
        \begin{prop}\label{sec:forbenius-map}
          Let $a\in k$ and $r\in \R+^*$. The Frobenius map satisfies
          the following properties:
          \begin{itemize}
          \item It induces an finite étale morphism $\Ak\setminus \{0\} \to
            \Ak \setminus\{0\}$.
          \item $\forp(\disf{a}{r})=\disf{a^p}{\phi(a,r)}$ where
              $\phi(a,r)=\max(|p||a|^{p-1}r, r^p)$.
            \item $\forp (x_{a,r})=x_{a^p,\phi(a,r)}$.             
          \end{itemize}
        \end{prop}

        \begin{Pro}\label{sec:forbenius-map-1}
        Let $y:=x_{a,r}$ with $r>0$. We set $x=\forp (y)$. Then we have:
        \begin{itemize}
        \item If $r<\omega |a|$, then $[\Hy:\Hx]=1$.
        \item If $r\geq\omega |a|$, then $[\Hy:\Hx]=p$.
          
        \end{itemize}
        
      \end{Pro}

\begin{cor}                     
        Let $y:=x_{0,r}$, with $r>0$. Let $n\in\NN\setminus\{0\}$, we
         set $x=(\forp )^n(y)$. Then we have $[\Hy :\Hx]=p^n$.
     \end{cor}

      \paragraph{\bf Tame case}\label{sec:tame-case-2}
      Let $n\in \NN\setminus\{0\}$. We assume that $n$ is coprime to $\crk$. 
      
\begin{prop}\label{sec:tame-case}
          Let $a\in k$ and $r\in \R+^*$. The $n$th power map satisfies
          the following properties:
          \begin{itemize}
          \item It induces a finite étale morphism $\Ak\setminus \{0\} \to
            \Ak \setminus\{0\}$.
          \item $\Delta_n(\disf{a}{r})=\disf{a^n}{\phi(a,r)}$ where
              $\phi(a,r)=\max(|a|^{n-1}r, r^n)$.
            \item $ \Delta_n(x_{a,r})=x_{a^n,\phi(a,r)}$.             
          \end{itemize}
        \end{prop}

        \begin{Pro}\label{sec:tame-case-1}
        Let $y:=x_{a,r}$ with $r>0$. We set $x=\Delta_n(y)$, then we have:
        \begin{itemize}
        \item If $r< |a|$, then $[\Hy:\Hx]=1$.
        \item If $r\geq |a|$, then $[\Hy:\Hx]=n$.
          
        \end{itemize}
        
      \end{Pro}

\printbibliography

\textsc{Tinhinane Amina, Azzouz}
\end{document}